\documentclass[reqno,english]{amsart}

\usepackage{amsmath,amsfonts,amssymb,graphicx,amsthm,enumerate,url}
\usepackage[noadjust]{cite}
\usepackage{stmaryrd}
\usepackage{mathrsfs,booktabs,tabularx}
\usepackage{xifthen,xcolor,tikz,setspace}
\usetikzlibrary{decorations.pathmorphing,patterns,shapes,calc,decorations}
\usetikzlibrary{decorations.pathreplacing}
\usepackage{mathtools}
\usepackage[final]{showkeys} 

\definecolor{refkey}{gray}{.75}
\definecolor{labelkey}{gray}{.5}
\usepackage[algo2e,boxed,vlined,algoruled]{algorithm2e}

\usepackage[letterpaper]{geometry}
\usepackage[colorinlistoftodos]{todonotes}

\usepackage[colorlinks=true]{hyperref}
\colorlet{DarkGreen}{green!50!black}
\colorlet{DarkGray}{gray!60!black}

\numberwithin{equation}{section}

\renewcommand{\restriction}{\mathord{\upharpoonright}}
\renewcommand{\epsilon}{\varepsilon}

\newcommand{\one}{\mathbf{1}}

\usepackage{color}
 \definecolor{refkey}{gray}{.5}
 \definecolor{labelkey}{gray}{.5}
\definecolor{light}{gray}{.9}

\definecolor{pillDarkGreenFG}{RGB}{35, 156, 77}
\definecolor{pillDarkGreenBG}{RGB}{177, 225, 199}
\definecolor{pillAzureFG}{RGB}{33, 160, 183}
\definecolor{pillAzureBG}{RGB}{134, 207, 219}
\definecolor{pillGreenFG}{RGB}{95, 215, 42}
\definecolor{pillGreenBG}{RGB}{190, 239, 161}
\definecolor{pillBlueFG}{RGB}{15, 130, 251}
\definecolor{pillBlueBG}{RGB}{192, 222, 254}
\definecolor{pillOrangeFG}{RGB}{252, 150, 39}
\definecolor{pillOrangeBG}{RGB}{255, 236, 162}
\definecolor{pillBrownFG}{RGB}{151, 82, 17}
\definecolor{pillBrownBG}{RGB}{201, 166, 127}
\definecolor{pillRedFG}{RGB}{235, 75, 78}
\definecolor{pillRedBG}{RGB}{242, 162, 166}

\newcommand{\pillarbox}[5]{\begin{scope}[shift={(#1,#2)}]
\draw[fill=#4, line width=0.25pt] (0,0) rectangle (8pt,8pt) ;
\draw[fill=#5, opacity = .5, line width=0.25pt] (0,8pt)--(3pt,11pt)--(11pt,11pt)--(8pt,8pt);
\draw[fill=#4, opacity =.5, line width=0.25pt] (8pt,0)--(11pt,3pt)--(11pt,11pt)--(8pt,8pt);
\draw (27pt, 6.25pt) node[text=#4]{#3};
\end{scope} }

\newcommand{\wallface}[5]{\begin{scope}[shift={(#1,#2)}]
\draw[fill=#4, line width=0.25pt] (0,0) rectangle (8pt,8pt) ;
\draw[fill=#4, opacity =.5, line width=0.25pt] (8pt,0)--(11pt,3pt)--(11pt,11pt)--(8pt,8pt);
\draw (27pt, 6.25pt) node[text=#4]{#3};
\end{scope} }

\newcommand{\ceilface}[5]{\begin{scope}[shift={(#1,#2)}]
\draw[fill=#5, opacity = .6,  line width=0.25pt] (0,4pt)--(3pt,7pt)--(11pt,7pt)--(8pt,4pt)--(0,4pt);
\draw (27pt, 6.25pt) node[text=#4]{#3};
\end{scope} }

\newtheorem{maintheorem}{Theorem}

\newtheorem{mainprop}[maintheorem]{Proposition}

\newtheorem{theorem}{Theorem}[section]
\newtheorem*{theorem*}{Theorem}
\newtheorem{lemma}[theorem]{Lemma}

\newtheorem{claim}[theorem]{Claim}
\newtheorem{proposition}[theorem]{Proposition}

\newtheorem{observation}[theorem]{Observation}
\newtheorem{fact}[theorem]{Fact}
\newtheorem{corollary}[theorem]{Corollary}

\newtheorem{remark}[theorem]{Remark}

\theoremstyle{definition}{

\newtheorem{definition}[theorem]{Definition}
\newtheorem*{definition*}{Definition}

}

\newcommand{\E}{\mathbb E}

\renewcommand{\P}{\mathbb P}

\newcommand{\R}{\mathbb R}
\newcommand{\Z}{\mathbb Z}
\newcommand{\cA}{\ensuremath{\mathcal A}}
\newcommand{\cB}{\ensuremath{\mathcal B}}
\newcommand{\cC}{\ensuremath{\mathcal C}}

\newcommand{\cE}{\ensuremath{\mathcal E}}
\newcommand{\cF}{\ensuremath{\mathcal F}}

\newcommand{\cH}{\ensuremath{\mathcal H}}
\newcommand{\cI}{\ensuremath{\mathcal I}}
\newcommand{\cJ}{\ensuremath{\mathcal J}}
\newcommand{\cK}{\ensuremath{\mathcal K}}
\newcommand{\cL}{\ensuremath{\mathcal L}}

\newcommand{\cO}{\ensuremath{\mathcal O}}
\newcommand{\cP}{\ensuremath{\mathcal P}}

\newcommand{\cS}{\ensuremath{\mathcal S}}

\newcommand{\cV}{\ensuremath{\mathcal V}}

\newcommand{\cZ}{\ensuremath{\mathcal Z}}
\newcommand{\llb }{\llbracket}
\newcommand{\rrb }{\rrbracket}

\newcommand{\fD}{\mathfrak{D}}
\newcommand{\fF}{\mathfrak{F}}

\newcommand{\fm}{\mathfrak{m}}
\newcommand{\fs}{\mathfrak{s}}
\newcommand{\fS}{\mathfrak{S}}
\newcommand{\fW}{\mathfrak{W}}
\newcommand{\fX}{\mathfrak{X}}

\newcommand{\sA}{{\ensuremath{\mathscr A}}}
\newcommand{\sB}{{\ensuremath{\mathscr B}}}

\newcommand{\sF}{{\ensuremath{\mathscr F}}}

\newcommand{\sT}{{\ensuremath{\mathscr T}}}
\newcommand{\sX}{{\ensuremath{\mathscr X}}}

\newcommand{\g}{{\ensuremath{\mathbf g}}}
\newcommand{\br}{{\ensuremath{\mathbf r}}}
\newcommand{\bB}{{\ensuremath{\mathbf B}}}
\newcommand{\bC}{{\ensuremath{\mathbf C}}}
\newcommand{\bD}{{\ensuremath{\mathbf D}}}

\newcommand{\bF}{{\ensuremath{\mathbf F}}}

\newcommand{\bH}{{\ensuremath{\mathbf H}}}
\newcommand{\bI}{{\ensuremath{\mathbf I}}}
\newcommand{\bJ}{{\ensuremath{\mathbf J}}}
\newcommand{\bW}{{\ensuremath{\mathbf W}}}
\newcommand{\bX}{{\ensuremath{\mathbf X}}}
\newcommand{\bY}{{\ensuremath{\mathbf Y}}}

 \renewcommand{\epsilon}{\varepsilon}

\DeclareMathOperator{\diam}{diam}

\DeclareMathOperator{\hgt}{ht}
\newcommand{\tv}{{\textsc{tv}}}
\newcommand{\spine}{{\textsc{sp}}}
\newcommand{\Tsp}{{\tau_\spine}}
\newcommand{\trivincr}{\varnothing}
\newcommand{\udarrow}{{\ensuremath{\scriptscriptstyle \updownarrow}}}
\newcommand{\lrvec}[1]{\overset{\,{}_\leftrightarrow}{#1}\!}
\newcommand{\red}{\textsc{red}}
\newcommand{\blue}{\textsc{blue}}

\makeatletter
\newcommand{\superimpose}[2]{%
  {\ooalign{$#1\@firstoftwo#2$\cr\hfil$#1\@secondoftwo#2$\hfil\cr}}}
\makeatother

\begin{document}

\title{Tightness and tails of the maximum in 3D Ising interfaces}

\author{Reza Gheissari}
\address{R.\ Gheissari\hfill\break
Courant Institute\\ New York University\\
251 Mercer Street\\ New York, NY 10012, USA.}
\email{reza@cims.nyu.edu}

\author{Eyal Lubetzky}
\address{E.\ Lubetzky\hfill\break
Courant Institute\\ New York University\\
251 Mercer Street\\ New York, NY 10012, USA.}
\email{eyal@courant.nyu.edu}

\maketitle
\begin{abstract}
Consider the 3D Ising model on a box of side length $n$ with minus boundary conditions above the $xy$-plane and  plus boundary conditions below it. At low temperatures, Dobrushin (1972) showed that the interface separating the predominantly plus and predominantly minus regions is localized: its height above a fixed point has exponential tails. Recently, the authors proved a law of large numbers for the maximum height $M_n$ of this interface: for every $\beta$ large, $M_n/ \log n\to c_\beta$ in probability as $n\to\infty$.

Here we show that the laws of the centered maxima $(M_n - \E[M_n])_{n\geq 1}$ are uniformly tight. Moreover, even though this sequence does not converge, we prove that it has uniform upper and lower Gumbel tails (exponential right tails and doubly exponential left tails). Key to the proof is a sharp (up to $O(1)$ precision) understanding of the surface large deviations. This includes, in particular, the shape of a pillar that reaches near-maximum height, even at its base, where the interactions with neighboring pillars are dominant.
\end{abstract}

\section{Introduction}
We study the low-temperature interface of the Ising model in dimensions three and higher. An Ising configuration on a subgraph $\Lambda$ of $\Z^d$ is an assignment of $\{+1,-1\}$ \emph{spins} to the ($d$-dimensional) cells of $\Lambda$, denoted~$\cC(\Lambda)$. The cells are identified with their midpoints which are the vertices of the dual graph $(\Z+\frac 12)^d$ (when $d=2$, the cells are the faces of $\Lambda$ and when $d=3$, they are cubes of side length $1$), and two cells~$u,v$ are considered adjacent ($v\sim w$) if their mid-points are at a Euclidean distance $1$. 
The Ising model on a finite $\Lambda\subset\Z^d$ at inverse-temperature $\beta$ is the probability measure over configurations $\sigma \in \{\pm 1\}^{\cC(\Lambda)}$ with weights 
\begin{align*}
    \mu_{\Lambda,\beta} \propto \exp [- \beta \cH(\sigma)]\,, \qquad \mbox{where}\qquad \cH(\sigma) = \sum_{v\sim w} \one\{\sigma_v \neq \sigma_w\}\,.
\end{align*}
For a set $R\supset \Lambda$ and a configuration $\eta$ on $R$, the Ising model on $\Lambda$ with boundary conditions $\eta$, denoted~$\mu_{\Lambda,\beta}^{\eta}$, is the measure $\mu_{R,\beta}$ conditioned on $\sigma$ coinciding with $\eta$ on $R\setminus \Lambda$. 
These definitions extend to infinite subsets $\Lambda\subset\Z^d$ by taking weak limits of measures on (say) finite boxes under suitable boundary conditions. 

The Ising model exhibits a rich and extensively studied phase transition on $\Z^d$ ($d\geq 2$): there exists a $\beta_c(d)$ such that when $\beta<\beta_c(d)$, there is a unique infinite-volume Gibbs measure on $\Z^d$, whereas when $\beta>\beta_c(d)$, there are multiple distinct infinite-volume measures, e.g., those obtained by taking a limit of finite volume measures with plus boundary conditions $\mu_{\Z^d, \beta}^+$ vs.\ minus boundary conditions $\mu_{\Z^d, \beta}^-$.  This can be seen by a classical argument of Peierls, which demonstrates that in a box of side length $n$ with $+$ boundary conditions, there exists a $\beta_0$ such that for $\beta>\beta_0$ the minus clusters become ``sub-critical,'' i.e., the probability that the origin is part of a connected component (cluster) of at least $r$ minus sites is at most $\exp[-c\beta r]$. 

The analysis of this low-temperature phase has then focused on the structure of the interface between predominantly plus and predominantly minus regions. Namely, consider the model on the infinite cylinder 
\[\Lambda_n : = \{-n,\ldots,n\}^{d-1}\times \{-\infty,\ldots,\infty\}\,,
\]
under \emph{Dobrushin boundary conditions}, which are plus on all cells of $\Z^d$ with negative $d$-th coordinates, and minus on all cells of $\Z^d $ with positive $d$-th coordinates: denote the resulting (infinite volume) measure $\mu_n^\mp$ (the uniqueness of which, for every $\beta>0$, follows by a classical coupling argument via the monotonicity of the Ising model in boundary conditions). At low temperatures $\beta>\beta_0$, these boundary conditions impose an \emph{interface} $\cI$, below which the configuration has the features of the plus phase $\mu_{\Z^d}^{+}$, and above which the measure has the features of the minus phase $\mu_{\Z^d}^{-}$. For a configuration $\sigma\in \{\pm 1\}^{\cC(\Lambda_n)}$, this interface $\cI$ is defined by taking the set of all $(d-1)$-cells (e.g., edges for $d=2$ and faces for $d=3$) separating disagreeing spins, and letting $\cI$ be the (maximal) connected component of such $(d-1)$-cells that separates the minus cluster of the boundary from the plus cluster of the boundary  (see Section~\ref{sec:preliminaries} for a precise definition).

In two dimensions, this interface forms a random decorated curve whose properties as $n\to\infty$ are by now very well understood. For every $\beta>\beta_c(2)$, this interface is \emph{rough}, with typical fluctuations that are $O(\sqrt n)$. In fact, after a diffusive rescaling, the interface is known to converge to a Brownian bridge, and as such its \emph{maximum} height in $\Lambda_n$ can be seen to also have fluctuations of order $\sqrt n$~\cite{DH97,DKS,Hryniv98,Ioffe94,Ioffe95,PV97,PV99}. 

In dimensions $d\geq 3$, the interface forms a random $(d-1)$-dimensional surface whose features are quite different from the two-dimensional case described above. For the ease of exposition, we focus on the important case $d=3$, where the interface forms a random 2D surface; unless otherwise noted, the new results stated for $d=3$ extend to $d>3$ with simple modifications (see Remark~\ref{rem:higher-dimensions}). A landmark result in the study of the 3D Ising interface in $\mu^\mp_n$ was the proof by Dobrushin~\cite{Dobrushin72a} in 1972 that for  large $\beta$ the interface is \emph{rigid}: its typical height fluctuations are $O(1)$ and in fact have an exponential tail, e.g., 
\begin{align*}
\mu_n^\mp (\max \{h: (0,0, h)\in \cI\} \geq r) \leq \exp [-\beta r/3]\,.
\end{align*} 
A consequence of this is the existence of an infinite-volume measure $\mu_{\Z^3}^\mp$ which is not a mixture of $\mu_{\Z^3}^+$ and $\mu_{\Z^3}^{-}$ (cf., dimension two where all infinite-volume measures are mixtures of $\mu_{\Z^2}^{+}$ and $\mu_{\Z^2}^{-}$~\cite{Ai80,Hi81}). 

The present paper studies the maximum height of the 3D Ising interface. 
Unlike several related models of 2D~surfaces whose maximum has been extensively analyzed in recent years---e.g., the (2+1)D Solid-on-Solid (SOS) model~\cite{BEF86,CLMST1,CLMST2}, the Discrete Gaussian and $|\nabla \phi|^p$ models~\cite{LMS16}, and the discrete Gaussian free field (DGFF)~\cite{BDG01,BDZ16,Zeitouni16}, to name a few (we refer the reader to~\cite[Section 1.4]{GL19a} for a more detailed review of this literature)---the 3D Ising interface $\cI$ is not a height function; it can have \emph{overhangs} intersecting a given column $(x_1,x_2)\times \R$ at multiple heights. Further, sub-critical \emph{bubbles} in the plus and minus phases under $\mu_n^\mp$, albeit unseen in $\cI$, do affect its distribution (bubbles and overhangs are precluded from SOS for instance).

Since the Ising interface is not a height function, we define its maximum height as
\begin{align*}
M_n : = \max \left\{x_3\,:\; (x_1,x_2,x_3)\in \cI\mbox{ for some }(x_1,x_2)\in [-n,n]^2\right\}\,.
\end{align*}
The above bound by Dobrushin on the bulk fluctuations of $\cI$ implies via a union bound that $M_n \leq (C/\beta) \log n$ with probability $1-o(1)$. Recently, the authors showed a law of large numbers for this maximum height~\cite{GL19a}: 
 there exists $\beta_0$ such that every $\beta>\beta_0$, 
\begin{align*}
\lim_{n\to\infty}\frac{M_n}{\log n} = \frac{2}{\alpha}\,, \qquad \mbox{in $\mu_n^\mp$-probability}\,,
\end{align*}
where, if  $v\xleftrightarrow[A~]{+}w$ denotes that there is a path of adjacent or diagonally adjacent plus spins between $v$ and $w$ in $A$ (henceforth, we refer to this notion of adjacency as $*$-connectivity; see~\S\ref{subsec:graph-notation} for more details), then
\begin{equation}\label{eq:alpha-alpha-h-def}
\alpha := \lim_{h\to\infty} \frac{\alpha_h}h\qquad\mbox{for}\qquad\alpha_h = -\log \mu_{\Z^3}^\mp\left((\tfrac12,\tfrac12,\tfrac12) \xleftrightarrow[\R^2\times [0,\infty)]{+} (\Z+\tfrac 12)^2\times\{h-\tfrac12\}\right)\,.
\end{equation}
The present work continues the analysis of the extrema of the random surface $\cI$ given by the interface of low-temperature Ising models in three dimensions, and looks at the law of $M_n$ beyond first order asymptotics.
We begin by characterizing the mean of $M_n$ in terms of the infinite-volume large deviation quantity $\alpha_h$, and proving the tightness of the centered maximum.

\begin{maintheorem}\label{mainthm:tightness}
There exist $\beta_0>0$ and a sequence $\epsilon_\beta>0$ going to $0$ as $\beta\to\infty$,  so that, for every~$\beta>\beta_0$, the maximum $M_n$ of the interface of the Ising model on $\Lambda_n$ with Dobrushin boundary conditions satisfies
\begin{align}\label{eq:mn-def} 
 m_n^\star-1-\epsilon_\beta \leq \E[M_n] \leq m_n^\star+\epsilon_\beta\quad\mbox{ where }\quad m_n^\star = \inf\{ h\geq 1 \,:\; \alpha_h > 2\log(2n)-2\beta\}\,,\end{align}
and
\begin{align*}
\mu_{n}^{\mp}\big(M_n - \E[M_n] \in \cdot \big) \mbox{ is a tight sequence}\,.
\end{align*}
\end{maintheorem}

\begin{remark}\label{rem:median}
We also show (cf.~Corollary~\ref{cor:mean-median}) that every \emph{median} $m_n$ of $\mu_n^\mp(M_n\in\cdot)$ has
$m_n^\star-1 \leq m_n \leq m_n^\star$.
\end{remark}
\begin{remark}\label{rem:higher-dimensions}
The results extend to the $d$-dimensional Ising model for any $d\geq 3$, where $m_n^\star$ in~\eqref{eq:mn-def} becomes $\inf\{h \geq 1:\alpha_h > (d-1) \log (2n)-(d-1)\beta\}$ and $\alpha_h$ 
addresses $(\tfrac12,\ldots,\tfrac12) \xleftrightarrow[\R^{d-1}\times[0,\infty)]{+} (\Z+\tfrac 12)^{d-1}\times\{h-\tfrac12\}$.
\end{remark}

\begin{figure}
    \centering
    \begin{tikzpicture}
        \node (fig1) at (-5,5) {
	\includegraphics[width=0.6\textwidth]{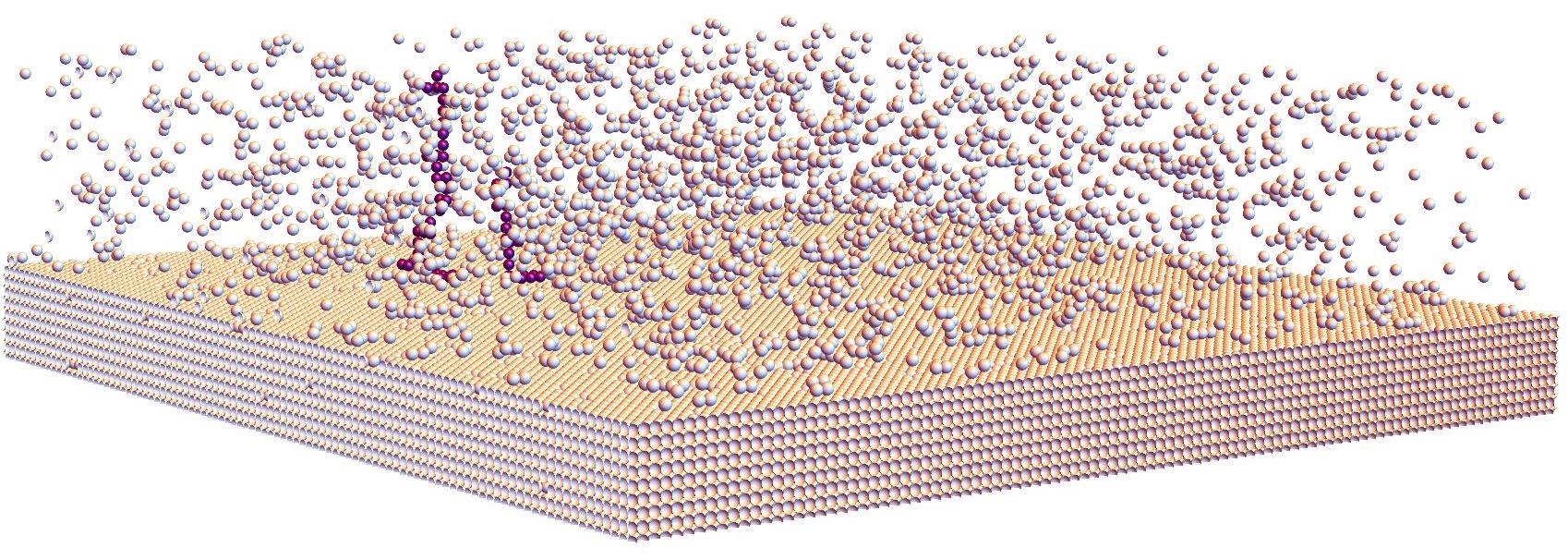}
	};
        \node (fig1) at (-5,2) {
	\includegraphics[width=0.6\textwidth]{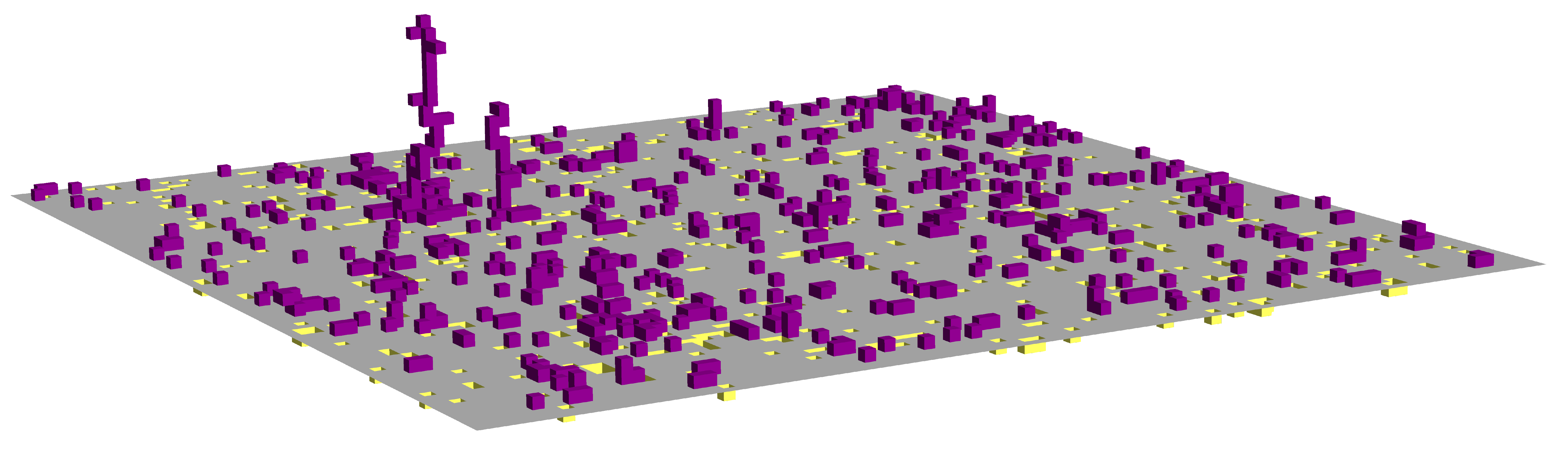}
    };
    \node (fig1) at (2.7,3.2) {
	\includegraphics[width=0.25\textwidth]{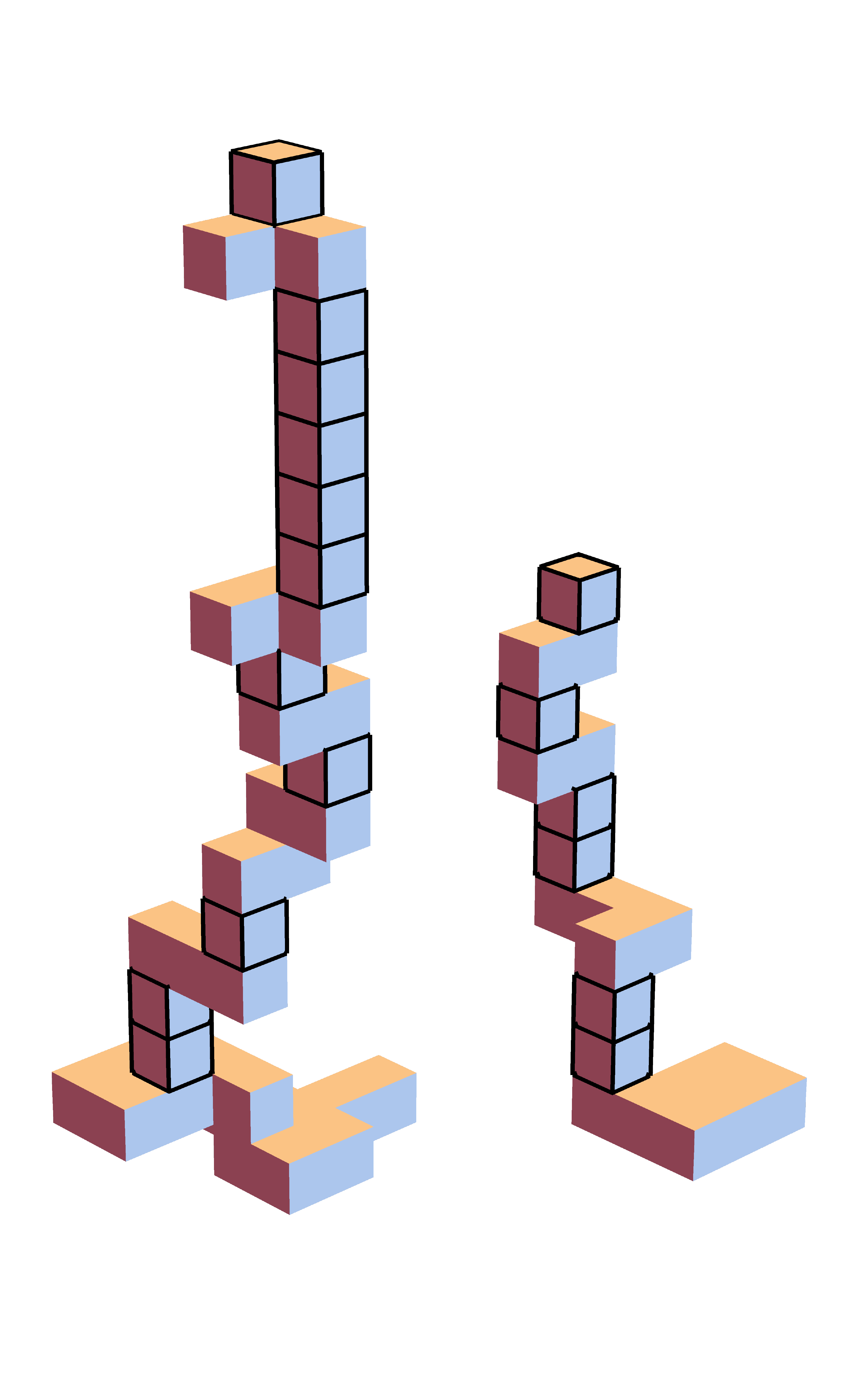}
    };
    \end{tikzpicture}
    \vspace{-0.4in}
    \caption{Top left: the plus sites of an Ising configuration under $\mu_n^\mp$. Bottom left: from that configuration, one obtains the interface $\cI$ by repeatedly flipping the spins in all finite connected components of plus or minus sites, and letting $\cI$ be the set of faces separating the remaining disagreeing spins. Right: two neighboring tall pillars of the resulting interface decompose into a base and a sequence of increments delimited by their cut-points (in bold). }
    \label{fig:interface}
    \vspace{-0.1in}
\end{figure}

Having established the tightness of the maximum height of the 3D Ising interface around its mean, one can then ask about the behavior of the limit/subsequential limits of the centered maximum.  Using a multi-scale argument, we prove the following regarding the limit points of $M_n -\E[M_n]$. 

\begin{maintheorem}\label{mainthm:gumbel-tails}
There exists $\beta_0$ such that for every $\beta>\beta_0$ there are $C,\bar\alpha>0$ so that the maximum $M_n$ of the interface of the Ising model on $\Lambda_n$ with Dobrushin boundary conditions has, for all $r>0$ and large enough~$n$,
\begin{align*}
 e^{-(\bar\alpha r+C)} & \leq \mu_n^{\mp}\big(M_n \geq \E[M_n] + r\big) \leq e^{ - (\alpha_r-C)}\,, \\
 e^{- e^{\bar\alpha r+C}} &  \leq \mu_n^{\mp}\big(M_n \leq \E[M_n] - r \big) \leq e^{ - e^{\alpha_r-C}}\,,
\end{align*}
in which, for every $r>0$, the ratio $\bar\alpha r /\alpha_r $ goes to $1$ as $\beta\to\infty$.
\end{maintheorem}
\begin{mainprop}\label{mainprop:divergence}
For no nonrandom sequence $(m_n)$ does $(M_n - m_n)$ converge weakly to a nondegenerate law.
\end{mainprop}

The above described behavior of uniform Gumbel tails and non-convergence of its centered maximum matches the behavior exhibited by, e.g., the maximum of i.i.d.\ independent geometric random variables. 
However, if the Ising interface were instead \emph{tilted} at an angle (say via boundary conditions that are minus above some plane with outward normal $\hat n\notin \{ \pm e_1,\pm e_2,\pm e_3\}$ and plus below it), a famous open problem is to establish that $\cI$ would then be rough even at very low temperatures, and resemble the DGFF (see~\cite{CerfKenyon01} for such a result at zero temperature). It would be interesting to compare the maximum displacement of said tilted interface to that of the DGFF, where the asymptotic behavior of the centered maximum was shown~\cite{BrZe12} to be tight and subsequently found~\cite{Ding2013,BDZ16} to converge to a randomly-shifted Gumbel distribution.

Unlike the DGFF (and other log-correlated random fields, e.g., BBM)---where the marginal at a site is Gaussian and the difficulty in the analysis of the maximum is due to the logarithmic correlations between sites---in the case of 3D Ising interfaces, obtaining a good understanding of the probability that the interface reaches a height $h$ above a fixed site in $\Z^2 \times \{0\}$ is already a major obstacle.

\subsection{Proof strategy and outline}
As in the prequel~\cite{GL19a}, our analysis of the maximum height of a 3D Ising interface centers on understanding the shape of the interface locally near where it attains atypically large heights. This was formalized via what we refer to as \emph{pillars}: for a face with midpoint $x= (x_1, x_2, 0)$, the pillar of $x$, denoted $\cP_x$, is obtained from a configuration $\sigma$ as follows   (see Figure~\ref{fig:interface}):
\begin{enumerate}
    \item Take all finite ($+$) or ($-$) clusters (i.e., sites not $*$-connected to the boundary), simultaneously flip their spins, then repeat this step until no finite clusters remain; the faces separating differing spins in the result comprise the interface~$\cI$.
    \item Remove from that configuration all sites in the lower half-space $\R^2 \times (-\infty,0]$.
    \item The pillar $\cP_x$ is then the (possibly empty) $*$-connected plus component containing the site with midpoint $x+(0,0, \frac 12)$, also identified with the set of faces of $\cI$ that bound it.  
\end{enumerate}
We study the shape of $\cP_x$ conditionally on the height of the pillar, $\hgt(\cP_x)$, exceeding some $h$. To do so, we define cut-points of $\cP_x$ as sites $v$ in $\cP_x$ such that no other site in $\cP_x$ is at the same height as $v$. Ordering the cut-points of $\cP_x$ in increasing height as $v_1, v_2 ,\ldots,v_{\sT+1}$, we decompose $\cP_x$ into a sequence of increments $(\sX_i)_{i\leq \sT}$, where $\sX_i$ is the set of sites in $\cP_x$ delimited by $v_i$ from below and $v_{i+1}$ from above.
We further decompose $\cP_x$ into its \emph{base} consisting of the sites in $\cP_x$ below $v_1$, and its \emph{spine} consisting of the sites in $\cP_x$ at or above $v_1$ (due to the difficulty in controlling interactions with nearby pillars, the base was defined differently in~\cite{GL19a}, namely, it also included a prefix of the spine to mitigate the effect of these interactions).

With these new definitions, the results in~\cite{GL19a} show that, conditionally on $\hgt(\cP_x) \geq h$, the pillar $\cP_x$ has a base of diameter $O(\log h)$ with an exponential tail beyond that, and all increments above $v_{C\log h}$ have exponential tails on their surface area. It was also shown in~\cite{GL19a} that the increment sequence of $\cP_x$ conditioned on having $\sT\geq T$ increments, behaves asymptotically as a stationary, weakly mixing sequence; in particular, observables of the increment sequence (e.g., its volume, surface area, and displacement) were shown to obey central limit theorems as $T\to\infty$. At the level of these central limit theorems, and the law of large numbers of $M_n$, errors of $O(\log h)$ or $O(\log T)$ on the bounds on the size of the base could be sustained. 

The following result removes these $O(\log h)$ errors, which we cannot afford in proving Theorems~\ref{mainthm:tightness}--\ref{mainthm:gumbel-tails}.

\begin{figure}
\vspace{-0.5cm}
\centering
  \begin{tikzpicture}
    \node (fig1) at (-4.4,0) {
	\includegraphics[width=.40\textwidth]{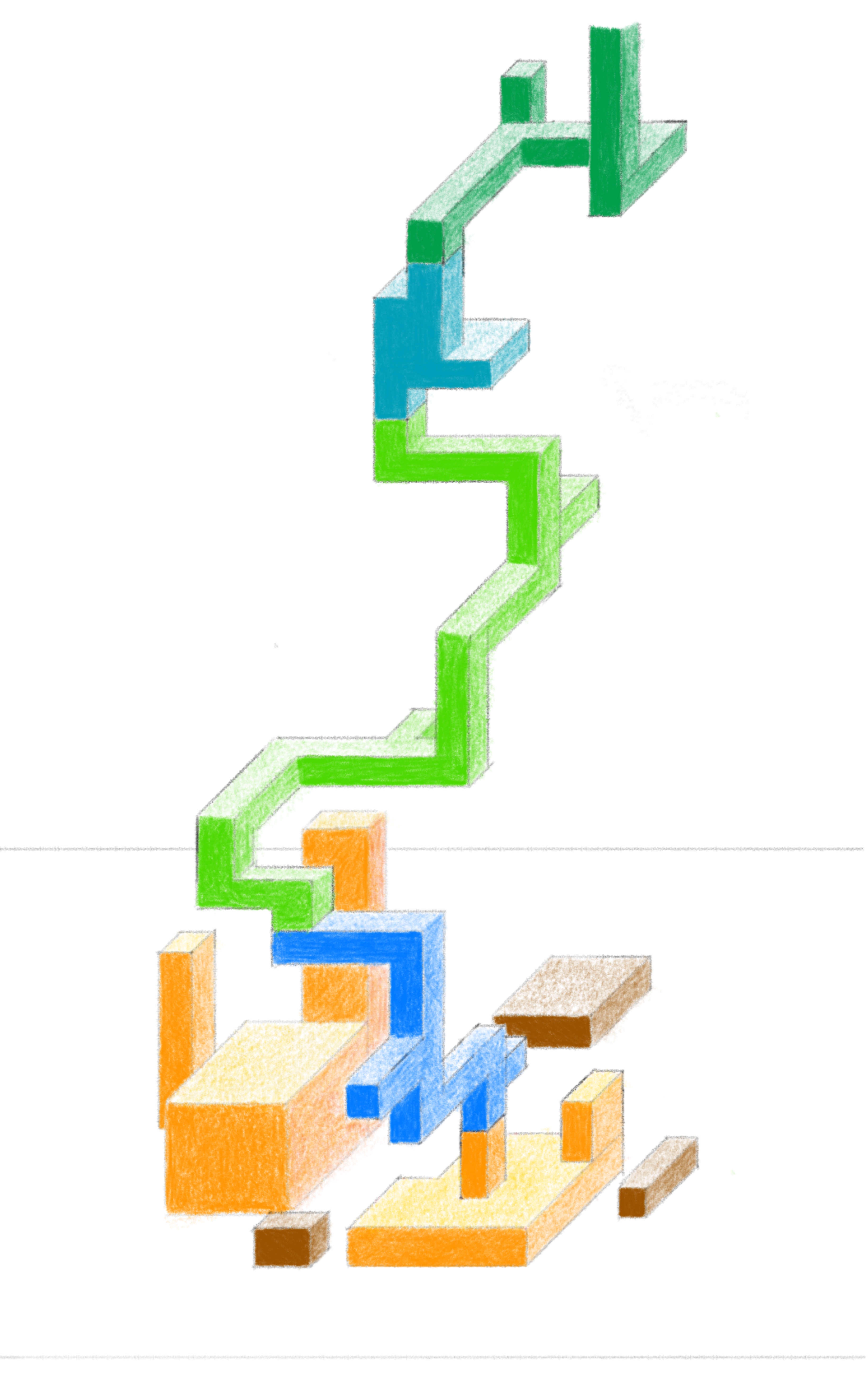}
	};
	
	\draw[densely dotted] (-5.05,-4.4)--(-4.75,-4.4)--(-4.6,-4.25)--(-4.9,-4.25)--(-5.05,-4.4);
  \node at (-4.825,-4.325) {\fontsize{3.5}{4}\selectfont $x$};

    \node (fig2) at (4.4,.18) {
    \includegraphics[width=.407\textwidth]{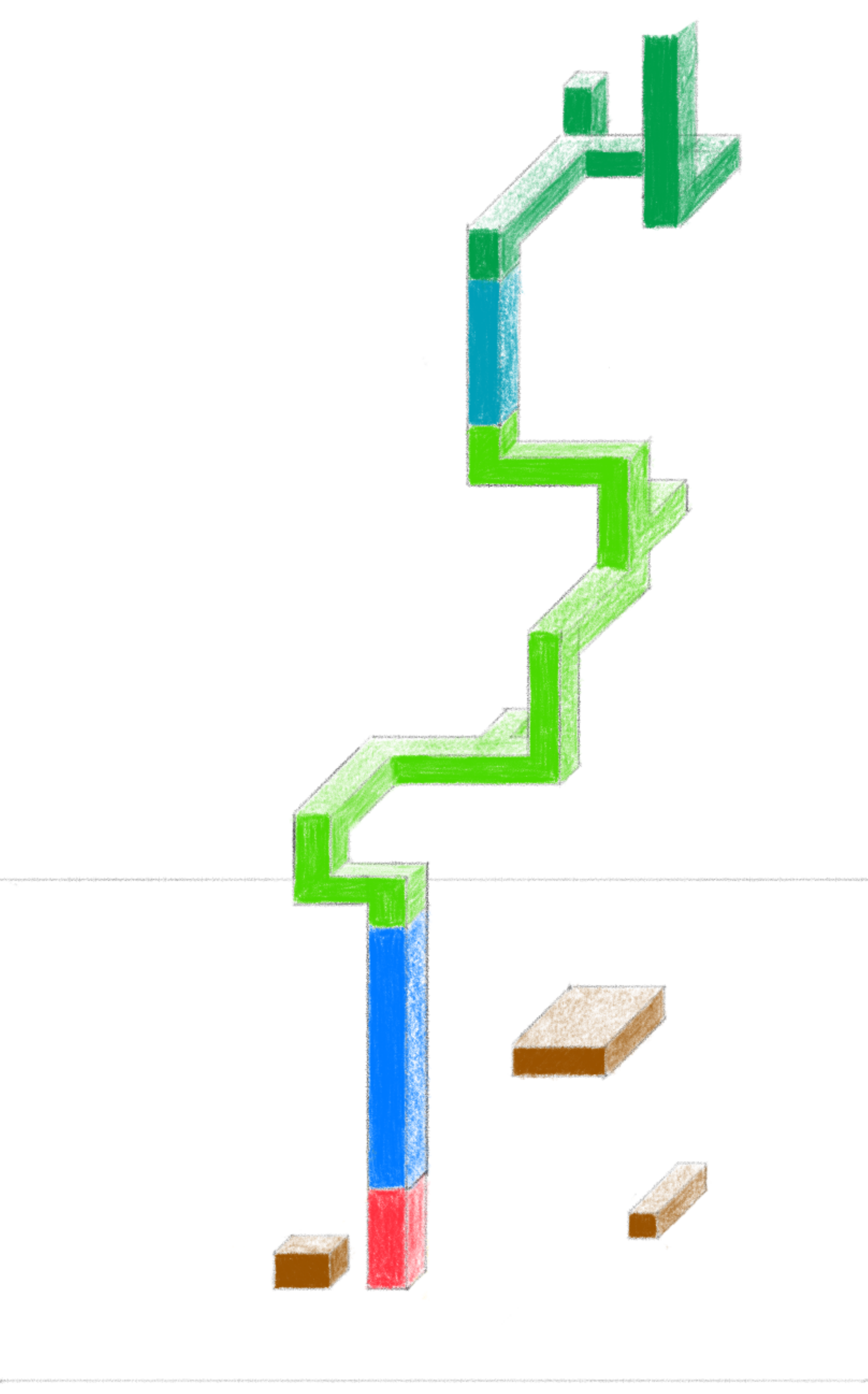}
    };
    
    	\draw[densely dotted] (3.9,-4.4)--(4.2,-4.4)--(4.35,-4.25)--(4.05,-4.25)--(3.9,-4.4);
  \node at (4.125,-4.325) {\fontsize{3.5}{4}\selectfont $x$};

        \draw[->, thick] (-2,2) -- (2,2);
    \node at (0,2.4) {$\Psi_{x,t}$};

  \end{tikzpicture}
\vspace{-.5cm}
 \caption{A map $\Psi_{x,t}$ locally modifies the structure of the interface near the base of a pillar $\cP_x$ that reaches a large height $h\gg 1$. The local modifications introduce delicate interactions with nearby fluctuations of the interface (nearby pillars).}
\vspace{-0.3cm}
\label{fig:big-map}
\end{figure}
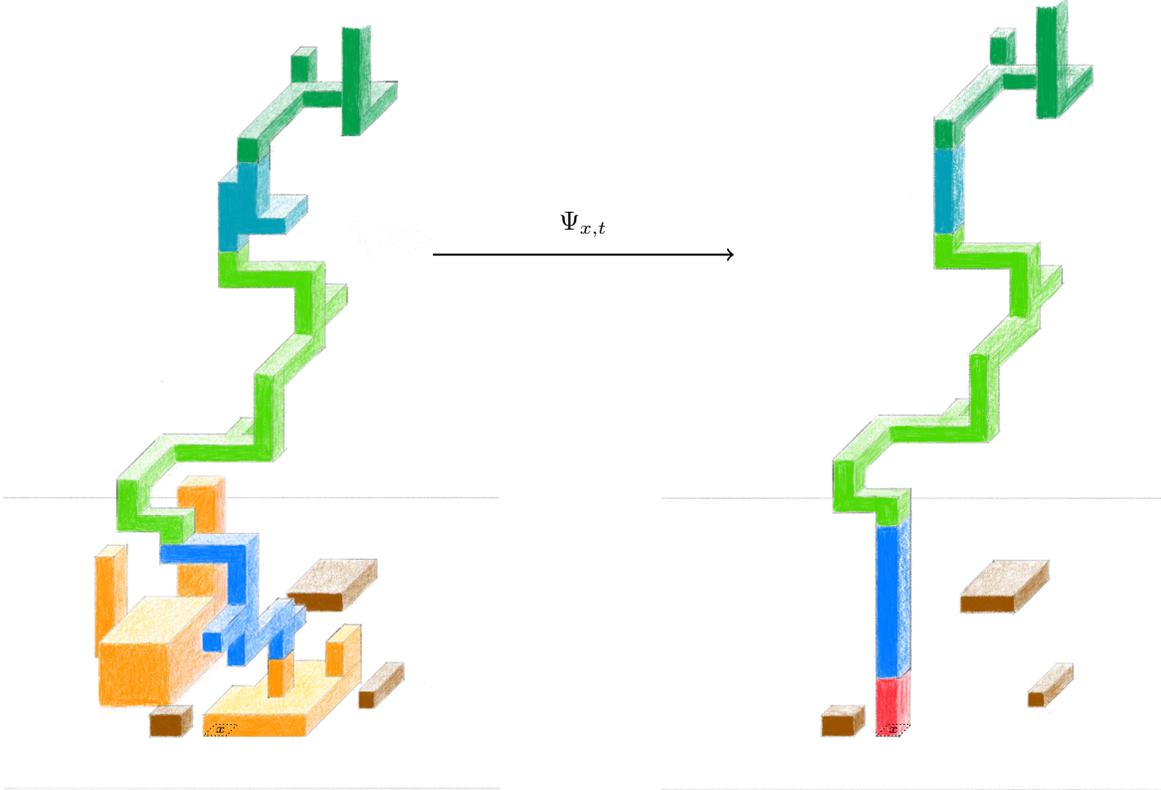

\begin{maintheorem}\label{mainthm:pillar-structure}
There exist $\beta_0, C>0$ such that for every $\beta>\beta_0$ and sequences $(h_n),(\Delta_n),(x_n)$ with  $h_n \ll \Delta_n$ and $x_n \in \llb -n+\Delta_n, n-\Delta_n\rrb^2 \times \{0\}$, the pillar $\cP_x$ of the interface of the Ising model on $\Lambda_n$ with Dobrushin boundary conditions,  conditional on
reaching height~$h_n$, has the following structure:
\begin{enumerate}
   \item The diameter of the base $\sB_x: = \{w\in \cP_x: \hgt(w)< \hgt(v_1)\}$ is at most $ r$ except with probability $C \exp[-(\beta - C) r]$ for every $1\leq r\leq \Delta_n$.
    \item For every $t$, the surface area of the increment $\sX_t:= \{w\in \cP_x: \hgt(v_t)\leq \hgt(w)\leq \hgt(v_{t+1})\}$ is at most $8+r$ except with probability $C \exp[-(\beta - C) r]$ for every $1\leq r \leq \Delta_n$. 
\end{enumerate}
\end{maintheorem}

\subsubsection{Proof of  Theorem~\ref{mainthm:pillar-structure}}\label{subsec:proof-outline-map} As in classical Peierls arguments, as well as~\cite{Dobrushin72a,GL19a}, we design a map $\Psi$ on a subset $\bJ$ of interfaces whose pillars have $\hgt(\cP_x)\geq h$; the map will show that the subset $\bJ$ is rare if
\begin{enumerate}
    \item $\mu_n^\mp(\cI)/\mu_n^\mp(\Psi(\cI))$ is exponentially small in $\beta(|\cI| - |\Psi(\cI)|)$, while 
    \item the map has bounded multiplicity, i.e., it maps at most $C^M$ elements of $\bJ$ with $|\cI|-|\Psi(\cI)| = M$ to each element of $\Psi(\bJ)$ (noting  $|\cI|-|\Psi(\cI)|$ is the
    excess energy of the interface $\cI$ compared to~$\Psi(\cI)$).
\end{enumerate}
In the context of interfaces of the Ising model, we emphasize that bounding $\mu_n^\mp(\cI)/\mu_n^\mp(\Psi(\cI))$ is complicated by the fact that the Ising measure is not a measure on interfaces, but rather on configurations: using cluster expansion, Dobrushin~\cite{Dobrushin72a} viewed the law $\mu_n^\mp(\cI \in \cdot)$ as a perturbation of a measure on interfaces whose weights are proportional to $e^{ - \beta |\cI|}$, where the perturbation $e^{ - \sum_{f\in \cI} \g(f,\cI)}$ captures the sub-critical bubbles in the plus and minus phases. Thus the difficulty in item (1) above reduces to showing that the cumulative effect of $\sum \g(f,\cI) - \sum \g(f',\Psi(\cI))$ is comparable to the energy gain $|\cI|-|\Psi(\cI)|$. 

In Dobrushin's proof of rigidity, this interaction term was controlled via the decomposition of the interface into \emph{groups of walls} describing the vertical fluctuations of the interface: this effectively reduced considerations of the interaction terms to horizontal interactions between distinct walls. 

In~\cite{GL19a}, as in this paper, we required estimates conditionally on $\hgt(\cP_x)\geq h$ and therefore had to examine the structure of pillars in a more refined manner than by reducing to their two-dimensional projection. By decomposing pillars into a base and a sequence of increments, and accepting $O(\log h)$ errors in the base, we were able to handle the interaction terms by separating out the vertical interactions between shifts of increments from the horizontal interactions induced by deletions of groups of walls. 

In both approaches, there is an inherent competition between (1) a desire to delete more walls/straighten additional increments (which simplifies the comparison between the interaction terms and $|\cI|-|\Psi(\cI)|$) and (2)~the inability to delete ``too much," at risk of losing\ control of the multiplicity of the map. 

Towards Theorem~\ref{mainthm:pillar-structure} we need both an exponential tail on the size of the base beyond $O(1)$, and control on increment sizes even at heights that are $O(1)$ (cf., the $O(\log h)$ errors in~\cite{GL19a}). 
As we will see in  Section~\ref{subsec:proof-outline-tightness}, this is fundamental to establishing Theorems~\ref{mainthm:tightness}--\ref{mainthm:gumbel-tails}. 
To prove Theorem~\ref{mainthm:pillar-structure}, we devise a delicate algorithmic map (Algorithm~\ref{alg:le-big-map}) that iteratively handles the interactions between the horizontal shifts of the increments in the pillar $\cP_x$ and the vertical shifts of nearby walls of distinct pillars. See Figure~\ref{fig:big-map} for a visualization of this map, and Section~\ref{subsec:reader-guide-map} for a more detailed discussion of the various steps in its construction.

\subsubsection{Proof of tightness and Gumbel tails for $M_n - \E[M_n]$}\label{subsec:proof-outline-tightness} 
In~\cite{GL19a} we used structural results on the shape of  pillars $\cP_x$ attaining a large height to prove approximate sub-multiplicativity for the sequence $\mu^\mp (\hgt(\cP_x)\geq h)$, with a multiplicative error that is $O(e^{\beta \diam(\sB_x)^2})$ where $\sB_x$ is the base of the pillar $\cP_x$. As the bounds in that paper showed that typically $\diam(\sB_x)=O(\log h)$, it follows from Fekete's Lemma that $\frac 1h \log \mu^\mp(\hgt(\cP_x)\geq h)$ has a limit $\alpha$ as $h\to\infty$. Further, deducing in that paper that $M_n/\log n \to 2/\alpha$, via a second moment argument, relied on the bound on $\diam(\sB_x)$, as if $y+(0,0,\frac 12)$ is interior to $\sB_x$, then $\cP_x=\cP_y$. 

Our proof of tightness of $M_n  - \E[M_n]$ therefore 
necessitates establishing that $\diam(\sB_x)=O(1)$, as well as an $O(1)$ error in the relevant sub-multiplicativity estimates. Such refined estimates, as well as others needed in the second moment argument that is used to establish tightness (such as controlling the increments and tail bounds on the event $\{\hgt(\cP_x)\geq h\}$ at heights that are $O(1)$), are derived from the new Theorem~\ref{mainthm:pillar-structure}. 
The Gumbel tails in Theorem~\ref{mainthm:gumbel-tails} are then obtained via a coupling of the maximum $M_n$ to the maximum of $(1+o(1))(n/L_n)^2$ i.i.d.\ copies of the maximum at a suitably chosen smaller scale $M_{L_n}$ (see Proposition~\ref{prop:multiscale-linear}), thereby boosting the exponential left tail into a doubly exponential one.

\subsection{Organization}
Section~\ref{sec:preliminaries} contains the prerequisite definitions of walls/ceilings in Ising interfaces as per Dobroshin's framework. Section~\ref{sec:pillars} defines pillars and their decomposition into a base and sequence of increments (refining those of the prequel). Section~\ref{sec:exp-tail-base}, which is the heart of the proof, defines the map $\Psi$ and proves Theorem~\ref{thm:exp-tail-base} and Proposition~\ref{prop:exp-tail-all-increments}, which are more detailed versions of Theorem~\ref{mainthm:pillar-structure} from above. Section~\ref{sec:submult} proves the refined sub-multiplicativity estimates (Proposition~\ref{prop:submult} and Corollary~\ref{cor:alpha-super-additive}). These are used in Section~\ref{sec:tightness}, via a second moment argument, to prove exponential tails---and thus tightness---for the maximum (Proposition~\ref{prop:exp-tails}). Section~\ref{sec:gumbel} builds the multi-scale coupling of the maximum (Proposition~\ref{prop:multiscale-linear}), used to boost the exponential tails into Gumbel tails and prove Theorem~\ref{mainthm:gumbel-tails} as well as Proposition~\ref{mainprop:divergence}.

\section{Preliminaries}\label{sec:preliminaries}
In this section, we formalize the setup of the low-temperature Ising model, define its interface under Dobrushin boundary conditions more precisely, and recall the decomposition of this interface into walls and ceilings introduced in~\cite{Dobrushin72a} to prove rigidity of the interface. 

\subsection{Graph notation}\label{subsec:graph-notation}
We begin by describing the graph notation we use throughout the paper; though the results in this paper generalize directly to dimensions greater than three, for ease of exposition we present everything in the setup of the three-dimensional integer lattice. 

Let $\Z^3$ be the three-dimensional integer lattice graph with vertex set $ \{(v_1, v_2, v_3)\in \Z^3\}$ and edge-set identified with the set of nearest-neighbor pairs of vertices $\{\{v,w\}:d(v,w)=1\}$, where $d(x,y)=|x-y|$ will always denote the Euclidean distance between two points $x,y$. 

A \emph{face} of $\Z^3$ is the open set of points in $\R^3$ bounded by four edges (and four vertices) forming a square of side length one (normal to one of the coordinate axes). A face is \emph{horizontal} if its outward normal vector is $\pm e_3$ and it is \emph{vertical} if its outward normal vector is $\pm e_1$ or $\pm e_2$. A \emph{cell} of $\Z^3$ is the open set of points bounded by six faces (and eight vertices) forming a cube of side length one. We will frequently identify edges, faces, and cells with their midpoints, so that points
with two integer and one half-integer coordinate are midpoints of edges, points with one integer and two
half-integer coordinates are midpoints of faces, and points with three half-integer coordinates are midpoints
of cells.  

For a set of vertices $\Lambda \subset \Z^3$, we denote by $\cE(\Lambda), \cF(\Lambda),\cC(\Lambda)$ the edges, faces, and cells, respectively, whose bounding vertices are all contained in $\Lambda$. 

Two distinct edges are adjacent if they share a vertex; two distinct faces are adjacent if they share a bounding edge; two distinct cells are adjacent if they share a bounding face. We use the notation `$\sim$' to denote adjacency. A set of faces (resp., edges, cells) $F$  is called connected if for every $f, f'\in F$, there is a sequence $f_1,\ldots, f_m \in F$ such that $f=f_0 \sim f_1 \sim \ldots \sim f_m = f'$. We say that two faces (resp., edges, cells) are \emph{connected in $\Lambda$} if $F\cap \cF(\Lambda)$ is connected (resp., $F\cap \cE(\Lambda)$ and $F\cap \cC(\Lambda)$ are connected). 

Two distinct edges/faces/cells are $*$-adjacent if they share a bounding vertex. We define $*$-connectivity, analogously to the above definitions for connectivity, w.r.t.\ the weaker notion of $*$-adjacency.

\medskip \noindent \emph{Subsets of $\Z^3$.} The subsets of $\Z^3$ we will primarily consider are boxes or cylinders centered at the origin. Let us denote the centered $(2n+1)\times (2m+1)\times (2h+1)$ box by 
\begin{align*}
    \Lambda_{n,m,h}= \llb -n,n\rrb \times \llb -m,m\rrb \times \llb -h,h\rrb\,,
\end{align*}
where if $a<b$ are integers, $\llb a,b\rrb: = \{a, a+1, \ldots,b-1, b\}$. We use $\Lambda_n$ to denote the infinite cylinder $\Lambda_{n,n,\infty}$. 

For any cell-set $C\subset \cC(\Z^3)$, its (outer) boundary $\partial C$ is the set of all cells in $\cC(\Z^3)\setminus C$ which are adjacent to some cell in $C$. We use the shorthand $\partial \Lambda_{n,m,h} = \partial (\cC(\Lambda_{n,m,h}))$. 

Other important subsets we consider are slabs of $\Z^3$. For an integer $h\in \Z$, let $\cL_h$ be the subgraph of $\Z^3$ with vertex set $\Z^2 \times \{h\}$ and the resulting face-set.  For half-integer $h\in \Z+\frac 12$, let $\cL_h$ consist of the faces and cells of $\Z^3$ whose midpoints have height $h$. Let $\cL_{>0} = \bigcup_{h>0} \cL_h$ be the cell and face-set of the upper half-space, and let $\cL_{<0}$ be the cell and face-set of the lower half-space. 

Abusing notation slightly, it will be helpful to use the notation 
\[\cL_{0,n} = \cF(\cL_0 \cap \Lambda_n)\,.
\]

\subsection{The Ising model}
Since our primary object of study is the interface of the 3D Ising model, it will be convenient to consider the Ising model as an assignment of $\{\pm 1\}$ \emph{spins} to the vertices of the dual graph $(\Z^3)^*$, identified with the cells of $\Z^3$. With this choice, the interface will be a connected subset of $\cF(\Z^3)$. 

An Ising configuration on a subset $\cC(\Lambda)\subset \cC(\Z^3)$ is an element $\sigma \in \{\pm 1\}^{\cC(\Lambda)}$. A boundary condition on~$\cC(\Lambda)$ is a configuration $\eta \in \{\pm 1\}^{\cC(\Z^3)}$. The Ising model at inverse-temperature $\beta>0$ on $\cC(\Lambda)$ with boundary conditions $\eta$ is the probability measure on $\sigma \in \{\pm 1\}^{\cC(\Lambda)}$ given by 
\begin{align*}
    \mu_{\Lambda,\beta}^\eta(\sigma) = \frac{1}{\cZ_{\Lambda,\beta}} \exp[- \beta \cH(\sigma)] \qquad \mbox{where}\qquad \cH(\sigma)  = - \sum_{\substack {v\sim w \\ v,w\in \cC(\Lambda)}} \one\{\sigma_v \neq \sigma_w\} - \sum_{\substack{v\sim w\\ v\in \cC(\Lambda), w\in \partial \Lambda}} \one\{\sigma_v\neq \eta_w\}\,,
\end{align*}
where the normalizing constant $\cZ_{\Lambda,\beta}$, called the \emph{partition function}, is such that $\mu_{\Lambda, \beta}^{\eta}$ is a probability measure.  

We suppress the dependence on $\beta$ as the choice of $\beta$ is typically fixed in the context. When $\Lambda = \Lambda_{n,m,h}$ we use the shorthand $\mu_{n,m,h}^\eta = \mu_{\Lambda_{n,m,h}}^\eta$ and when $\Lambda = \Lambda_{n} = \Lambda_{n,n,\infty}$, we use the shorthand $\mu_n^\eta = \mu_{\Lambda_{n,n,\infty}}^\eta$. 

In this paper, we are interested in \emph{Dobrushin boundary conditions}, which are the assignment 
\begin{align*}
    \eta_v = \begin{cases} -1 \qquad v \in \cC(\Z^2 \times \llb 0,\infty\rrb) \\ +1 \qquad v\in \cC(\Z^2 \times \llb -\infty, 0\rrb) \end{cases}\,, 
\end{align*}
and we use the shorthand $\mp$ for this choice of $\eta$. 

\subsubsection*{Domain Markov Property}
Observe that the only dependence of the measure $\mu_{\Lambda, \beta}^\eta$ on the boundary conditions $\eta$ is through the restriction of $\eta$ to $\partial \Lambda$. This leads to what is known as the domain Markov property: for any two finite subsets $A \subset B \subset \cC(\Z^2)$, and every configuration $\eta$ on $B \setminus A$, 
\begin{align*}
    \mu_B(\sigma_A \in \cdot \mid \sigma_{B\setminus A} = \eta) = \mu_A^{\eta_{B\setminus A}} (\sigma_A \in \cdot)\,,
\end{align*}
where $\sigma_A$ denotes the restriction of $\sigma$ to the set $A$. 

\subsubsection*{FKG Inequality} The Ising model satisfies an important positive correlation inequality known as the FKG inequality. Consider the natural partial order on configurations $\sigma\in \{\pm 1\}^A$ and suppose $f$ and $g$ are non-decreasing functions in that partial order. Then
\begin{align*}
    \E_{\mu_{\Lambda}}[f(\sigma)g(\sigma)]\geq \E_{\mu_\Lambda} [f(\sigma)] \E_{\mu_\Lambda} [g(\sigma)]\,,
\end{align*}
where $\E_{\nu}$ is the expectation with respect to the law $\nu$. A special case of this is when $f$ and $g$ are indicator functions of non-decreasing events.

A recurring example of such an increasing event is $*$-connectivity via plus cells. For a set $A \subset \cC(\Lambda)$, we say $v,w\in A$ are in the same $*$-connected plus component of $A$ if $v$ and $w$ are $*$-connected in $\{u\in A: \sigma_u = +1\}$. We use the shorthand 
$v\xleftrightarrow[A]{\;+\;}w$ to denote this event. When $A=\cC(\Lambda)$ we omit it from the notation.

\subsubsection*{Infinite-volume Gibbs measures and DLR condition}
If the underlying geometry $\Lambda$ is an infinite (rather than finite) subset of $\Z^3$, the normalizing constant $\cZ_{\Lambda,\beta}$ is not finite and the measure $\mu_{\Lambda, \beta}^{\eta}$ is a priori undefined. Such \emph{infinite-volume Gibbs measures} are instead defined via a consistency relation known as the \emph{DLR conditions}. For an infinite set $\cC(\Lambda)$, a measure $\nu_{\Lambda}$, defined by its finite-dimensional distributions, satisfies the DLR conditions if for every finite $A \subset \cC(\Lambda)$
\begin{align*}
    \E_{\nu_\Lambda (\sigma_{\cC(\Lambda)\setminus A}\in \cdot)} \big[\nu_{\Lambda}(\sigma_{A}\in \cdot \mid \sigma_{\cC(\Lambda) \setminus A})\big] = \nu_{\Lambda}(\sigma_A \in \cdot)\,.
\end{align*}
Infinite-volume Gibbs measures need not be unique. For the Ising model on $\Z^d$, the phase transition of the model is described in terms of the uniqueness/non-uniqueness of the infinite-volume Gibbs measure. In the low-temperature regimes we are interested in, distinct infinite-volume measures are attained by taking weak limits of Ising models on finite boxes with different boundary conditions (e.g., all-plus, all-minus). 

We denote the infinite-volume Gibbs measures obtained by taking limits of $\mu_{n,n,n}^+$ and $\mu_{n,n,n}^-$ as $n\to\infty$ by $\mu_{\Z^3}^{+}$ and $\mu_{\Z^3}^{-}$, respectively. When $\beta>\beta_c$, $\mu_{\Z^3}^{+} \neq\mu_{\Z^3}^{-}$. Dobrushin~\cite{Dobrushin72a,Dobrushin73} proved that there exists a $\beta_0$ such that when $\beta>\beta_0$ there exist DLR measures on $\Z^3$ that are not mixtures of $\mu_{\Z^3}^+$ and $\mu_{\Z^3}^-$, namely those obtained by taking the limit of $\mu_{n,n,n}^{\mp}$ as $n\to\infty$.

\subsection{The Ising interface with Dobrushin boundary conditions}
The infinite volume measure $\mu_{\Z^3}^{\mp}$ is characterized by an \emph{interface} separating the minus and plus phases, which look like $\mu_{\Z^3}^{-}$ and $\mu_{\Z^3}^{+}$ respectively. Dobrushin showed that this interface is localized---on finite boxes (its height fluctuations above the origin are $O(1)$)---and we are interested in characterizing the law of its maximum height. To that end, we formally define the interface $\cI$ separating the plus and minus phases.  

\begin{definition}[Interface]
Consider the Ising model with Dobrushin boundary conditions on $\Lambda_{n,m,h}$, i.e., $\mu_{n,m,h}^\mp$. For a configuration $\sigma$ on $\cC(\Lambda_{n,m,h})$, define its \emph{interface} $\cI = \cI(\sigma)$ as follows. 
\begin{enumerate}
    \item Extend the configuration $\sigma$ to a configuration on all of $\cC(\Z^3)$ by taking $\sigma_v = -1$ if $v\in \cL_{>0}\setminus \cC(\Lambda_{n,m,h})$ and $\sigma_v = +1$ if   $v\in \cL_{<0}\setminus \cC(\Lambda_{n,m,h})$. 
    \item Let $F(\sigma)$ be the set of faces separating cells of different spins under $\sigma$. 
    \item Call the (maximal) $*$-connected component of $\cL_0 \setminus \cF(\Lambda)$ in $F(\sigma)$, the \emph{extended interface}. (This is also the unique infinite $*$-connected component in $F(\sigma)$.)
\item Let $\cI = \cI(\sigma)$ be the restriction of the extended interface to $\cF(\Lambda)$. 
\end{enumerate}
\end{definition}

\begin{remark}
One could use alternative definitions for singling out the interface $\cI$ out of the connected sets of faces that separate the minus and plus phases of the boundary, e.g., the minimal one, or one obtained by a splitting rule. Locally, the difference set between two such definitions would have an exponential tail via a Peierls argument. 
However, these other choices are not as well-tuned to the arguments that follow.   
\end{remark}

Just as a configuration $\sigma$ identifies an interface $\cI$, every interface $\cI$ identifies a configuration $\sigma(\cI)$ for which $F(\sigma(\cI)) = \cI$; i.e., this configuration is minus everywhere ``above" $\cI$ and plus ``below". One can obtain this configuration by starting from the boundary sites $\partial \mathcal C(\Lambda_{n,m,h})$, and iteratively, from the boundary inwards, assigning spins to the cells of $\Lambda_{n,m,h}$ such that adjacent cells have differing spins if and only if there is a face in $\cI$ separating them. 

If we call $\mathcal C^-(\cI)$ the set of all minus spins of $\sigma(\cI)$, and $\mathcal C^+(\cI)$ the set of all plus spins of $\sigma(\cI)$, we find that each of $\mathcal C^-$ and $\mathcal C^+$ is a \emph{single} infinite $*$-connected set of cells (though we emphasize that each may break up into distinct \emph{nearest-neighbor} connected components $\mathcal C^-_1,...,\mathcal C^-_r$ and $\mathcal C^+_1,...,\mathcal C^+_s$). Recall the following consequence of the definition of the interface above and the domain Markov property of the Ising model. 

\begin{observation}\label{obs:interface}
Conditionally on having an interface $\cI$, the Ising model on $\cC(\Lambda_{n,m,h})$ with $\pm$-boundary conditions is the Ising measure on $\cC(\Lambda_{n,m,h})\setminus \partial \cI$ (where $\partial \cI$ is the set of all cells that share a bounding vertex with a face of $\cI$) with its induced boundary conditions being $\pm$ on $\partial \Lambda_{n,m,h}$ and being the restriction $\sigma(\cI)\restriction_{\partial \cI}$ on $\partial \cI$. This Ising measure is evidently a product of Ising measures on the nearest-neighbor connected components $(\mathcal C^-_i)_i$ with minus boundary conditions and $(\mathcal C^+_i)_i$ with plus boundary conditions. 
\end{observation}

It follows straightforwardly by Borel--Cantelli that for every $\beta>0$, we can take a limit of the measure $\mu_{n,m,h}^\mp$ as $h\to\infty$ and obtain an infinite-volume measure on the cylinder $\Lambda_{n,m,\infty}$ whose interface is almost surely finite. 
With this in hand, we can move to the Ising interface on $\Lambda_n = \Lambda_{n,n,\infty}$ under $\mu_n^{\mp}$. As in the preceding works~\cite{Dobrushin72a,GL19a}, we move from the Ising measure to a measure over interfaces, where the energetic cost of an interface is seen to be given by its cardinality, and the lowest energy interface is that coinciding with $\cL_0$. We therefore define the notion of \emph{excess energy} of one interface with respect to another by 
\begin{align*}
    \fm(\cI;\cJ) = |\cI| - |\cJ| = \cH(\sigma(\cJ)) - \cH(\sigma(\cI))\,.
\end{align*}

Informally, by Observation~\ref{obs:interface}, given an interface $\cI$ the measure $\mu_n^\mp$ looks like a combination of the measure $\mu^-_{n}$ above $\cI$ and $\mu^+_n$ below it.  However, the choice of a particular interface modifies these measures above and below the interface as it precludes, say, plus sites that appear under $\mu^-_n$ but would be $*$-adjacent to $\cI$. At low temperatures, these plus droplets have exponential tails on their size and we can sum over their cumulative effect in order to characterize the Ising measure as a Gibbs measure over interfaces with an additional perturbative term.

\begin{theorem}[{\cite[Lemma 1]{Dobrushin72a}}]
\label{thm:cluster-expansion} Consider the Ising measure $\mu_{n}= \mu^{\mp}_{n}$ on the cylinder $\Lambda_n= \Lambda_{n,n,\infty}$. There exist $\beta_0>0$ and a function $\g$ such that for every $\beta>\beta_{0}$ and any
two interfaces $\cI$ and $\cI'$, 
\begin{align*}
\frac{\mu_n^\mp(\cI)}{\mu_n^\mp(\cI')}= & \exp\bigg[-\beta \fm(\cI, \cI') +\Big(\sum_{f\in\cI}\g(f,\cI)-\sum_{f'\in\cI'}\g(f',\cI')\Big)\bigg]
\end{align*}
and the function $\g$ satisfies the following for some $\bar{c},\bar{K}$
independent of $\beta$: for all $\cI, \cI'$ and $f\in \cI$ and $f'\in\cI'$,
\begin{align}
|\g(f,\cI)| & \leq\bar{K} \label{eq:g-uniform-bound} \\
|\g(f,\cI)-\g(f',\cI')| & \leq \bar K e^{-\bar{c}\br(f,\cI;f',\cI')} \label{eq:g-exponential-decay}
\end{align}
where $\br(f,\cI;f',\cI')$ is the largest radius around
the origin on which $\cI-f$ ($\cI$ shifted by the midpoint of the face $f$) is \emph{congruent} to $\cI'-f'$: that is to say 
\begin{align*}
\br(f,\cI; f', \cI') = \sup \big\{r: (\cI - f) \cap B_{r}(0) \equiv (\cI' - f') \cap B_r(0)\big\}\,,
\end{align*}
where $B_r(0)$ is the ball of radius $r$ around $(0,0,0)$ and the congruence relation $\equiv$ is equality as subsets of $\R^3$, up to, possibly, reflections and $\pm \frac \pi2$ rotations in the horizontal plane.  
\end{theorem}

We will use the phrase \emph{$\br(f,\cI; f',\cI')$ is attained by $g\in \cI$ (resp., $g'\in \cI'$)} if $g$ (resp., $g'$) is a face of minimal distance to $f$ (resp., to $f'$) whose presence prevents $\br(f,\cI; f',\cI')$ from being any larger. 

\subsection{Walls, ceilings, and groups of walls}
Dobrushin's proof of rigidity of the 3D Ising interface used a combinatorial decomposition of the interface to effectively reduce it to a two-dimensional polymer model on $\cL_0$ given by projections of \emph{walls} of $\cI$. We recap the definitions introduced therein in this section and describe the bijection between admissible collections of standard walls and Dobrushin interfaces. 

\begin{definition}
For a set of faces or cells $A$, define its projection $\rho(A) : = \{(x_1, x_2, 0): x=(x_1,x_2,x_3)\in A\}$ so that $\rho(A)\subset \cL_0$. 
\end{definition}

Notice that the projection of a horizontal face is in $\cF(\cL_0)$ while the projection of a vertical face is in $\cE(\cL_0)$. For an interface $\cI$ and an edge or face $u\in \cE(\cL_0)\cup \cF(\cL_0)$, denote by 
\begin{align*}
    N_\rho(u) : = \#\{f\in \cI: \rho(f) = u\}\,.
\end{align*}

\begin{definition}[Walls and ceilings]\label{def:walls-and-ceilings}
A face $f\in \cI$ is a \emph{ceiling face} if $f$ is a horizontal face and $N_\rho(\rho(f))= 1$. A face $f\in \cI$ is a \emph{wall face} if it is not a ceiling face. 
A \emph{ceiling} is a $*$-connected set of ceiling faces. By construction, all faces in a ceiling $C$ have the same $e_3$ coordinate, and we can call that the height of the ceiling $\hgt(C)$.  
A \emph{wall} is a $*$-connected set of wall faces. Clearly, the projections of distinct walls are disjoint. 
\end{definition}

See Figure~\ref{fig:walls-ceil-ex} for a depiction of these definitions in subtle scenarios.

\begin{figure}
    \centering
    \raisebox{-13pt}{\includegraphics[width=0.3\textwidth]{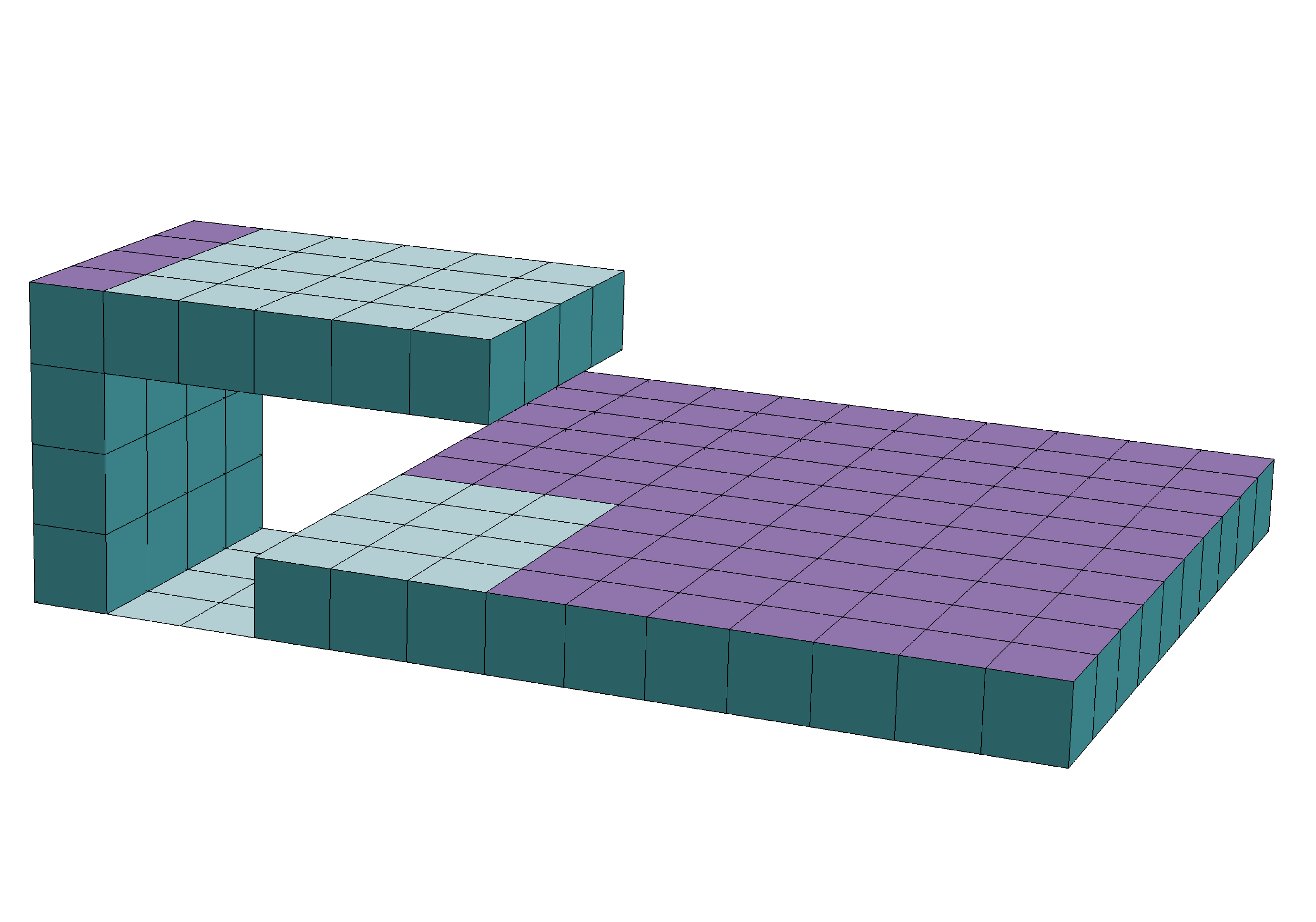}}
    \hspace{10pt}
    \includegraphics[width=0.2\textwidth]{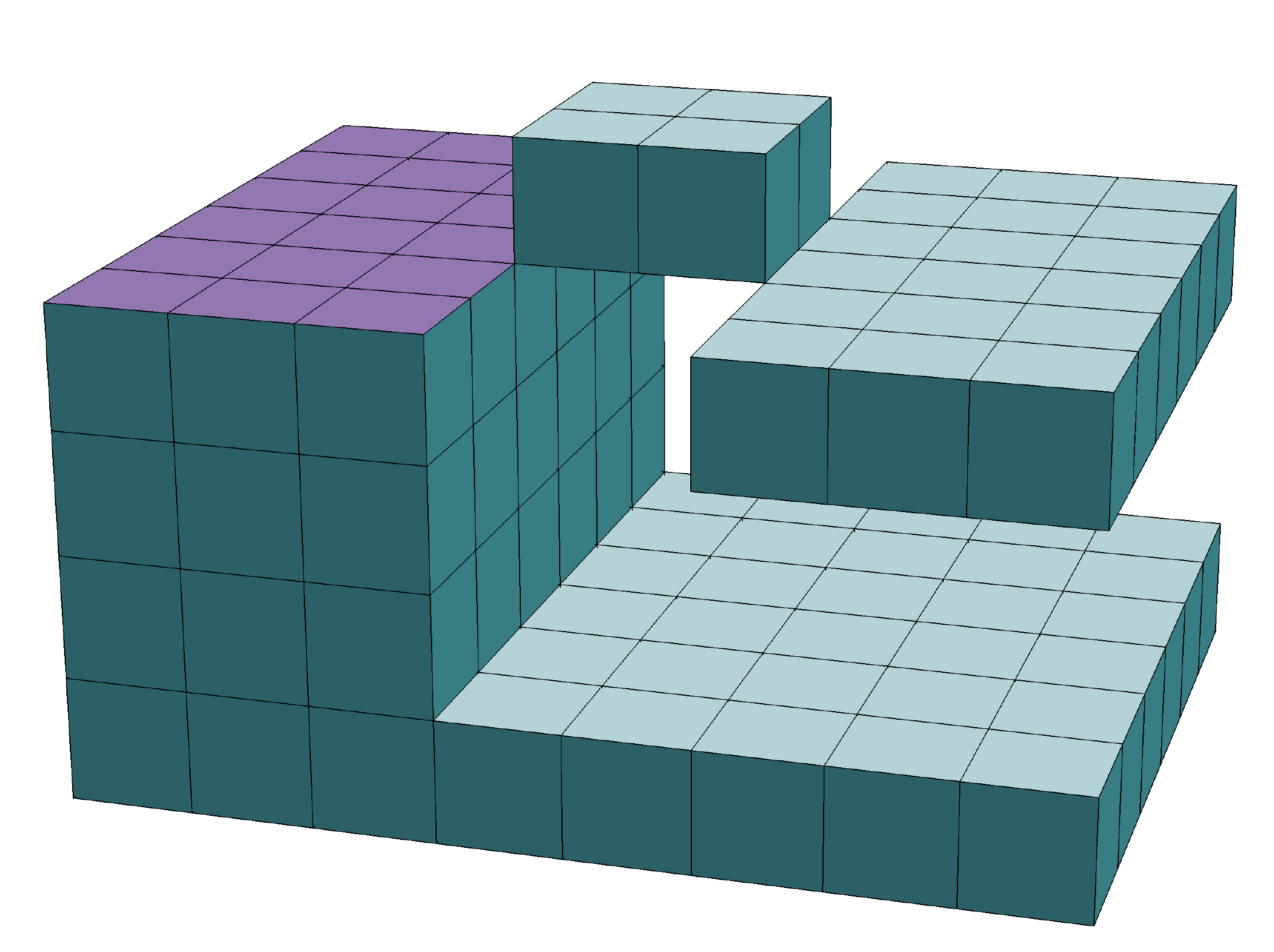}
    \hspace{20pt}
    \includegraphics[width=0.23\textwidth]{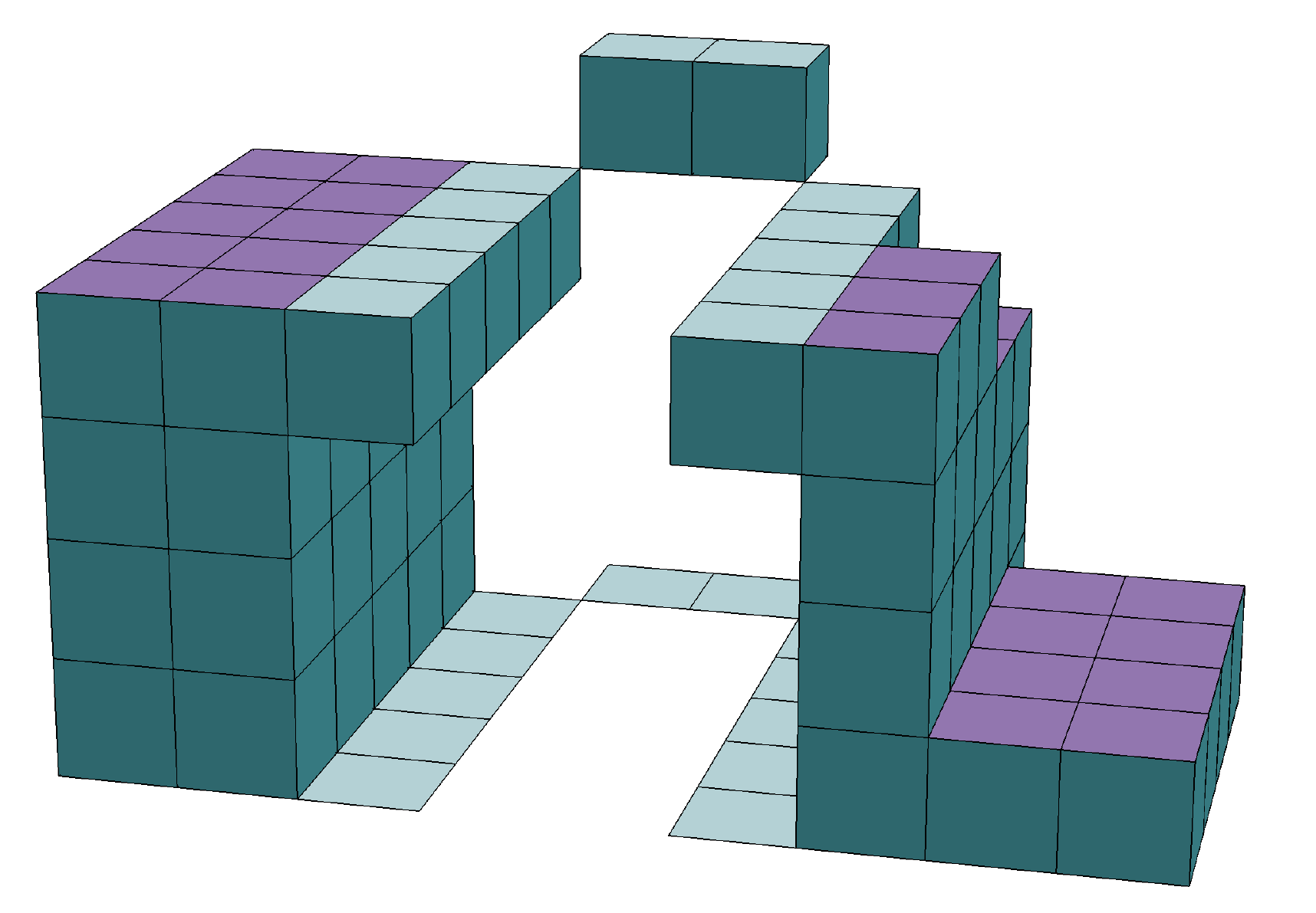}
    \vspace{-0.2in}
    \caption{Three distinct standard walls, with their interior ceiling faces (purple) and wall faces (vertical in teal, horizontal in light blue) as per Definition~\ref{def:walls-and-ceilings}. Left example features two distinct ($+$) components in $\R^2\times \R_+$ which correspond to two distinct pillars but form a single wall. 
Middle and right examples feature wall faces that would be disconnected via nearest-neighbor connectivity, but form a wall using $*$-connectivity.
    }
    \label{fig:walls-ceil-ex}
\end{figure}

\begin{definition}[Floors and ceilings of a wall]
For a wall $W$, define the complement of its projection 
\begin{align*}
    \rho(W)^c: = (\cF(\cL_0) \cup \cE(\cL_0))\setminus \rho(W)\,,
\end{align*}
and notice that it splits into an infinite connected component, and some finite ones (here connectivity is seen in $\R^2$). Any ceiling  adjacent to a wall $W$ projects into one of the connected components of $\rho(W)^c$. Call that ceiling that projects into the infinite component of $\rho(W)^c$ the \emph{floor} of $W$, denoted by $\lfloor W\rfloor$ and collect all other ceilings adjacent to $W$ into $\lceil W\rceil$. For distinct walls $W, W'$, the sets $\lceil W\rceil$ and $\lceil W'\rceil$ are disjoint.
\end{definition}

Importantly, given all the walls of an interface $\cI$, one can reconstruct the full interface by iteratively reading off the heights of the ceilings from the wall collection. 

\begin{definition}[Standard walls]
A wall $W$ is a \emph{standard wall} if there exists an interface $\cI_W$ such that its only wall is $W$. As such, a standard wall must have that $\lfloor W\rfloor \subset \cL_0$.

For a wall $W$, we define its \emph{standardization} (called \emph{drift} in~\cite{Dobrushin72a}) as its vertical shift by $-\hgt(\lfloor W\rfloor)$ and denote it by $\theta_{\textsc{st}} W$. For any wall $W$, its standardization is a standard wall. 
\end{definition}

\begin{remark}[Indexing of walls]
We can index the walls of $\cI$ as follows: assign an arbitrary ordering to the faces of $\cL_0$. Index a wall $W$ by the minimal face of $\cL_0$ that is interior to $W$ and incident to $\rho(W)$. Clearly, for any admissible collection of standard walls, the indices of distinct walls are distinct.   
\end{remark}

Since the projections of walls are distinct, the projection of a wall $W'$ is a subset of $\rho(W)^c$: if it is a subset of one of the finite components of $\rho(W)^c$, we say that $W'$ is \emph{nested in} $W$ and write $W' \Subset W$. In this way, to every $W' \Subset W$, we can identify a ceiling in $\lceil W\rceil$ which is the one projecting into that same finite component of $\rho(W)^c$. 

\begin{definition}[Nesting of walls]
We say $u\in \cE(\cL_0)\cup \cF(\cL_0)$ is \emph{interior to} a wall $W$ if it is not in the infinite component of $\rho(W)^c$. We say that $f\in \cI$ is interior to $W$ if $\rho(f)$ is interior to $W$.

Then for any $u\in \cE(\cL_0)\cup \cF(\cL_0)$, we define its \emph{nested sequence of walls} $\fW_u$ as the collection of all walls to which $u$ is interior. By the definition above, this forms a nested collection $W^u_1 \Subset W^u_2\Subset\cdots$.
\end{definition}

We say a collection of standard walls is \emph{admissible} if their projections are distinct.
The following lemma shows a bijection between admissible collections of standard walls and interfaces (also see~\cite[Lemma~2.12]{GL19a} for more details).

\begin{lemma}[The standard wall representation of $\cI$]\label{lem:standard-wall-representation}
There is a 1-1 correspondence between admissible collections of standard walls and interfaces. Namely, to obtain the standard wall representation of an interface $\cI$, take the union of the standardizations of all its walls. From an admissible collection of standard walls, recover an interface as follows: 

\begin{enumerate}
\item Iteratively, for every standardization of a wall $\theta_{\textsc{st}}W$,
\begin{itemize}
    \item If $W'\Subset \theta_{\textsc{st}}W$ and is identified with ceiling $C\in \lceil \theta_{\textsc{st}} W\rceil$, then shift $W'$ by $\hgt(C)$.
    \end{itemize}
    \item From this wall collection, fill in the ceiling faces to obtain the interface $\cI$. 
\end{enumerate}
\end{lemma}

Using the standard wall representation defined above, we note the following important observation.

\begin{observation}\label{obs:1-to-1-map-faces}
Consider interfaces $\cI$ and $\cJ$, such that the standard wall representation of $\cI$ contains that of $\cJ$ (and additionally has the standardizations $\mathbf W=W_1,...,W_r$). By the construction  in Lemma~\ref{lem:standard-wall-representation}, there is a 1-1 map between the faces of $\cI \setminus \mathbf W$ and the faces of $\cJ \setminus \mathbf H$ where $\mathbf H$ is the set of faces in $\cJ$ projecting into $\rho (\mathbf W)$. Moreover, this bijection can be encoded into a map $f\mapsto \theta_{\udarrow} f$ that only consists of vertical shifts, and such that all faces projecting into the same component of $\rho(\mathbf W)^c$ undergo the same vertical shift. 
\end{observation}

\begin{definition}For a wall $W$, define its excess area as 
\begin{align*}
    \fm(W) : = \fm(\cI_{\theta_{\textsc{st}}W}; \cL_{0,n}) = |W| - |\cF(\rho(W))|\,, 
\end{align*}
and notice that this always satisfies 
\begin{align}\label{eq:wall-excess-area-inequalities}
    \fm(W) \geq \frac 12 |W| \qquad \mbox{and} \qquad \fm(W)\geq |\cE(\rho(W))| + |\cF(\rho(W))|\,.
\end{align}
Notice that for an interface $\cI$ with standard wall collection $(W_z)_{z\in \cL_{0,n}}$, we have $\fm(\cI;\cL_{0,n}) = \sum \fm(W_z)$.
\end{definition}

\begin{definition}[Closeness and groups of walls]
We say that two walls $W$ and $W'$ are close if there exist $u\in \rho(W)$, $u'\in \rho(W')$ such that 
\begin{align*}
    |u - u'| \leq \sqrt{N_\rho(u)} + \sqrt{N_\rho(u')}\,.
\end{align*}
A collection of walls $F = \bigcup_i W_i$ is a \emph{group of walls} if every wall in $F$ is close to another wall in $F$, and no wall not in $F$ is close to a wall in $F$. 
For a nested sequence of walls $\fW_u = W^u_1 \cup W^u_2 \cup \ldots$, this allows us to collect the union of all its groups of walls into 
\begin{align*}
    \fF_u : = \bigcup_i F^u_i \qquad \mbox{where} \qquad F^u_i \mbox{ is the group of walls containing $W^u_i$}\,.
\end{align*}
For collections of walls, e.g., groups of walls, nested sequences, define their excess area as the sum of the excess areas of the constituent walls.

Groups of walls are indexed by the minimal index of their constituent walls. However, notice that we do not employ a unique labeling procedure for nested sequences of walls or their groups of walls of nested sequences of walls; if $u, u'$ are both interior to $W$, then $\fW_u \cap \fW_{u'}\neq \emptyset$. 
\end{definition}

\section{Decomposition of tall pillars}\label{sec:pillars}
\subsection{Pillars and increments}
In~\cite{GL19a}, the authors introduced pillars and their decomposition into an increment sequence in order to understand the large deviations of the interface $\cI$ (e.g., its structure at points where it attains atypically large heights). In this section we recall these definitions, though we note crucially that the division of pillars into spines and bases has been modified from the prequel; in~\cite{GL19a} we absorbed imprecisions of $O(\log \log n)$, which we cannot afford when proving tightness of the maximum.

\begin{definition}[Pillar]
For an interface $\cI$ and a face $x \in \cF(\cL_0)$, we define the pillar $\cP_x$ as follows: consider the Ising configuration $\sigma(\cI)$ and let $\sigma(\cP_x)$ be the (possibly empty) $*$-connected plus component of the cell with mid-point $x+(0,0, \frac 12)$ in the upper half-space $\cL_{>0}$. The face set $\cP_x$ is then the set of bounding faces of $\sigma(\cP_x)$ in  $\cL_{>0}$.  
\end{definition}

The following relation between pillars and nested sequences of walls is important. 

\begin{observation}\label{obs:pillar-nested-sequence-of-walls}
The walls of the pillar $\cP_x$ are contained in the nested sequence of walls $\fW_x= \bigcup_i W^x_i$ together with all walls nested in some $W^x_i$. Namely, if $\cI$ and $\cJ$ agree on $\fW_x$ and on all walls nested in walls of $\fW_x$, then $\cP_x^\cI = \cP_x^\cJ$. Therefore, if $f\in \cP_x$, there exists $W$ such that both $f$ and $x$ are interior to~$W$. 
\end{observation}

\begin{definition}[Cut-points]
A half-integer $h\in \{\frac 12, \frac 32, \ldots \}$ is a \emph{cut-height} of $\cP_x$ if the intersection $\sigma(\cP_x) \cap \R^2 \times [ h- \frac 12, h+\frac 12]$ consists of a single cell. In that case, that cell (identified with its midpoint $v\in (\Z+\frac 12)^3$) is a cut-point of $\cP_x$. 
We enumerate the cut-points of $\cP_x$ in order of increasing height as $v_1 , v_2 , \ldots$.  
\end{definition}

\begin{definition}[Spine and base]
The spine of $\cP_x$, denoted $\cS_x$ is the set of cells in $\sigma(\cP_x)$ (resp., faces in $\cP_x$) intersecting $\R^2 \cap [\hgt(v_1), \infty)$. The base $\sB_x$ of $\cP_x$ is the set of cells in $\sigma(\cP_x) \setminus \cS_x$ (resp., faces in $\cP_x \setminus \cF(\cS_x)$).
\end{definition}

\begin{remark}
We draw attention to the fact that our decomposition of the pillar into a spine and base differs from that used in~\cite{GL19a}. There, the beginning of the spine was marked not by $v_1$ but by a random $v_{\Tsp}$: the first cut-point to, informally, have height greater than all other pillars in a radius of $R\propto \hgt(\cP_x)$. This was tailored to the fact that we could sustain errors that were logarithmic in the height of the pillar.  
\end{remark}

\begin{definition}[Increments]
We decompose a spine $\cS_x$ with cut-points $v_1,v_2,\ldots$ into its constituent increments. If there are at least $\sT+1 \geq 2$ cut-points, for every $i\leq \sT$, define the $i$-th increment as
\[\sX_i = \cS_x \cap (\R^2 \times [\hgt(v_i) - \tfrac 12,\hgt(v_{i+1})+\tfrac 12])\,,
\]
so that the $i$-th increment is the subset of $\cS_x$ delimited from below by $v_i$ and from above by $v_{i+1}$ and there are exactly $\sT$ increments. (If there are fewer than two cut-points, we say that $\sT=0$.)

Besides the increments, the spine additionally may have a \emph{remainder} $\sX_{>\sT}$, which we define as the set of faces intersecting $\R^2 \times [\hgt(v_{\sT+1}), \infty)$. For readability, for a spine $\cS_x$ with increment sequence $\sX_1, \ldots, \sX_{\sT}, \sX_{>\sT}$, we use the notation $\sX_{\sT+1}:= \sX_{>\sT}$ so that we can consistently index over increments and the remainder.
\end{definition}

Abusing notation, we may view increments not as subsets of an interface, but as  finite $*$-connected set of cells with at least two cells, and whose only cut-points are its bottom-most and top-most cells (modulo lattice translations, achieved by, say, rooting them at the origin). Call the set of all such increments $\fX$. The face-set of such an increment consists of all its bounding faces except its bottom-most and top-most horizontal ones.  A remainder increment is defined similarly, but its only cut-point is its bottom-most cell.  

\begin{lemma}
There is a 1-1 correspondence between triplets of $v_1$, a sequence of increments  $(X_1 ,\ldots, X_{T})\in \fX^{T}$ and a remainder $X_{>T}$, and possible spines of $T$ increments with first cut-point at $v_1$.    
\end{lemma}

Indeed this follows by identifying the bottom cut-point of $X_1$ with $v_1$ and sequentially translating the increments in the increment sequence to identify their bottom cut-point with the top cut-point of the previous increment. For more details, see \cite[Section~3]{GL19a}.

The simplest increment is what we call the \emph{trivial increment} $X_\trivincr$, consisting of two vertically consecutive cells, one on top of the other (resp., its eight bounding vertical faces). In proofs where we show that increments have exponential tails, the maps we apply \emph{trivialize} an increment $\sX_j$ by replacing it in the increment sequence of $\cS_x$ by $\hgt(v_{j+1}) - \hgt(v_j)$ consecutive trivial increments.  
Excess areas of increments will be defined w.r.t.\ this trivialization scheme. Namely, for an increment $\sX_i$ ($1\leq i \leq \sT$), define $\fm(X)$ as
 \begin{align}\label{eq:increment-excess-area}
     \fm(\sX_i) = |\cF(
     \sX_i)| - 4(\hgt(v_{i+1}) - \hgt(v_i)+1)
 \end{align}
 (recall that $\cF(X)$ does not include the top most and bottom most faces bounding $X$). For the remainder increment $\sX_{\sT+1} = \sX_{>T}$, where $v_{\sT+2}$ does not exist, this can be defined consistently by arbitrarily setting $\hgt(v_{\sT+2}):=\hgt(\cP_x)- \frac 12$. With these definitions, we notice that if $\sX_i\neq X_\trivincr$ then
 \begin{align*}
     \fm(\sX_i) \geq 2(\hgt(v_{i+1})- \hgt(v_i) -1)\,\vee\, 2\qquad \mbox{and} \qquad |\cF(\sX_i)|\leq 3\fm(\sX_i)+4\,,
 \end{align*}
 since the intersection of $\sX_i$ with any height which is not a cut-height has at least six faces vs.\ four faces in a trivial increment (a nontrivial increment $X$ that has height $1$ satisfies $\fm(X)=2$ and $|\cF(X)|=10$).
 
 For a spine $\cS_x$ and a fixed $T$, we define its excess area with respect to the reference increment sequence of $T$ trivial increments by 
\begin{align*}
\fm(\cS_x) = \Big(\sum_{i\leq \sT} \fm(\sX_i)\Big) + \fm(\sX_{>\sT})\,.
\end{align*}
We can define an excess area of the base of a pillar as being with respect to the pillar of the same height and no base: for a pillar $\cP_x$ with base $\sB_x$ and first cut-point $v_1$, define 
$$\fm(\sB_x): = |\cF(\sB_x)|- |\cF(\rho(\sB_x))| - 4(\hgt(v_1)- \tfrac 12)\,.$$

 For an $x \in \cL_{0,n}$, collect the interfaces with $\cP_x$ having at least $T$ increments and at least height $H$ in 
\begin{align*}
    \bI_{x,T,H}  = \{\cI: \sT\geq T \mbox{ and } \hgt(\cP_x)\geq H\}\,.
\end{align*}

\subsection{Preliminary estimates on tall pillars} In this section we recap some results which can be deduced from Dobrushin's proof of rigidity~\cite{Dobrushin72a} and simple modifications around that argument, together with the definitions of pillars and increments. See~\cite{GL19a} for short proofs of these. 

\begin{proposition}\label{prop:dobrushin-exp-tail}
There exists $C$ and $\beta_0$ such that for every $\beta>\beta_0$, for every $x\in \cL_{0,n}$ and every $r\geq 1$, 
\begin{align*}
    \mu_n^\mp(\fm(W_x) \geq r)\leq e^{- (\beta - C)r}\,, \qquad \mbox{and}\qquad \mu_n^\mp(\fm(\fF_x)\geq r)\leq e^{ - (\beta - C)r}\,.
\end{align*}
\end{proposition}

As a consequence of Observation~\ref{obs:pillar-nested-sequence-of-walls}, the proposition implies the following. 

\begin{corollary}
There exists $C$ and $\beta_0$ such that for every $\beta>\beta_0$, for every $x\in \cL_{0,n}$ and every $r\geq 1$, 
\begin{align*}
    \mu_n^\mp(\max \{h: (x_1,x_2,h)\in \cI\} \geq r)\leq e^{ - 4(\beta - C)r}\qquad \mbox{and} \qquad \mu_n^\mp (\hgt(\cP_x) \geq r)\leq e^{ -4( \beta - C)r}\,.
\end{align*}
\end{corollary}

In fact, by a simple application of the FKG inequality and forcing argument, we obtain a corresponding lower bound, yielding the following.  

\begin{proposition}[see~{\cite[Prop.~2.29]{GL19a}}]\label{prop:Ex-bounds}
There exists $C>0$ and a sequence $\epsilon_\beta$ vanishing as $\beta \to\infty$ such that for every $\beta>\beta_0$, $x\in \cL_{0,n}$ and $h\geq 1$, 
\[ (1-\epsilon_\beta)e^{ - (4\beta + e^{ - 4\beta})h}\leq \mu_n^\mp(\hgt(\cP_x)\geq h)\leq e^{ - 4(\beta - C)h}\,.
\]
\end{proposition}

\subsection{Tame pillars} In this section we consider the set of all pillars that have at least $T$ increments and reach a height $H$. We show that a subset of them, which we call \emph{tame} have large probability and from there on in Section~\ref{sec:exp-tail-base}, we restrict attention to tame pillars on which our future maps will be well-defined. 

Notice that if $H \leq T$, the event $\{\hgt(\cP_x)\geq H\}$ is vacuous, so we take $T< H$. 
\begin{definition}\label{def:tameness}
For a given $x\in \cL_{0,n}$ and $T< H$, we say that an interface $\cI \in \bI_{x,T,H}$ is \emph{tame} if $\cI$ is in
\begin{align*}
    \bar{\bI}_{x,T,H} : = \{ \cI \in \bI_{x,T,H}: \diam(\sB_x) + \tfrac 14 \fm (\cS_x)< d(x,\partial \Lambda_n)\}
\end{align*}
We observe geometrically that for a pillar $\cP_x$, 
\begin{align}\label{eq:geometric-observation-tameness}
    \max_{f\in \cP_x} d(x,\rho(f))\leq \diam(\sB_x)+\frac 14 \fm(\cS_x)\,.
\end{align}
\end{definition}
Notice that this is less restrictive than the corresponding definition of tameness from~\cite{GL19a} as it is the minimal requirement for our (more robust) map in Section~\ref{sec:exp-tail-base} to be well-defined. 

\begin{proposition}\label{prop:tameness}
There exists $C>0$ such that for every $\beta>\beta_0$ and every $x\in \cL_{0,n}$ and $T,H$ satisfying $0 \leq T < H \ll r$, we have 
\[\mu_n^\mp \big(\diam(\sB_x)+ \tfrac 14 \fm(\cS_x) \geq r\mid \bI_{x,T,H}\big) \leq Ce^{ - 2(\beta - C) r}\,,
\]
and, in particular, taking $r= d(x,\partial \Lambda_n)$, 
\[
\mu_n^\mp(\bar{\bI}_{x,T,H}^c \mid \bI_{x,T,H}) \leq Ce^{- 2(\beta-C) d(x,\partial \Lambda_n)}\,.
\]
\end{proposition}
\begin{proof}
This can be read off from a combination of various preliminary bounds in~\cite{Dobrushin72a,GL19a}. For the sake of completeness, and as an indication of the structure of the proofs in Section~\ref{sec:exp-tail-base}, we present a full proof using a map $\Phi_{x,T,H}$ which deletes the pillar $\cP_x$ and replaces it by a column of $H-1$ trivial increments. 

Namely, let $\Phi_{x,T,H}:\bI_{x,T,H} \to \bI_{x,T,H}$ be the following map. 
First, denote by \begin{align*}
\fF_{[x]}= \fF_x \cup \bigcup_{f\in \partial_0 x} \fF_f
\end{align*} (where $\partial_0 x$ are the four faces of $\cL_0$ adjacent to  $x$). 
Then, from an interface $\cI$ we obtain $\Phi_{x,T,H}(\cI)$ as the interface with the following standard wall representation:
\begin{enumerate}
    \item Remove the standardizations of $\fF_{[x]}\cup \fF_{\rho(v_1)}$ from the standard wall representation of $\cI$.
    \item Add the standard wall consisting of the bounding faces of a stack of $H-1$ trivial increments above $x$ (i.e., the cells with midpoints $\{x+ (0,0,n+\frac 12): n = 1,\ldots,H\}$). 
\end{enumerate}
(In the exceptional case $\sT=0$ and $v_1$ doesn't exist, interpret $v_1$ as any site of $\sigma(\cP_x)$ whose height is $\hgt(\cP_x)- \frac 12$.)
It is straightforward that $\Phi_{x,T,H}(\cI)$ is a valid interface in $\bI_{x,T,H}$ as we deleted all walls containing $x$ or its adjacent faces in their interior in step (1), so that adding the wall in step (2) preserves the admissibility of the standard wall collection. The pillar $\cP_x^{\Phi_{x,T,H}(\cI)}$ 
of the resulting interface
clearly consists only of the wall added in step (2) and therefore it has $H-1\geq T$ trivial increments and reaches height $H$; hence,~$\Phi_{x,T,H}(\cI) \in \bI_{x,T,H}$. 

Notice that for every $\cI\in \bI_{x,T,H}$, 
$$\fm(\cI; \Phi_{x,T,H}(\cI)) = \fm(\fF_{[x]} \cup\fF_{\rho(v_1)}) - 4H \geq 2\diam(\sB_x) + \frac 12 \fm(\cS_x)\,.$$
As such, 
it suffices for us to show the bound 
\begin{align}\label{eq:need-to-show-tameness}
    \mu_n^\mp(\fm(\cI; \Phi_{x,T,H} (\cI))\geq 2r \mid \bI_{x,T,H}) \leq e^{ - 2(\beta - C) r}\,.
\end{align}

We first consider how $\Phi_{x,T,H}$ transforms weights of interfaces. Namely, we claim that the map sends interfaces of low probability to ones of higher probability: there exists $C>0$ such that for every $\cI \in \bI_{x,T,H}$ having $\fm(\cI; \Phi_{x,T,H}(\cI)) \geq H$, 
\begin{align}\label{eq:interface-weights-tameness}
    \Big| \log \frac{\mu_n^\mp(\cI)}{\mu_n^\mp(\Phi_{x,T,H}(\cI))} + \beta \fm(\cI;\Phi_{x,T,H}(\cI))\Big| \leq C\fm(\cI; \Phi_{x,T,H}(\cI))\,.
\end{align}

For ease of notation, let $\cJ = \Phi_{x,T,H}(\cI)$. 
We split the set of faces in $\cI$ and $\cJ$ into the following: 
\begin{itemize}
    \item $\bW$: the faces in $\fF_{[x]} \cup \fF_{\rho(v_1)}$ in $\cI$.
    \item $\bB$: all other faces in $\cI$ (consisting of all ceiling faces of $\cI$ along with all wall faces besides $\bW$). 
    \item $\bH$: the set of faces in $\cJ$ whose projection is in $\cF(\rho(\bW))$. 
    \item $W_x^\cJ$: the set of faces in $\cJ$ from the wall added in step (2) of $\Phi_{x,T,H}$. 
    \item $\theta_{\udarrow}\bB$: all other faces in $\cJ$. 
\end{itemize}
By Lemma~\ref{lem:standard-wall-representation}, there is a 1-1 correspondence between $\theta_{\udarrow} \bB$ and faces in $\bB$ given by the vertical shifts induced by ceilings of deleted/added walls from the standard wall representation: encode this 1-1 correspondence into $f\mapsto \theta_{\udarrow} f$. 
With this splitting in hand, by Theorem~\ref{thm:cluster-expansion}, we need to bound
\begin{align*}
    \Big| \sum_{f\in \cI} \g(f,\cI) - \sum_{f'\in \cJ} \g(f',\cJ)\Big| & \leq \sum_{f\in \bW} |\g(f, \cI)| + \sum_{f'\in \bH} |\g(f', \cJ)| + \sum_{f'\in W_x^\cJ} |\g(f',\cJ)| \\ 
    & \qquad + \sum_{f\in \bB} |\g(f, \cI) - \g(\theta_{\udarrow} f, \cJ)|\,.
\end{align*}
The first term is at most $\bar K |\bW| \leq 2\bar K \fm(\bW)$ by~\eqref{eq:g-uniform-bound} and \eqref{eq:wall-excess-area-inequalities}. The second term is similarly at most $\bar K[|\cE(\rho(\bW))|+ |\cF(\rho(\bW))|]\leq \bar K \fm(\bW)$ and the third term is at most $4\bar K H$. 
The last term satisfies 
\begin{align*}
    \sum_{f\in \bB}|\g(f,\cI)- \g(\theta_{\udarrow}f,\cJ)|\leq \sum_{f\in \bB} \bar K e^{ - \bar c \br (f,\cI;\theta_{\udarrow}\cJ)}
\end{align*}
Since the distance between two faces is at least the distance between their projections, and the radius $\br(f,\cI; \theta_{\udarrow} f, \cJ)$ must be attained by a wall face, we see that the right-hand side is in turn at most 
\begin{align*}
 \sum_{f\in \bB} \sum_{u\in \rho(\bW \cup W_x^\cJ)} e^{ - \bar c d(\rho(f), u)} = \sum_{u'\in \rho(\bW \cup W_x^\cJ)^c} \sum_{u\in \rho(\bW\cup W_x^\cJ)} N_\rho(u') e^{ - \bar c d(u,u')}
\end{align*}
By definition of groups of walls, for every $u' \in \rho(\bW\cup W_x^\cJ)^c$, we have $N_\rho(u')\leq |u- u'|^2 +1$ (if $u'$ is the projection of a ceiling face, $N_\rho(u')= 1$) and therefore, the right-hand side above is at most $\bar K [2\fm(\bW) + 4H]$.    

Altogether, this implies that for some $C$, we have the bound 
\begin{align*}
\Big|\sum_{f\in \cI} \g(f,\cI) - \sum_{f'\in \cJ} \g(f',\cJ)\Big| \leq C[\fm(\bW) + 4H]
\end{align*}
which implies the bound of~\eqref{eq:interface-weights-tameness} for a different $C$ as long as $\fm(\cI;\cJ)\geq H$, say. 

On the other hand, let us bound the \emph{multiplicity} of the map $\Phi_{x,T,H}$. Namely, we bound the number of elements in the pre-image $\{\cI \in \Phi_{x,T,H}^{-1}(\cJ): \fm(\cI;\cJ)= M\}$ by some $C_\Phi^M$ uniformly over $\cJ\in \Phi_{x,T,H}(\bI_{x,T,H})$.
To do so, we associate to each possible such $\cI$, a $*$-connected face subset of $\Z^3$ rooted at $x$, together with a coloring of those faces by $\{\blue,\red\}$, and bound the number of possible such so-called \emph{witnesses}, from which together with $\cJ$ we can reconstruct $\cI$. Our witness will consist of the following:  
\begin{enumerate}
\item Take the standardizations of all walls in $\fF_{[x]} \cup \fF_{\rho(v_1)}$. Color all these faces $\blue$. 
\item For every $u\in \rho(\bW)$, add all faces in $\cL_0$ a distance at most $\sqrt{N_\rho(u)}$ from $u$ and color them $\red$.
\item Connect (via a shortest path of faces in $\cL_0$) $x$ to $W^x_1$ and $W^x_i$ to $W^x_{i+1}$ for all $i$. Do the same for $\rho(v_1)$. Also, add the face at $x$ and connect $x$ to $\rho(v_1)$. Color those faces added $\red$. 
\end{enumerate}
That this forms a $*$-connected face subset follows from the definition of closeness of walls. 
One can easily recover $\cI$ from $\cJ$ and the witness by taking the standard wall representation of $\cJ$, removing from it $W_x^\cJ$ and adding in all $\blue$ faces of our witness to obtain the standard wall representation of $\cI$. 

The number of $\blue$ faces in a witness corresponding to an interface $\cI$ with $\fm(\cI;\cJ) = M$ is at most 
$|\fF_{[x]}\cup \fF_{\rho(v_1)}|\leq 2\fm(\bW)$. The number of \red\ faces added in step (2) of the witness construction is, by definition of closeness and groups of walls, at most 
$$\sum_{u\in \rho(\bW)} N_\rho(u)\leq |\bW|\leq 2\fm(\bW)\,.$$
Finally, by Observation~\ref{obs:pillar-nested-sequence-of-walls}, there is some wall that is deleted to which both $x$ and $\rho(v_1)$ are interior. The number of \red\ faces added in step (3) of the witness construction is therefore at most 
$$ 3\fm(\bW) + \sum_i \fm(W_i^x)+ \sum_i \fm(W_i^{\rho(v_1)})\leq 5\fm(\bW)\,.$$
The number of possible witnesses corresponding to interfaces $\cI$ with $\fm(\cI;\cJ) = M$ is then at most the number of possible rooted face subsets of $\Z^3$ with at most $10 M$ faces,  multiplied by the number of possible $\{\blue, \red\}$ colorings of those faces. Recall the following combinatorial fact (see e.g.,~\cite{Dobrushin72a}).

\begin{fact}\label{fact:number-of-connected-face-sets}
There exists a universal constant $s>0$ such that the number of $*$-connected face-subsets rooted at (incident to) a fixed vertex, edge, or face of $\Z^3$, consisting of at most $M\geq 1$ faces, is at most $s^M$. 
\end{fact}

With the above fact in hand, we see that there are at most $s^{10M}$ choices for the face subset of the witness, and an additional multiplicative $2^{10M}$ for the number of possible colorings of those faces. 

Combining this multiplicity bound with~\eqref{eq:interface-weights-tameness}, we can deduce~\eqref{eq:need-to-show-tameness} as follows: for $r\geq H$, 
\begin{align*}
    \mu_n^\mp(\fm(\cI; \Phi_{x,T}(\cI))\geq r, \bI_{x,T,H}) & \leq 
    \sum_{M \geq r}\,\sum_{\cJ\in \Phi_{x,T,H}(\bI_{x,T,H})} \, \mu_n^\mp(\cJ) \sum_{\cI\in \Phi_{x,T,H}^{-1}(\cJ): \fm(\cI;\cJ) = M} e^{ - (\beta - C)M}  \\ 
    & \leq \sum_{M\geq r} \,\sum_{\cJ \in \Phi_{x,T,H}(\bI_{x,T,H})} e^{ - (\beta - C - \log C_\Phi) M }\mu_n^\mp(\cJ) 
\end{align*}
Since $\Phi_{x,T,H}(\bI_{x,T,H}) \subset \bI_{x,T,H}$, we have for some other $C>0$ 
\begin{align*}
    \mu_n^\mp(\fm(\cI; \Phi_{x,T}(\cI)) \geq r, \bI_{x,T,H}) \leq Ce^{-(\beta - C)r} \mu_n^\mp(\bI_{x,T,H})\,,
\end{align*}
from which dividing by $\mu_n^\mp(\bI_{x,T,H})$, we obtain~\eqref{eq:need-to-show-tameness}. 
\end{proof}

\section{Sharp estimates on the structure of tall pillars}\label{sec:exp-tail-base}
In this section, we obtain estimates on the structure of tall pillars (conditionally on $\bI_{x,T,H}$) up to $O(1)$ precision. This is a prerequisite to obtaining tightness of the maximum $M_n$ via a second moment method, as the size of the base $\sB_x$ contains the positive correlations between the events $\{\hgt(\cP_x) \geq h\}$ and $\{\hgt(\cP_y)\geq h\}$: e.g., if the base $\sB_x$ contains $y$ in its interior, the events are fully correlated. 

\begin{theorem}\label{thm:exp-tail-base}
There exist $\beta_0,C>0$ such that for every $\beta >\beta_0$, every $x\in \cL_{0,n}$ and $T,H$ satisfying $0\leq T< H \ll d(x,\partial \Lambda_n)$, the following holds. 
\begin{enumerate}[(a)]
    \item \emph{Base estimate:} for every $1\leq r \leq d(x,\partial \Lambda_n)$, 
    \begin{align*}
    \mu_{n}^{\mp} \big(|v_1 - (x+(0,0, \tfrac 12))|\geq r\mid \bI_{x,T,H}\big) \leq C e^{ - 2(\beta - C)r}\,,
\end{align*}
and in fact, 
\begin{align*}
    \mu_{n}^{\mp}\big(\fm (\sB_x)\geq r\mid \bI_{x,T,H}\big)\leq Ce^{ - (\beta - C) r}\,.
\end{align*}
\item \emph{Increment estimate:} for every $t\geq 1$, and every $1 \leq r \leq d(x, \partial \Lambda_n)$, 
\begin{align*}
    \mu_n^\mp \big(\fm(\sX_t) \geq r \mid \bI_{x,T,H}\big)\leq Ce^{ - (\beta - c)r}\,.
\end{align*}
\end{enumerate}
\end{theorem}

\begin{proposition}\label{prop:exp-tail-all-increments}
There exist $\beta_0,C>0$ such that for every $\beta>\beta_0$, every $x\in \cL_{0,n}$ and $T,H$ satisfying  $0\leq  T< H \ll d(x,\partial \Lambda_n)$, the following holds. For every half-integer $\ell\geq \frac 12$ and $1\leq r\leq d(x,\partial \Lambda_n)$, 
\begin{align*}
    \mu_{n}^{\mp}\big(|\cF(\cP_x\cap \cL_{\ell})| \geq 4+r \mid \bI_{x,T,H} \big)\leq Ce^{ - (\beta - C) r}
\end{align*}
\end{proposition}

Recall from the introduction that in~\cite{GL19a} the authors proved a bound of $O(\log H)$ on $\diam(\sB_x)$ and the exponential tails on increment sizes were restricted to those with index above $O(\log H)$. In that work, affording an $O(\log H)$ error, the interactions between the horizontal shifts of the spine under the map were decoupled from the base and nearby pillars because, with high probability, no other pillars in the shadow of the pillar $\cP_x$ reach a height larger than $O(\log H)$. 

At heights that are $O(1)$, we need to deal directly and simultaneously with the interactions between vertical shifts (arising from deletions of groups of walls as in~\cite{Dobrushin72a}) and the horizontal shifts arising from trivializing increments and shifting the spine appropriately. This induces substantial complications. 
After defining a map $\Psi_{x,t}:\bar{\bI}_{x,T,H}\to \bI_{x,T,H}$, in Section~\ref{subsec:reader-guide-map}, we give a reader's guide to the various difficulties encountered in construction of this map and justify the necessity of its various steps. 

\subsection{A new base and increment map}
In this section, we define a new map that shrinks the base of $\cP_x$ and trivializes the $t$-th increment of the pillar. The map is significantly more involved than the maps in~\cite{GL19a} as it deals directly with the interactions between the horizontal shifts of $\cS_x$ with the walls near its base which the spine may get close to or hit.

For an increment $j$, denote the centered \emph{trivialization} of $\sX_j$ by 
\[ \Theta_\trivincr \sX_j =  [-\tfrac12,\tfrac12]^2 \times [\hgt(v_j)- \tfrac 12, \hgt(v_{j+1})+\tfrac 12]\,. \]

Denote by $W_y^\cK$ the wall indexed by $y$ in the interface $\cK$. 
If $(\tilde W_z)_{z\in \cL_{0,n}}$ is the standard wall representation of $\cK$, let
\[
\Theta_\udarrow^\cK \tilde W_y = \{W_y^{\cK'}\cup \lceil W_y^{\cK'} \rceil\,:\; \cK' \mbox{ has standard wall representation }(\tilde W_z) \mbox{ or } (\tilde W_z \setminus \tilde \fF_{y'}) \mbox{ for some $y'$}\}\,;
\]
namely, $\Theta_\udarrow^\cK \tilde W_y$ is the set of all possible vertical shifts induced on $\tilde W_y$ via Lemma~\ref{lem:standard-wall-representation} by deleting the group of walls of a nested sequence of walls.

Finally, for some $x\in\cL_{0,n}$ identified with its midpoint, an interface $\cI$, and two shift vectors $\omega_1,\omega_2\in \Z^2\times\{0\}$, 
denote by $(\tilde W_z)$ the standard wall representation of $\cI \setminus \cS_x$, and for every wall $\tilde W_y$ define 
\[ \fD_x(\tilde W_y,j,\omega_1,\omega_2) = d(\Theta_\udarrow^{\cI\setminus \cS_x} \tilde W_y\,,\, \sX_j \cup \theta_{\rho(x+\omega_1+\omega_2)}\sX_j \cup \theta_{\rho(x+\omega_2)}\Theta_\trivincr \sX_j)\,,\]
where for $\omega \in \R^3$ and $A\subset \R^3$, $\theta_\omega A$ is the shift of $A$ by the vector $\omega$, i.e., $\theta_\omega A = A+\omega$.  
\begin{figure}
\begin{algorithm2e}[H]
\DontPrintSemicolon
\nl
Let $\{\tilde W_y : y\in \cL_{0,n}\}$ be the standard wall representation of the interface $\cI\setminus \cS_x$. Also let $\cO_{v_1}$ be the nested sequence of walls of $v_1$, so that $\theta_{\textsc{st}} \cO_{v_1} =  \tilde {\mathfrak W}_{v_1}$.\;

\BlankLine
\tcp{Base modification}
\nl
Mark $[x]=\{x\}\cup \partial_0 x$ and $\rho(v_1)$ for deletion (where $\partial_0 x$ denotes the four faces in $\cL_0$ adjacent to $x$).\;

\nl 
\If{the interface with standard wall representation $\tilde \fW_{v_1}$ has a cut-height}
{Let $h^\dagger$ be the height of the highest such cut-height.\;
Let $y^\dagger$ be the index of a wall that intersects $(\cP_x \setminus \cO_{v_1}) \cap \cL_{h^\dagger}$ and mark $y^\dagger$ for deletion.\;
}

\BlankLine
\tcp{Spine modification (A): the 1st increment}

\nl
Set $\fs_1\gets 0$ and $y^*_A \gets \emptyset$.\;
 
\For{$j=1$ \KwTo $\sT+1$}{     
Let $s\gets\fs_j$ and $\fs_{j+1}\gets \fs_j$.\;
     \If(\hfill\tcp*[h]{(A1)}){
\quad$ \fm(\sX_j) \geq j-1 $ \quad\mbox{}}{Let $\fs_{j+1} \gets j$.} 
     \If(\hfill\tcp*[h]{(A2)}){ 
	\quad$\fD_x( \tilde W_y,j,-v_{s+1},0)\leq \fm(\tilde W_y)$\,\quad  for some $y$\quad\mbox{}}{
      Let $\fs_{j+1} \gets j$ and mark for deletion every $y$ for which~({\tt A2}) holds.}
     \If(\hfill\tcp*[h]{(A3)}){\quad$\fD_x(\tilde W_y,j,-v_{s+1},0) \leq (j-1)/2$\qquad\quad for some $y$\quad\mbox{}}{
      Let $\fs_{j+1}\gets j$ and let $y^*_A $ be the minimal index $y$ for which~({\tt A3}) holds.}
}
 Let $j^* \gets \fs_{\sT+2}$ and mark $y^*_A$ for deletion.\;

\BlankLine
\tcp{Spine modification (B): the $t$-th increment}
\nl 
\eIf{$t>j^*$}{
Set $\mathfrak{s}_{t}\gets t-1$ and $y^*_B \gets \emptyset$.\;
\For{$k=t$ \KwTo $\sT+1$}{     
Let $s\gets \fs_k$ and $\fs_{k+1}\gets \fs_k$.\;
     \If(\hfill\tcp*[h]{(B1)}){
\quad$  \fm(\sX_k) \geq k - t $\quad\mbox{}}{Let $\fs_{k+1} \gets k$.} 
     \If(\hfill\tcp*[h]{(B2)}){ 
	$\fD_x(\tilde W_y,j,-v_{s+1},v_t-v_{j^*+1})\leq \fm(\tilde W_y)$ \quad\quad for some $y$\,\,\mbox{}}{
      Let $\fs_{k+1} \gets k$ and mark for deletion every $y$ for which~({\tt B2}) holds.}
     \If(\hfill\tcp*[h]{(B3)}){\:$\fD_x(\tilde W_y,j,-v_{s+1},
     v_t-v_{j^*+1}) \leq (k-t)/2$ \,\,\quad for some $y$\:\mbox{}}{
      Let $\fs_{k+1} \gets k$ and let $y^*_B $ be the minimal index $y$ for which~({\tt B3}) holds.}
}
 Let $k^* \gets \fs_{\sT+2}$ and mark $y^*_B$ for deletion.\;
 }{Let $k^* \gets j^*$.}
\BlankLine

\nl 
\lForEach{index $y\in\cL_{0,n}$ marked for deletion}
{delete $\tilde\fF_y$ from the standard wall representation $(\tilde W_y)$.}

\nl Add a standard wall $W_x^\cJ$ consisting of $\hgt(v_1)- \frac 12$ trivial increments above $x$.

\nl Let $\cK$ be the (unique) interface with the resulting standard wall representation.

\nl Denoting by $(\sX_i)_{i\geq 1}$ the increment sequence of $\cS_x$, set
\[ \cS \gets \begin{cases}\big( \underbrace{X_\trivincr,X_\trivincr,\ldots,X_\trivincr}_{\hgt(v_{j^*+1})-\hgt(v_1)},\sX_{j^*+1},\ldots,\sX_{t-1}, \underbrace{X_\trivincr,X_\trivincr,\ldots,X_\trivincr}_{\hgt(v_{k^*+1})-\hgt(v_t)},\sX_{k^*+1},\ldots \big) &\mbox{ if $t > j^*$},\\
\noalign{\medskip}
 \big( \underbrace{X_\trivincr,X_\trivincr,\ldots,X_\trivincr}_{\hgt(v_{j^*+1})-\hgt(v_1)},\sX_{j^*+1},\ldots \big) &\mbox{ if $t \leq j^*$}\,.\end{cases}\]
\nl Obtain $\Psi_{x,t}(\cI)$ by appending the spine with increment sequence $\cS$ to $\cK$ at $x+(0,0,\hgt(v_1))$.\;

\caption{The map $\Psi_{x,t}$}
\label{alg:le-big-map}
\end{algorithm2e}
\end{figure}

\begin{definition}
For $x\in \cL_{0,n}$, every $t\geq 1$ and every $0 \leq T < H$, define the map $\Psi_{x,t}: \bar\bI_{x,T,H}\to\bar\bI_{x,T,H}$ as specified in Algorithm~\ref{alg:le-big-map} below.
\end{definition}

\begin{remark}
In the exceptional case $T=0$, when we are applying $\Psi_{x,t}$ to interfaces having pillars $\cP_x$ with $\sT=0$ increments (so that they have either zero or one cut-points), we interpret the steps in $\Psi_{x,t}$ in the following way for it to be well-defined. If $\sT=0$ but $v_1$ exists, then recalling that $\sX_1:= \sX_{>0}$, the remainder will be trivialized and the rest of the map is applied as is. If $\sT=0$ and $\cP_x$ has no cut-points, then take an arbitrary face of $\cP_x$ having height $\hgt(\cP_x)$ to stand-in as ``$v_1$," and steps 4--6 will be vacuous. 

Notice, more generally, that if $t>\sT$, step 5 would be vacuous but the map is still well-defined and our results hold by interpreting $\fm(\sX_t)=0$ if $\sX_t=\emptyset$.  
\end{remark}

\subsection{Strategy of the map \texorpdfstring{$\Psi_{x,t}$}{Psi(x,t)}}\label{subsec:reader-guide-map}
We now motivate the different steps in the map $\Psi_{x,t}$ and describe why each one is important to the trade-off described in Section~\ref{subsec:proof-outline-map} between control of the interaction terms and the multiplicity of the map. 
Let us recall in more detail the maps introduced in the prequel~\cite{GL19a} on pillars that reach height $H$, and used there to establish a bound of $O(\log H)$ on the height of the base and exponential tails on the increments above that height. Let $v_{\Tsp}$ be a special cut-point index of the spine, marking the first increment whose height is larger than all other pillars in a ball of radius $CT$ about $x$. For~$t>\Tsp$, the map that proved an exponential tail on the $t$-th increment would simply ``trivialize" $\sX_t$ and $\sX_j$ ($j>t$) in the increment sequence of $\cS_x$ if $\fm(\sX_j)\geq e^{\bar c(j-t)/2}$. In bounding the interactions by~\eqref{eq:g-exponential-decay}, this competed with $e^{\bar c (j-\Tsp)}\leq e^{-\bar c (j-t)}$ (because the horizontal shift of the portion of the spine above $v_{t+1}$ keeps the distance between $\sX_j$ and $\cI \setminus \cP_x$ at least $j-\Tsp$), and summing these terms over $j$ was $O(1)$.

However, for $t<\Tsp$, we have no control on the distance between  the new spine and walls in $\cI\setminus\cP_x$; in fact horizontal shifts of increments $j\in \llb t+1, \Tsp\rrb$ could even \emph{hit} a  neighboring wall of $\cI \setminus \cS_x$ (Figure~\ref{fig:criteria-violations}, left), and the map would not yield a valid interface.  
We thus have to consider the full geometry of these interactions as walls $\tilde W_y$ undergo vertical shifts, and nearby increments  simultaneously undergo horizontal shifts.  

With these difficulties laid out, we discuss the various steps in the definition of $\Psi_{x,t}$ and the different scenarios they are designed to address. The base modification (steps 2--3) here is very similar to that used in~\cite{GL19a}: it marks the nested sequence $\tilde \fW_{[x]}$ for deletion so that the modified spine $\cS$ can later be placed above $W_x^\cJ$ at $x$ to form the new pillar $\cP_x^\cJ$; the additional deletion of $\tilde \fW_{\rho(v_1)}$ and $\tilde \fW_{y^\dagger}$ is to exploit the fact that $\cP_x$ has no cut-points below $v_1$ and ensure that the gain in energy $\fm(\cI;\Psi(\cI))$ is larger than $\hgt(v_1)- \frac 12$.  

\begin{figure}
\begin{tikzpicture}

    \node (fig1) at (-4.4,0) {
	\includegraphics[width=.30\textwidth]{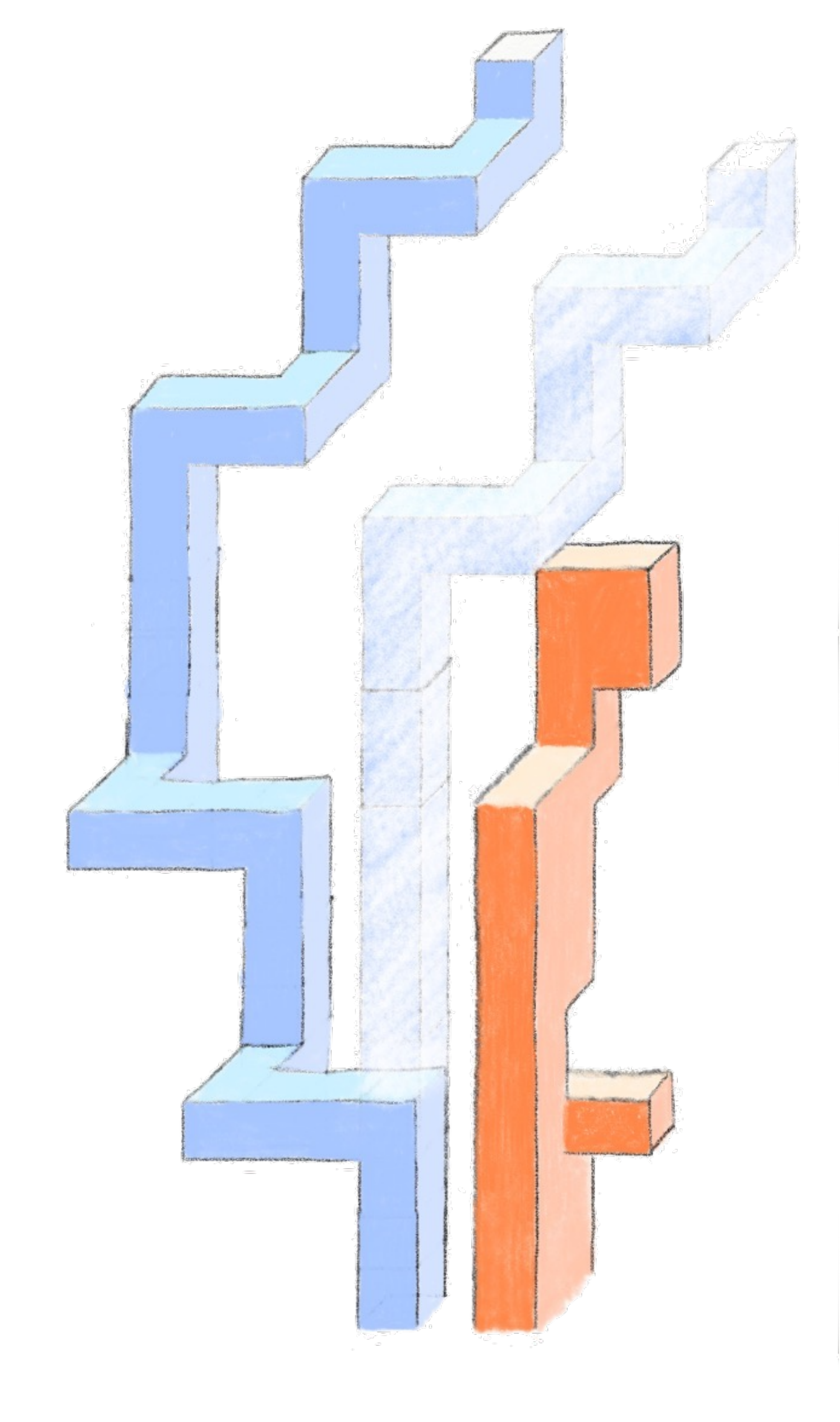}
	};
	
	\draw[->] (-4.7,1.7)--(-4, 1.5);
	
    \node (fig2) at (4.4,.18) {
    \includegraphics[width=.55\textwidth]{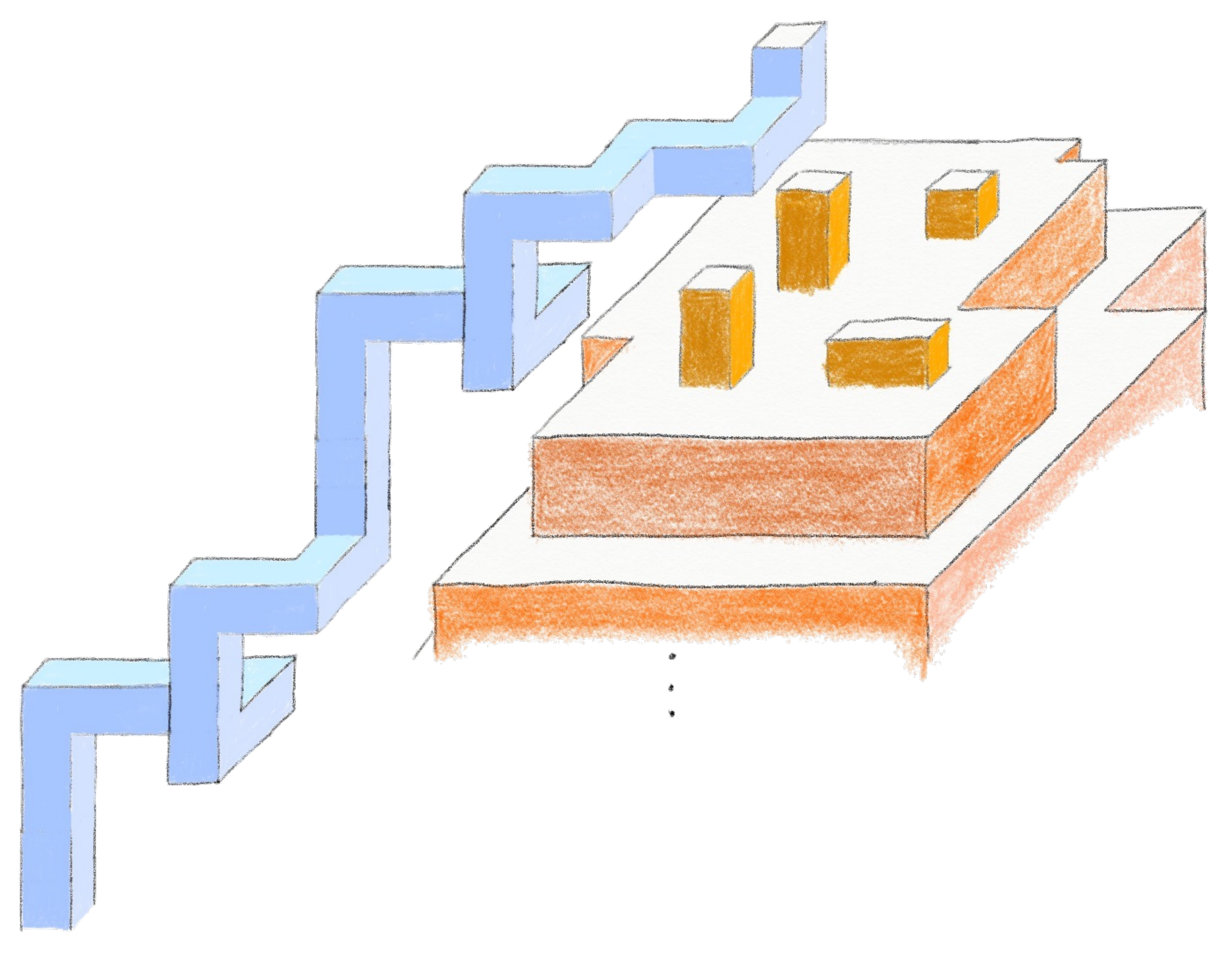}
    };
    \end{tikzpicture}
    \vspace{-0.2in}
    \caption{Left: After trivializing increments $\sX_1,\sX_4$, the shifted spine intersects an adjacent wall (orange). This constitutes an ({\tt{A2}}) violation. Right: the walls in yellow-brown are small (and therefore do not violate ({\tt{A2}}), but still interact strongly with the increment sequence; these will violate criterion ({\tt{A3}}).}
    \label{fig:criteria-violations}
    \vspace{-0.1in}
\end{figure}

The spine modification is substantially more involved. Note that the modifications in ({\tt{A}}) (Step 4) and those in ({\tt{B}}) (Step 5) are essentially the same, with the latter applied at the $t$-th increment so that we can prove the exponential tail on $\fm(\sX_t)$ simultaneously with the exponential tail on $\fm(\sB_x)$. Thus, let us only discuss the steps in the former (the spine modification ({\tt{A}}) at the first increment). 

\begin{enumerate}[({\tt{A}}1)]
    \item aims to control interactions between the horizontal shifts of the increments within the spine itself. 
    Unlike~\cite{GL19a}, where the corresponding modification used a threshold of $\fm(\sX_j) \geq e^{\bar c (j-t)/2}$ which was in a sense the ``most lenient" criterion for trivializing increments, it is important that here we use a ``strictest possible" criterion only allowing a linear growth of the excess areas of increments. 
    \item ensures that after $\Psi_{x,t}$ is applied, any walls that were hit by horizontal shifts of the spine are deleted. This is achieved via a soft threshold comparing the distance between various horizontal shifts of $\sX_i$ to a wall $W_y$ (the relevant quantity in bounding $|\g(f,\cI)- \g(f',\Psi_{x,t}(\cI))|$ via~\eqref{eq:g-exponential-decay}), with the excess area $\fm(W_y)$. Notice that the threshold cannot be done with respect to $\fm(\fW_y)$ instead of $\fm(W_y)$ because one large wall nesting many smaller walls only counts once towards $\fm(\cI;\Psi_{x,t}(\cI))$. (See Figure~\ref{fig:criteria-violations}, left.)   
    \item addresses the additional scenario in which many distinct walls of small excess area $W_1,W_2 , \ldots$ are nested in some $W'$, and the spine draws close to $W_1 , W_2,\ldots$ without violating ({\tt{A2}}). In this situation, only one \emph{highest} nested sequence of walls violating this criterion is deleted in addition to the trivialization of the increments; otherwise we would again be overcounting the nesting wall.  (See Figure~\ref{fig:criteria-violations}, right.)
\end{enumerate}

Finally, it is crucial that the horizontal shifts $\omega_1,\omega_2$ considered in $\fD(\tilde W_y, j, \omega_1, \omega_2)$ in ({\tt{A2}})--({\tt{A3}}) are to be determined in an  algorithmic manner. Namely, we want to ensure that if an increment is \emph{not} trivialized, its horizontal shift in $\Psi_{x,t}(\cI)$ did not violate any of the criteria above; the horizontal shift vector with which this needs to be checked is determined by the last increment to have violated one of the trivialization criteria.    

\subsection{Properties of \texorpdfstring{$\Psi_{x,t}$}{Psi(x,t)}}In this section, we begin by showing that the map $\Psi_{x,t}$ is well-defined on tame interfaces. We then give a decomposition of the interfaces $\cI$ and $\Psi_{x,t}(\cI)$ and prove some simple inequalities on the excess area $\fm(\cI;\cJ)$.

\begin{proposition}[Well-definedness of map $\Psi_{x,t}$]\label{prop:well-defined}
For every $T< H$, and every $\cI\in \bar{\bI}_{x,T,H}$, the interface $\Psi_{x,t}(\cI)$ is well-defined and is an element of $\bI_{x,T,H}$. 
\end{proposition}

\begin{proof}
Firstly, we claim that the standard wall representation obtained after step 7 is admissible. This is because, after $\tilde \fF_{[x]}= \fF_x \cup \bigcup_{f\in \partial_0 x} \fF_f$ is deleted in step 6, the wall $W_x^{\cJ}$  has disjoint projection from all remaining standard walls. We next must ensure that when adding the modified $\cS$ in step 10 to $\cK$, it does not intersect any part of the pre-existing interface, or $\partial \Lambda_n$. 

For this, notice that if $t> j^*$, then $\cS$ is exactly 
\begin{align*}
    \bigcup_{1\leq j \leq j^*} \theta_{\rho(x)}\Theta_\trivincr \sX_j \cup \bigcup_{j^* < i < t} \theta_{\rho(x- v_{j^*+1})}\sX_i \cup \bigcup_{t\leq k\leq k^*} \theta_{\rho(x+v_t-v_{j^*+1})} \Theta_\trivincr \sX_k \cup \bigcup_{i> k^*+1} \theta_{\rho(x-v_{k^*+1}+v_t - v_{j^*+1})}\theta \sX_i
\end{align*}
and if $t\leq j^*$, then $\cS$ is exactly 
\begin{align*}
    \bigcup_{1\leq j \leq j^*} \theta_{\rho(x)}\Theta_\trivincr \sX_j \cup \bigcup_{i>j^*} \theta_{\rho(x- v_{j^*+1})}\sX_i\,.
\end{align*}
Now, make the following observation regarding the sequence of shifts observed while running~$\Psi_{x,t}$. 
\begin{observation}\label{obs:fs_i}
The sequence $(\fs_i)$ has $\fs_{i+1}\neq \fs_i$ if and only if one of criteria ({\tt{A1}}), ({\tt{A2}}), ({\tt{A3}}) or ({\tt{B1}}), ({\tt{B2}}), ({\tt{B3}}) were attained for $i$, in which case $\fs_{i+1} = i$. Consequently, $\fs_i = j^*$ for every $j^* <i<t$, and $\fs_i = k^*$ for every $i>k^*$. 
\end{observation}

\noindent Thus, whether $t>j^*$ or $t\leq j^*$, the shifts and trivializations comprising $\cS$ are considered in the criteria
\begin{align*}
    \fD_x(\tilde W_y,j,-v_{\fs_j+1},0) \leq (j-1)/2\qquad  \mbox{and}\qquad \fD_x(\tilde W_y,j,-v_{\fs_j+1},
     v_t-v_{j^*+1}) \leq (k-t)/2\,.
\end{align*}
Also, by definition, every face $f$ of $\cK\setminus (W_x^\cJ\cup \cL_0)$ is in $\Theta_\udarrow^{\cI\setminus \cS_x} \tilde W_y$ for some $y$. As such, if $\cS$ intersects some pre-existing part of $\cK\setminus W_x^\cJ$, there would have been some pair $y,j$ such that the above distance $\fD_x$ would be zero; in that case, that $y$ would have been marked for deletion, and the corresponding face in $\cK$ would be in $\cL_0$ yielding a contradiction. 

In order to see that the addition of $\cS$ does not hit $\partial \Lambda_n$, we use the definition of tameness. In particular, the  horizontal displacement of the spine $\cS$ from $x$ is always bounded above by
\begin{align*}
   \diam(\sB_x) +\tfrac 14 \fm(\cS_x) <d(x,\partial \Lambda_n)\,,
\end{align*}
where the inequality was by definition of $\cI \in \bar{\bI}_{x,T,H}$ and $\Psi_{x,t}(\cI)$ is a valid interface. 

Finally, we observe that the resulting interface is in $\bI_{x,T,H}$. Notice that the resulting pillar of $x$ in $\Psi_{x,t}(\cI)$ consists of $W_x^\cJ\cup \cS$; on the one hand, this has at least $T$ increments since trivializing increments only increases the total number of increments and on the other hand, it has the same height as $\cP_x$ by construction.  
\end{proof}

\subsubsection{Decomposition of the interfaces}\label{subsec:decomposition-interfaces}
Fix any interface $\cI \in \bar{\bI}_{x,T,H}$ and for ease of notation, let $\cJ = \Psi_{x,t}(\cI)$. We begin by partitioning the faces of $\cI$ and $\cJ$ into their constituent parts as dictated by the map $\Psi_{x,t}$. This partioning will govern the pairings of $\g(f,\cI)$ with $\g(f',\cJ)$ when applying~\eqref{eq:g-exponential-decay}. 

\begin{figure}
    \centering
\begin{tikzpicture}
    \node (glossary-I) at (90pt,-28pt) {\includegraphics[width=0.4\textwidth]{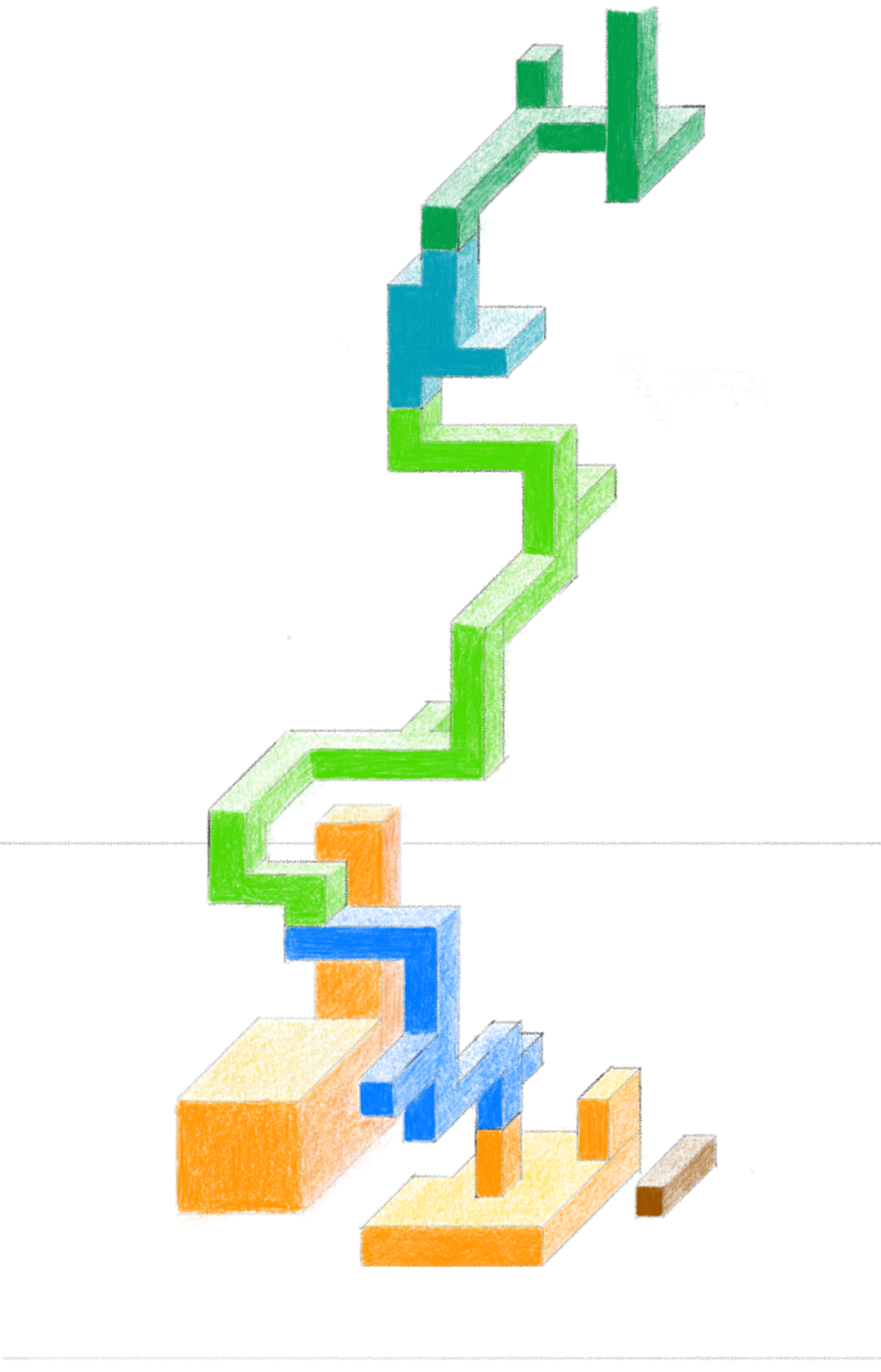}};
    \draw[rounded corners, color=gray] (-5pt,-25pt) rectangle (40pt,115pt); 
    \pillarbox{0}{100pt}{$\bX_2^\cI$}{pillDarkGreenFG}{pillDarkGreenBG};
    \pillarbox{0}{80pt}{$\bX_B^\cI$}{pillAzureFG}{pillAzureBG};
    \pillarbox{0}{60pt}{$\bX_1^\cI$}{pillGreenFG}{pillGreenBG};
    \pillarbox{0}{40pt}{$\bX_A^\cI$}{pillBlueFG}{pillBlueBG};
    \wallface{0}{20pt}{$\bW$}{pillOrangeFG}{pillOrangeBG};
    \ceilface{0}{0pt}{$\bC_1$}{pillOrangeFG}{pillOrangeBG};
    \pillarbox{0}{-20pt}{$\bC_2$}{pillBrownFG}{pillBrownBG};

    \begin{scope}[shift={(250pt,0)}]
    \node (glossary-J) at (90pt,-23.5pt) {\includegraphics[width=0.4\textwidth]{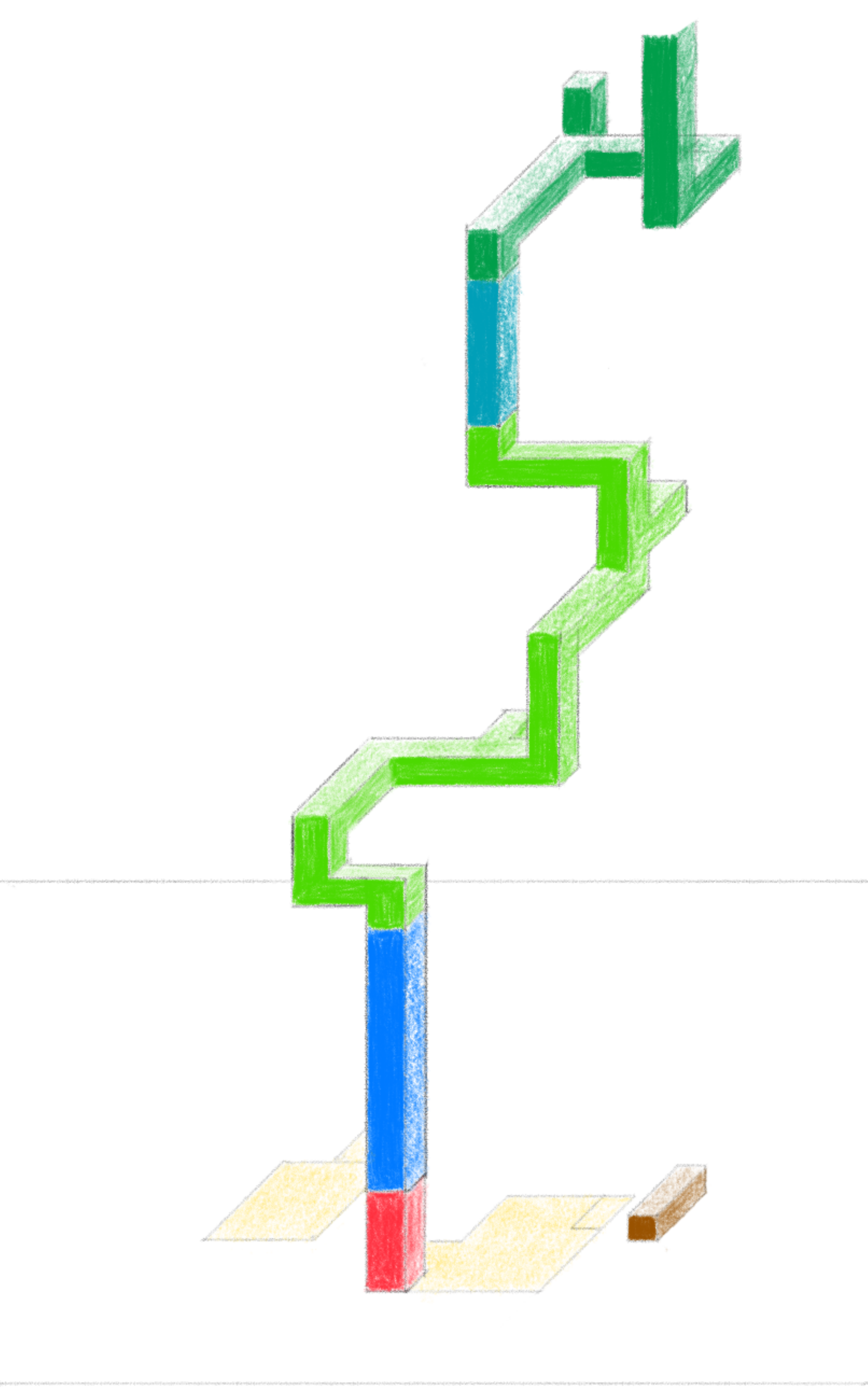}};
    \draw[rounded corners, color=gray] (-5pt,-25pt) rectangle (40pt,115pt); 
    \pillarbox{0}{100pt}{$\bX_2^\cJ$}{pillDarkGreenFG}{pillDarkGreenBG};
    \pillarbox{0}{80pt}{$\bX_B^\cJ$}{pillAzureFG}{pillAzureBG};
    \pillarbox{0}{60pt}{$\bX_1^\cJ$}{pillGreenFG}{pillGreenBG};
    \pillarbox{0}{40pt}{$\bX_A^\cJ$}{pillBlueFG}{pillBlueBG};
    \wallface{0}{20pt}{$W_x^\cJ$}{pillRedFG}{pillRedBG};
    \ceilface{0}{0}{$\theta_\udarrow \bC_1$}{pillOrangeFG}{pillOrangeBG};
    \pillarbox{0}{-20pt}{$\theta_{\udarrow}\bC_2$}{pillBrownFG}{pillBrownBG};
    


	\draw[densely dotted] (-5.05-1.14,-4.4-.93)--(-4.75-1.14,-4.4-.93)--(-4.6-1.14,-4.25-.93)--(-4.9-1.14,-4.25-.93)--(-5.05-1.14,-4.4-.93);
  \node at (-4.825-1.14,-4.325-.93) {\fontsize{3.5}{4}\selectfont $x$};

    	\draw[densely dotted] (3.9-1.23,-4.4-.93)--(4.2-1.23,-4.4-.93)--(4.35-1.23,-4.25-.93)--(4.05-1.23,-4.25-.93)--(3.9-1.23,-4.4-.93);
  \node at (4.125-1.23,-4.325-.93) {\fontsize{3.5}{4}\selectfont $x$};

    \end{scope}
\end{tikzpicture}
\vspace{-0.2in}
\caption{The decomposition of the interfaces $\cI$ (left) and $\cJ$ (right) into their constituent parts.}
    \label{fig:glossary}
\end{figure}

Let $\bY\subset \cL_0$ be the set of indices of walls in $\cI \setminus \cS_x$ that were marked for deletion. Let ${\bD}\subset \cL_0$ be the indices of walls that were deleted (i.e., walls in $\bigcup\{ \tilde \fF_y : y\in \bY\}$). 
Split up the faces of $\cI$ as follows:
{\renewcommand{\arraystretch}{1.3}
\[
\begin{tabular}{m{0.05\textwidth}m{0.3\textwidth}m{0.5\textwidth}}
\toprule
\midrule
 $\bX_A^\cI$ & $\bigcup_{1\leq j \leq j^*} \cF(\sX_j)$ & Increments between $v_1$ and $v_{j^*+1}$  \\ 
 $\bX_{1}^{\cI}$ & $\bigcup_{j^* + 1\leq j \leq t-1}\cF(\sX_j)$ & If $t> j^*$, increments between $v_{j^*+1}$ and $v_{t}$ \\
 $\bX_{B}^{\cI}$ & $\bigcup_{t \leq k \leq k^*} \cF(\sX_k)$ & If $t> j^*$, increments between $v_{t}$ and $v_{k^* + 1}$ \\ 
 $\bX_{2}^\cI$ & $\bigcup_{k \geq k^*} \cF(\sX_k)$ & Increments above $v_{k^* + 1}$ \\
 $\bW$ & $\bigcup_{z\in \bD} \tilde W_z$ & All walls that were deleted \\
 $\bB$ & $\cI\setminus(\bX_A^\cI\cup\bX_1^\cI\cup\bX_B^\cI\cup\bX_2^\cI\cup\bW)$
 & The remaining set of faces in $\cI$ \\
 \bottomrule
 \end{tabular}
\]
}
where $\bB$ splits further into
{\renewcommand{\arraystretch}{1.25}
\[
\begin{tabular}{m{0.05\textwidth}m{0.3\textwidth}m{0.5\textwidth}}
\toprule
\midrule
    $\bC_1$ & 
    $\bigcup_{y\in \bY} \lceil \tilde W_y \rceil$
& Ceilings of walls marked for deletion \\
 $\bC_2$& $\bigcup_{z\in \bD\setminus \bY} \lceil \tilde W_z\rceil \cup \bigcup_{z\notin \bD} \tilde W_z \cup \lceil \tilde W_z \rceil $& 
Ceilings of walls that were not marked, along with all non-deleted walls and their ceilings\\
$\bF$ & $(\bB \cap \cL_0) \setminus \bigcup_{z} \tilde W_z \cup \lceil \tilde W_z \rceil$
& Faces in $\bB$ that are only floors of walls \\
\bottomrule
 \end{tabular}
\]
}
\noindent See Figure~\ref{fig:glossary} (left) for a depiction of this splitting. 

We next partition the faces of $\cJ $. Let us first introduce a few pieces of notation. Denote by $\cP_x^\cJ$ the pillar of $x$ in $\cJ = \Psi_{x,t}(\cI)$; observe that by construction, in $\cJ$ the spine $\cS_x^\cJ$ is all of $\cP_x^\cJ$. For a given $\cI,x,t$ define the shift map $\lrvec \theta$ as the horizontal shift on the increments $\sX_{j^* + 1},\ldots, \sX_{t-1}$ and $\sX_{k^*+1},\ldots$ from $\Psi_{x,t}$. Namely, for $f\in \bX_1^\cI\cup \bX_2^\cI$, let
\begin{align}\label{eq-theta-lr-def}
        \lrvec \theta f =     
        \begin{cases}
        \theta_{\rho(x- v_{j^*+1})} f & \qquad \mbox{if $f\in \bX_{1}^{\cI}$ or if $f\in\bX_2^\cI$ and $t \leq j^*$}\,,\\
        \theta_{\rho(x- v_{k^* + 1} +v_{t} - v_{j^*+1})} f & \qquad \mbox{if $f\in\bX_2^\cI$ and $t>j^*$}\,.
  \end{cases}
\end{align}
We can also define the map $\theta_{\udarrow}$ on faces in $\bB$, that vertically shifts faces of $\bB$ to obtain corresponding faces of $\cJ$ as dictated by the bijection Lemma~\ref{lem:standard-wall-representation} and the removal of the walls in $\bW$. With these, let:  

{\renewcommand{\arraystretch}{1.3}
\[
\begin{tabular}{m{0.05\textwidth}m{0.33\textwidth}m{0.51\textwidth}}
\toprule
\midrule
 $\bX_A^\cJ$ & $\Theta_{\trivincr} \bX_A^\cI$ & Increments of $\cJ$ between $\hgt(v_1)$ and $\hgt(v_{j^*+1})$  \\ 
 $\bX_{1}^{\cJ}$ & $\lrvec \theta \bX_{1}^{\cI}$ & Increments of $\cJ$ between $\hgt(v_{j^*+1})$ and $\hgt(v_{t})$ \\
 $\bX_{B}^{\cJ}$ & $\theta_{\rho(v_t-v_{j^*+1})} \Theta_{\trivincr} \bX_B^\cI$ & Increments of $\cJ$ between $\hgt(v_{t})$ and $\hgt(v_{k^* + 1})$ \\ 
 $\bX_{2}^\cJ$ & $\lrvec \theta \bX_2^\cI$ & Increments of $\cJ$ above $\hgt(v_{k^* + 1})$ \\
 $W_x^\cJ$ & $\cF(x+[ -\frac 12, \frac 12]^2\times [0,\hgt(v_1)-\frac 12])$ & New faces added in step 9 of $\Psi_{x,t}$ \\
 $\theta_{\udarrow} \bB$ & $\theta_{\udarrow}\bC_1 \cup \theta_\udarrow \bC_2 \cup \bF$ & Vertical translations of $\bB$ due to the deletion of walls~$\bW$ \\
 $\bH$ & $\bigcup \{f\in \cJ\setminus \cP_x^\cJ: \rho(f) \in \cF(\rho(\bW))\}$ & Faces added to ``fill in" the rest of the interface  \\
 \bottomrule
 \end{tabular}
\]
}
Refer to Figure~\ref{fig:glossary} (right) for a depiction of this splitting.

\subsubsection{The excess area of the map}
With the above decomposition in hand, note that the change in energy between $\cI$ and $\cJ = \Psi_{x,t}(\cI)$ is given by
\begin{align}\label{eq:m(I;J)}
\fm(\cI;\cJ)= \sum_{z\in \bD} \fm(\tilde W_z) + \sum_{j=1}^{j^*} \fm(\sX_j) + \sum_{k=t}^{k^*} \fm(\sX_k)\one\{t>j^*\} - |W_x^\cJ|\,.
\end{align}

The following inequalities regarding $\fm(\cI; \cJ)$ will be used repeatedly. 

\begin{claim}\label{clm:m(W)}For every $\cI \in \bar{\bI}_{x,T,H}$, denoting by $\cJ = \Psi_{x,t}(\cI)$, we have
\begin{align}\label{eq:|W_x^J|}
|W_x^\cJ| \leq \frac23 \fm(\tilde \fW_{v_1} \cup \tilde \fW_{y^\dagger})
\end{align}
and in particular
\begin{equation}\label{eq:m(I;J)-WxJ} |W_x^\cJ|\leq 2 \fm(\cI;\cJ)\,,\quad\mbox{and}\quad \fm(\bW) \leq 3 \fm(\cI;\cJ)\,.
\end{equation}
We also have 
\begin{align}\label{eq:m(I;J)-j-k}
    j^* - 1 \leq 6 \fm(\cI;\cJ) \qquad \mbox{and}\qquad  k^* - t \leq 6 \fm(\cI; \cJ)\,.
\end{align}
\end{claim}

Before getting to the proof of Claim~\ref{clm:m(W)}, we need some simple geometric observations. 
Recall that for a face-set $W$, its height is given by $\hgt(W) = \max\{x_3 : (x_1,x_2,x_3)\in W\}$. Observe that $\hgt(W) = \hgt( W \cup \lceil W \rceil)$. 
\begin{fact}\label{fact:useful} For every $j\geq 1$ and every index $y \in \cL_{0,n}$,
\[ j-1 \leq \min_{W\in \Theta_\udarrow^{\cI\setminus \cS_x}  \tilde W_y} \big[d_\udarrow( W \cup \lceil  W \rceil, \sX_j) + \hgt(W)\big] \,,\]
where $d_\udarrow(A,B) = \min\{|x_3-y_3| : x\in A, y\in B\}$. In particular, for every $\omega_1$ and $\omega_2$,
\[ j-1 \leq \fD_x(\tilde W_y ,j,\omega_1,\omega_2)+ \fm(\tilde \fF_y)\,.\]
\end{fact}
\begin{proof}
The first claim follows from the triangle inequality and  $\min \{x_3: (x_1,x_2,x_3) \in \sX_j\}\geq j-1$. 
Then, for every $\omega_1,\omega_2$, 
\begin{align*} \min_{W \in \Theta_\udarrow^{\cI\setminus \cS_x}  W_y} d_\udarrow(W\cup \lceil W \rceil,\sX_j) 
&= \min_{W \in \Theta_\udarrow^{\cI\setminus \cS_x}  \tilde W_y} d_\udarrow(W\cup \lceil W \rceil, \sX_j \cup \theta_{\rho(x+\omega_1+\omega_2)}\sX_j \cup \theta_{\rho(x+\omega_2)}\Theta_\trivincr \sX_j) \\
&\leq \min_{W \in \Theta_\udarrow^{\cI\setminus \cS_x}  \tilde W_y} d(W\cup \lceil W \rceil, \sX_j \cup \theta_{\rho(x+\omega_1+\omega_2)}\sX_j \cup \theta_{\rho(x+\omega_2)}\Theta_\trivincr \sX_j) \\&= \fD_x( \tilde W_y, j, \omega_1,\omega_2)\,. 
\end{align*}
 The proof concludes from the observation that $\max_{W\in \Theta_\udarrow^{\cI\setminus \cS_x} \tilde W_y} \hgt(W) \leq \fm(\tilde \fF_y)$.
\end{proof}

\begin{claim}\label{clm:cut-point-property}
Let $\{W_z\}_{z\in \cL_{0,n}}$ be the collection of walls corresponding to some interface $\cI$. There exists some $z_0\in\cL_{0,n}$ such that every cut-point $v$ of $\cI$ must belong to $W_{z_0}$ (its four vertical bounding faces are in $W_{z_0}$). 

Consequently, for an interface $\cI$ with pillar $\cP_x$, the face set of $\cS_x$ consists of exactly one wall, together with at most one ceiling face projecting into $\rho(v_1)$. 
\end{claim}
\begin{proof}
Suppose by way of contradiction that the interface $\cI$ has two cut-points $w_1, w_2$ with $0<\hgt(w_1) <\hgt(w_2)$ such that the walls containing the bounding faces of $w_1$ and $w_2$, namely $W_{\rho(w_1)}$ and $W_{\rho(w_2)}$ are distinct. If both $W_{\rho(w_1)}, W_{\rho(w_2)}$ are standard, then $W_{\rho(w_i)}$ intersect each of $\cL_{\frac 12},\ldots,\cL_{\hgt(w_i)}$ in at least one cell (i.e., if  $\cI_{W_{\rho(w_i)}}$ is the interface whose only wall is $W_{\rho(w_i)}$, then $\sigma(\cI_{W_{\rho(w_i)}})$ intersects every height between $\frac 12$ and $\hgt(w_i)$ in at least one cell). As a consequence,  no height between $\frac 12$ and $\hgt(w_1)$ can be a cut-height of $\cI$. If at least one of $W_{\rho(w_i)}$ is not standard, then it is identified with a ceiling $\cC_i \subset \lceil W \rceil$ for some other wall $W \Supset W_{\rho(w_i)}$. Call the tallest of those ceilings $\cC$ with height $\hgt(\cC)$. By definition of ceiling faces, we have the following observation. 

\begin{observation}Every cell $w$ sharing a projection with a face $f\in \cC$ and having $\hgt(w)\leq \hgt(\cC)$ is in~$\sigma(\cI)$.
\end{observation}

At the same time, since $W_{\rho(w_i)} \Subset W$ and $W_{\rho(w_i)}\neq W$, the ceiling $\cC$ has at least eight faces. Therefore, there are no cut-heights below $\hgt(\cC)$, yielding a contradiction if $\hgt(\cC)\geq \hgt(w_1)$. If $\hgt(\cC)<\hgt(w_1)$, $W_{\rho(w_1)} \cap \cL_h \neq \emptyset$ and $W_{\rho(w_2)} \cap \cL_h\neq\emptyset$ for every $\hgt(\cC)<h\leq \hgt(w_1)$, also yielding a contradiction. 

To see the conclusion for the spine of a pillar $\cP_x$, take $\cI_{\cP}$ to be the interface which is at $\cL_0$ except for the faces of $\cP_x$. Applying the first part of the claim to $\cI_{\cP}$, we see that all cut-points of $\cS_x$ are in the same wall $W$, and by definitions of cut-points and the observation above, all other faces of $\cS_x$ must in $W$, except possibly one ceiling face projecting into $\rho(v_1)$. 
\end{proof}

\begin{corollary}\label{cor:no-cut-points-after-step-3}
The walls whose standardizations are $\tilde \fW_{v_1}\cup \tilde \fW_{y^\dagger}$ intersect each of $\cL_{\frac 12},\ldots,\cL_{\hgt(v_1)-1}$ in at least six faces (i.e., the corresponding interface intersects every such height in at least two cells).
\end{corollary}
\begin{proof}
By definition, the walls $\cO_{v_1}$ (defined s.t.\ $\theta_{\textsc{st}}\cO_{v_1} =\tilde {\mathfrak W}_{v_1}$) intersected every height between $h^\dagger+1$ and $\hgt(v_1)-1$ in at least two cells. Now consider heights between $0$ and $h^\dagger$. 

Since $y^\dagger$ and $v_1$ are both in $\cP_x$, by Observation~\ref{obs:pillar-nested-sequence-of-walls}, there must exist a wall $W$ such that $y^\dagger, v_1$ are both interior to $W$. Since $\tilde W_{\rho(v_1)} \neq \tilde W_{y^\dagger}$, there are inner-most ceilings of $\cC_{y^\dagger}$ and $\cC_{v_1}$ in $\lceil W\rceil$ nesting those respective walls. As such, by the observation above, every height between below $\hgt(\cC_{v_1}) \vee \hgt(\cC_{y^\dagger})$ is intersected by at least eight cells. Finally, since the walls whose standardizations are $\tilde W_{v_1}$ and $\tilde W_{y^\dagger}$ attain height $h^\dagger$, every height between $\hgt(\cC_{v_1}) \vee \hgt(\cC_{y^\dagger})$ and $\hgt(h^\dagger)$ is intersected by at least one cell by each of those walls. 
\end{proof}

\begin{proof}[\textbf{\emph{Proof of Claim~\ref{clm:m(W)}}}] By Corollary~\ref{cor:no-cut-points-after-step-3}, there are no cut-points in $\tilde\fW_{v_1}\cup\tilde\fW_{y^\dagger}$, and therefore
\[ |W_x^\cJ| = 4(\hgt(v_1)-\tfrac 12) \leq \tfrac23\fm(\tilde\fW_{v_1}\cup\tilde\fW_{y^\dagger})\,.\]
Hence, $|W_x^\cJ|\leq \frac23 \fm(\bW)$, and~\eqref{eq:m(I;J)-WxJ}
 follows from the fact $\fm(\cI;\cJ) \geq \fm(\bW) - |W_x^\cJ|$ by~\eqref{eq:m(I;J)}.
 
 Let us now turn to proving the comparisons with $j^*-1$ and $k^* - t$, namely~\eqref{eq:m(I;J)-j-k}. Suppose $j^* >1$ as otherwise the inequality is trivial. 
 Since $\sX_{j^*}$ was deleted, it was due to one of criteria ({\tt A1}) or ({\tt A2}) or ({\tt A3}):
\begin{enumerate}[({A}1)]
\item In this case, $\fm(\sX_{j^*}) \geq j^*-1$.
\item In this case, $\fD_x(\tilde W_y,j^*,-v_{\fs_j+1},0)\leq \fm(\tilde W_y) $ for some $y\in\bY$.
By Fact~\ref{fact:useful} applied to $j^*$,
\[  j^*-1 \leq \fD_x(\tilde W_y, j^*, -v_{\fs_j+1},0) +\fm(\tilde \fF_y) \leq 2\fm(\tilde \fF_y) \,,\]
\item In this case, $\fD_x(\tilde W_y,j^*,-v_{\fs_j+1},0) \leq (j-1)/2$ for $y  = y^*_A$.
By Fact~\ref{fact:useful} applied to $j^*$ and $y^*_A$,
\[ j^*-1 \leq \fD_x(\tilde W_{y^*_A}, j^*, -v_{\fs_j+1},0) + \fm(\tilde \fF_{y^*_A}) \leq (j^*-1)/2 + \fm(\tilde \fF_{y^*_A}) \qquad \mbox{so} \qquad j^*-1 \leq 2 \fm(\tilde \fF_{y^*_A})\,.\]
\end{enumerate}
In any of these above cases, we have $j^*-1 \leq 2\fm(\bW)\vee\fm(\sX_{j*})\leq 6 \fm(\cI; \cJ)$ by~\eqref{eq:m(I;J)-WxJ}.

If $j^* \geq t$, then we are done. Otherwise, since $\sX_{k^*}$ was deleted, it was due to either ({\tt B1}) or ({\tt B2}) or ({\tt B3}):
\begin{enumerate}[({B}1)]
\item In this case, $\fm(\sX_{k^*}) \geq k^* - t$.
\item In this case, $\fD_x(\tilde W_y,k^*,-v_{\fs_k+1},v_t-v_{j^*+1})\leq \fm(\tilde W_y) $ for some $y\in\bY$.
By Fact~\ref{fact:useful},
\[  k^*-t \leq k^*-1 \leq \fD_x(\tilde W_y, k^*, -v_{\fs_k+1},v_t-v_{j^*+1}) +\fm(\tilde \fF_y) \leq 2\fm(\tilde \fF_y) \,.\]
\item In this case, $\fD_x(\tilde W_y,k^*,-v_{\fs_k+1},v_t-v_{j^*+1}) \leq (k-t)/2$ for $y = y_B^*$.
By Fact~\ref{fact:useful},
\[ k^*-t \leq \fD_x(\tilde W_{y^*_B}, k^*, -v_{\fs_k+1},v_t-v_{j^*+1}) + \fm(\tilde \fF_{y^*_B}) \leq (k^*-t)/2 + \fm(\tilde \fF_{y^*_B})\,\qquad \mbox{so}\qquad k^*-t \leq 2 \fm(\tilde \fF_{y^*_B})\,.\]
\end{enumerate}
In any of these cases, we have $k^*-t \leq 6\fm(\cI;\cJ)$ by~\eqref{eq:m(I;J)-WxJ}.
\end{proof}

\subsection{Proof part 1: interface weights}\label{subsec:interface-weights-psi}
In this section, we show that the map $\Psi_{x,t}$ amplifies the weights of interfaces by something exponential in the excess area $\fm(\cI;\Psi_{x,t}(\cI))$. 

As in the preceding works~\cite{Dobrushin72a,GL19a}, the difficulty here is ensuring that the cumulative effect of the perturbative terms $\g$ in Theorem~\ref{thm:cluster-expansion} (capturing interactions between different parts of the interface through sub-critical droplets) is comparable to $\fm(\cI;\Psi_{x,t}(\cI))$. As described in Section~\ref{subsec:reader-guide-map}, this is particularly complicated here as we cannot reduce the interactions to only their horizontal, or only their vertical parts.

\begin{proposition}\label{prop:partition-function-contribution}
There exists $C>0$  and $\beta_0$ such that for every $\beta>\beta_0$ the following holds. For every $x\in \cL_{0,n}$ and $0\leq T < H$, for every $t$ and every $\cI \in \bar{\bI}_{x,T,H}$,
\begin{align*}
    \Big|\log \frac{\mu_n^{\mp}(\cI)}{\mu_n^{\mp}(\Psi_{x,t}(\cI))} + \beta \fm(\cI;\Psi_{x,t}(\cI))\Big| \leq C\fm(\cI;\Psi_{x,t}(\cI))\,.
\end{align*}
\end{proposition}

We first prove a series of preliminary estimates to which we will reduce Proposition~\ref{prop:partition-function-contribution} by pairing faces together according to the decomposition of $\cI$ and $\cJ$ from \S\ref{subsec:decomposition-interfaces}.  

\begin{claim}\label{clm:W-and-H-vs-everything}There exists $\bar C$ such that for every $\cI \in \bar{\bI}_{x,T,H}$, 
\[ 
\sum_{f\in \cF(\Z^3)} \sum_{g\in \bW\cup\bH}   e^{-\bar c d(f,g)} \leq \bar C \fm(\bW)\,.
 \]
\end{claim}
\begin{proof}By summing the exponential tail, there exists $C>0$ such that 
\begin{align*}
\sum_{g\in \bW\cup\bH} \sum_{f\in \cF(\Z^3)} e^{-\bar c d(f,g)} &\leq \sum_{z\in {\bD}}\bigg(\sum_{g\in \tilde W_z} C + \sum_{\substack{g'\in\cJ\setminus \cS_x^\cJ \\ \rho(g') \in \cF(\rho(\tilde W_z))}} C\bigg) \leq \sum_{z\in {\bD}} \left( C|\tilde W_z| + C|\cF(\rho(\tilde W_z))|\right)\,,
\end{align*}
which by~\eqref{eq:wall-excess-area-inequalities} is at most $3C\sum_{z\in \bD} \fm(\tilde W_z) \leq 3C \fm(\bW)$.
\end{proof}
\begin{claim}\label{clm:X1-and-X3-vs-everything}There exists $\bar C$ such that for every $\cI \in \bar{\bI}_{x,T,H}$,
\[ 
\sum_{f\in \cF(\Z^3)} \sum_{\iota\in\{A,B\}}   \sum_{g\in \bX_\iota^\cI \cup \bX_\iota^\cJ}  e^{-\bar c d(f,g)} \leq \bar C \fm(\cI;\cJ)\,,
\]
\end{claim}
\begin{proof}Summing over all $f\in \cF(\Z^3)$, there exists $C>0$ such that
\begin{align*} \sum_{\iota\in\{A,B\}} \sum_{g\in \bX_\iota^\cI \cup \bX_\iota^\cJ} \sum_{f\in \cF(\Z^3)} e^{-\bar c d(f,g)} 
 &\leq C \bigg(\;\sum_{1\leq j \leq j^*} [\fm(\sX_j) + 8 (\hgt(v_{j+1})-\hgt(v_j)+1)] \\ & \quad \qquad +\!\!\sum_{t\leq k \leq k^*} [\fm(\sX_k) + 8 (\hgt(v_{k+1})-\hgt(v_k)+1)]\bigg)\,,
\end{align*}
where in the first inequality, the factor of 8 accounts for 4 faces from the $\sX_i$ (Eq.~\eqref{eq:increment-excess-area}) and 4 from $\Theta_\trivincr \sX_i$ for each height. 
The telescopic sums give $8(\hgt(v_{j^*+1})-\hgt(v_1))$, and $\hgt(v_{j^*+1})-\hgt(v_1) \leq j^* + \frac 12 \sum_{j=1}^{j^*} \fm(\sX_j)
$ since for each height in in $\hgt(v_1),\ldots,\hgt(v_{j^*+1})$ that wasn't a cutpoint, we added an excess area of at least~2. Accounting for an extra additive $j^*$, as well as the similar telescoping for $t\leq k \leq k^*$, we see this is at most 
\begin{align*}
 16C (j^*-1)  + 16 C (k^*-t) + 5 C \fm(\cI;\cJ) + 32C\,,
\end{align*}
which, by~\eqref{eq:m(I;J)-j-k} of Claim~\ref{clm:m(W)}, is in turn at most $\bar C\fm(\cI;\cJ)$ for some other constant $\bar C>0$.
\end{proof}
\begin{claim}\label{clm:Wx^J}
There exists $\bar C$ such that for every $\cI \in \bar{\bI}_{x,T,H}$
\[ 
\sum_{f\in \cF(\Z^3)} \sum_{g\in W_x^\cJ }  e^{-\bar c d(f,g)} \leq \bar C \fm(\bW) \,.
\]
\end{claim}
\begin{proof}
Summing over all $f\in \cF(\Z^3)$, there exists $C>0$ such that
\[ 
\sum_{g\in W_x^\cJ } \sum_{f\in \cF(\Z^3)} e^{-\bar c d(f,g)} \leq C |W_x^\cJ|\,,
\]
which is at most $\frac{2}{3}C\fm(\bW)$ by~\eqref{eq:|W_x^J|}.
\end{proof}

The following series of lemmas bounds the interactions (through the subcritical droplets, via $|\g(f,\cI) - \g(f', \cJ)|$) between different subsets of $\cI$ and $\cJ$. The first of these concerns interactions between horizontal shifts of $\bX_1^\cI$ and $\bX_2^\cJ$ with $\bF$. 

\begin{lemma}\label{lem:F-vs-X2-and-X4}There exists $\bar C$ such that for every $\cI \in \bar{\bI}_{x,T,H}$, 
\[ 
\sum_{f\in \bX_1^\cI\cup \bX_2^\cI} \sum_{g\in \cL_0} e^{-\bar c d(f,g)} \leq \bar C\,,
\]
\end{lemma}

\begin{proof}Noticing that for every $i$, $d(\sX_i, \cL_0)\geq i$, there exists $C>0$ such that 
\begin{align*}
 \sum_{g\in \bX_1^\cI\cup \bX_2^\cI} \sum_{f\in \cL_0} e^{-\bar c d(f,g)} &  \leq C \sum_{j^* < i < t} |\cF(\sX_i)|  e^{-\bar c i} + C\sum_{i>k^*} |\cF(\sX_i)|  e^{-\bar c i}  \\ 
& \leq C \sum_{i \in (j^*,t) \cup (k^*, \infty)} [4+3\fm(\sX_i)]e^{ - \bar c i} \,.
\end{align*}
In turn, using the fact that $\fm(\sX_i)\leq i-1$  for $i\in (j^*, t)$ and $i >k^*$ (criterion ({\tt A1})), this is at most
\[\sum_{i > j^*} C [1+3i] e^{-\bar c i} \leq \bar C e^{-\bar c j^*}\,.\qedhere
\]
\end{proof}

The next lemma helps control horizontal interactions induced by vertical shifts of walls and ceilings in $\bB$. 
\begin{lemma}\label{lem:B-vs-B}There exists $\bar C$ such that for every $\cI \in \bar{\bI}_{x,T,H}$, 
\[ \sum_{f\in \bB} \sum_{u\in\rho(\bigcup_{z\in \bD}\tilde W_z)} e^{-\bar c d(\rho(f),u)} \leq \bar C\fm(\bW)\,.\]
\end{lemma}
\begin{proof}By definition of $N_\rho(u)$, we have 
\[ \sum_{f\in \bB} \sum_{u\in\rho(\bigcup_{z\in \bD}\tilde W_z)} e^{-\bar c d(\rho(f),u)} = \sum_{u'\in\rho(\bigcup_{z\in \bD}\tilde W_z)^c}\sum_{u\in\rho(\bigcup_{z\in \bD}\tilde W_z)} N_\rho(u') e^{-\bar c d(u,u')}\,.
\]
Since $\bD$ is closed under closeness of walls, for every such $u,u'$, we have $ N_\rho(u') \leq |u-u'|^2+1$. Thus there exists a $C>0$ such that the right-hand side above is in turn at most
\[ C \sum_{z\in\bD} |\cE(\rho(\tilde W_z))| + |\cF(\rho(\tilde W_z))| \leq  C \sum_{z\in\bD} \fm(\tilde W_z)\,.\qedhere\]
\end{proof}

The following lemma controls the vertical interactions between the shift in $\bX_2^\cI$ relative to faces in $\bX_1^\cI$.  
In this way, $\lrvec \theta \bX_1^\cI = \bX_1^\cJ$ and $\lrvec \theta \bX_2^\cI= \bX_2^\cJ$.

\begin{lemma}\label{lem:X2-vs-X4}There exists $\bar C$ such that for every $\cI \in \bar{\bI}_{x,T,H}$, 
\[ \sum_{f\in \bX_1^\cI} \sum_{g\in \bX_2^\cI} e^{-\bar c d(f,g)} + e^{-\bar c d(\lrvec\theta f,\lrvec\theta g)} \leq \bar C\,.\]
\end{lemma}
\begin{proof}Assume $t>j^*$ as otherwise $\bX_1^\cI$ is empty. We can bound the left-hand side above by 
\begin{align*}
\sum_{f\in \bX_1^\cI} \sum_{k> k^*} \sum_{g\in \sX_k} e^{-\bar c d(f,g)}+ e^{ - \bar c d(\lrvec \theta f, \lrvec\theta g)} & \leq
    C\sum_{k> k^*} |\cF(\sX_k)| e^{-\bar c (k-t)}
\end{align*} 
for some $C>0$. By criterion ({\tt B1}), for every $k >k^*$, $\fm(\sX_k)\leq k-t$ and this is at most 
\[     C \sum_{k>k^*} [4+3(k-t)] e^{-\bar c (k-t)}  \leq \bar C\,.\qedhere
\]
\end{proof}

The remaining two lemmas are more involved as they control the interactions between faces in $\bC_1$ and $\bC_2$ (which may shift vertically) with the horizontal shifts of $\bX_1^\cI \cup \bX_2^\cI$. Such terms were not considered in previous works and they cannot be reduced to either two-dimensional bound via projections, nor to a one-dimensional bound via height differences. As explained in Section~\ref{subsec:reader-guide-map}, these bounds are very sensitive to the particular choices for the deletion criteria, particularly ({\tt A1}),({\tt B1}) and ({\tt A3}),({\tt B3}). 

Recall that for all $f\in \bB$, $\theta_{\udarrow}f$ was defined as the vertical shift of $f$ induced by removal of the walls in $\bW$ per Lemma~\ref{lem:standard-wall-representation}. With this in mind, note that $\theta_{\udarrow} f \in \Theta_{\udarrow}^{\cI \setminus \cS_x} W$ for every $f\in W\cup \lceil W\rceil$.

\begin{lemma}\label{lem:C1-vs-X2-and-X4} There exists $\bar C$ such that for every $\cI \in \bar{\bI}_{x,T,H}$, 
\[ \sum_{f\in \bC_1} \sum_{g\in \bX_1^\cI\cup\bX_2^\cI} e^{-\bar c d(f,g)} + e^{-\bar c d(\theta_\udarrow f, \lrvec\theta g)} \leq \bar C \fm(\cI;\cJ)\,.\]
\end{lemma}
\begin{proof}Begin by considering $g\in \bX_1^\cI$ (assuming $t>j^*$ as otherwise $\bX_1^\cI$ is empty). There exists $C>0$ such that
\begin{align*}
 \sum_{f\in \bC_1} \sum_{g\in \bX_1^\cI} e^{-\bar c d(f,g)} + e^{-\bar c d(\theta_\udarrow f, \lrvec\theta g)} &\leq C\sum_{j^*<i<t} \sum_{y\in \bY} |\cF(\sX_i)|\left( e^{- \bar cd(\lceil \tilde W_y\rceil, \sX_i)} +
  e^{-\bar c d(\theta_\udarrow(\lceil \tilde W_y\rceil),\lrvec\theta\sX_i)}\right)\\
  &\leq 2C \sum_{j^*<i<t} \sum_{y\in \bY}  |\cF(\sX_i)| e^{-\bar c\fD_x(\tilde W_y,i,v_{j^*+1},0)}\,,
 \end{align*}
since $\Theta_\udarrow^{\cI\setminus \cS_x} \tilde W_y$ includes $\lceil \tilde W_y\rceil$ and $\theta_\udarrow(\lceil \tilde W_y\rceil)$, and $\lrvec\theta \sX_i$ is exactly $\theta_{\rho(x-v_{j^*+1})}\sX_i$, which was one of the horizontal translates considered in the definition of $\fD_x$. Further, since $\sX_i$ was not deleted, by~({\tt A1}) and ({\tt A3}),
\[ \fm(\sX_i) < (i-1)\qquad\mbox{and}\qquad\fD_x(\tilde W_y,i,-v_{j^*+1},0) > (i-1)/2\,, \]
so that for every $j^* < i <t$ and every $y\in \bY$, 
\[\fm(\sX_i) < 2\fD_x(\tilde W_y,i,-v_{j^*+1},0)\,.\]
Using $|\cF(\sX_i)| \leq 3\fm(\sX_i)+4$, we have that the above sum is at most
\[ 
  8C \sum_{j^*<i<t} \sum_{y\in \bY}  \fD_x(\tilde W_y,i,-v_{j^*+1},0) e^{-\bar c\fD_x(\tilde W_y,i,-v_{j^*+1},0)}\,.
  \]
  Observe first that for some $C>0$, 
  \begin{equation}\label{eq:per-incr-D-bound} \sum_{y\in\bY} \fD_x(\tilde W_y,i,-v_{j^*+1},0) e^{-\bar c \fD_x(\tilde W_y,i,-v_{j^*+1},0)} < C \quad\mbox{ for each $j^*<i<t$}\,.\end{equation}
  Indeed this follows by writing
  \begin{align*} \fD_x(\tilde W_y,i,-v_{j^*+1},0)e^{-\bar c \fD_x(\tilde W_y,i,-v_{j^*+1},0)} &\leq 
  d(\Theta_\udarrow^{\cI\setminus \cS_x}  \tilde W_y,\sX_i)e^{-\bar c d(\Theta_\udarrow^{\cI\setminus \cS_x}  \tilde W_y,\sX_i)} \\ & \quad + 
  d(\Theta_\udarrow^{\cI\setminus \cS_x}  \tilde W_y,\lrvec\theta\sX_i)e^{-\bar c d(\Theta_\udarrow^{\cI\setminus \cS_x}  \tilde W_y,\lrvec\theta\sX_i)}\\
  & \quad +
  d(\Theta_\udarrow^{\cI\setminus \cS_x}  \tilde W_y,\Theta_\trivincr\sX_i)e^{-\bar c d(\Theta_\udarrow^{\cI\setminus \cS_x}  \tilde W_y,\Theta_\trivincr\sX_i)}\,,
  \end{align*}
  and noticing that if $y\neq y'$ then $\Theta_\udarrow^{\cI\setminus \cS_x} \tilde W_y \cap \Theta_\udarrow^{\cI\setminus \cS_x} W_{y'}=\emptyset$, so that 
  after summing over $y\in\bY$, each term on the right-hand side contributes a constant.
(However, we cannot afford an overall bound of order $t-j^*$, which may not be comparable to $\fm(\cI;\cJ)$.)
  
Thus we only use the above bound to deal with increments whose height is at most the maximal height of a ceiling of $\tilde W_y$ or one of its possible vertical shifts. Namely, let 
\[ \bar\jmath=\min\left\{j : \hgt(v_{j+1}) > \max_{y\in\bY} \max_{f\in\Theta_\udarrow^{\cI\setminus \cS_x} \tilde W_y}\hgt(f) \right\}\,,\]
and denote by $\bar y$ the index of the wall attaining this height. Then, using~\eqref{eq:per-incr-D-bound},
\[ \sum_{j^*<i\leq \bar\jmath} \sum_{y\in \bY}  \fD_x(\tilde W_y,i,-v_{j^*+1},0) e^{-\bar c\fD_x(\tilde W_y,i,-v_{j^*+1},0)} \leq C \bar\jmath \leq C \fm(\tilde \fF_{\bar y})\,.
\]
For the remaining increments, for every $y\in \bY$, let 
\[ d_\rho(y,i) := d\big(\rho(\lceil \tilde W_y\rceil),\rho(\sX_i \cup \theta_{\rho(x-v_{j^*+1})}\sX_i \cup \Theta_\trivincr \sX_i) \big)\,,\]
and let $ \ell_1 = \bar\jmath + 1 < \ell_2 < \ldots < \ell_r$ be the record times of the function $d_\rho(y,\cdot)$, i.e.,  
\begin{align*} 
&d_\rho(y,\ell_j) < d_\rho(y,i)\quad\mbox{for all $\bar\jmath < i<\ell_j$}\,,\\
&d_\rho(y,\ell_r) = \min\{ d_{\rho}(y,i) : \bar\jmath < i < t\}\,.
\end{align*}

\begin{figure}
\begin{tikzpicture}

    \node (fig1) at (4.4,0) {
	\includegraphics[width=.40\textwidth]{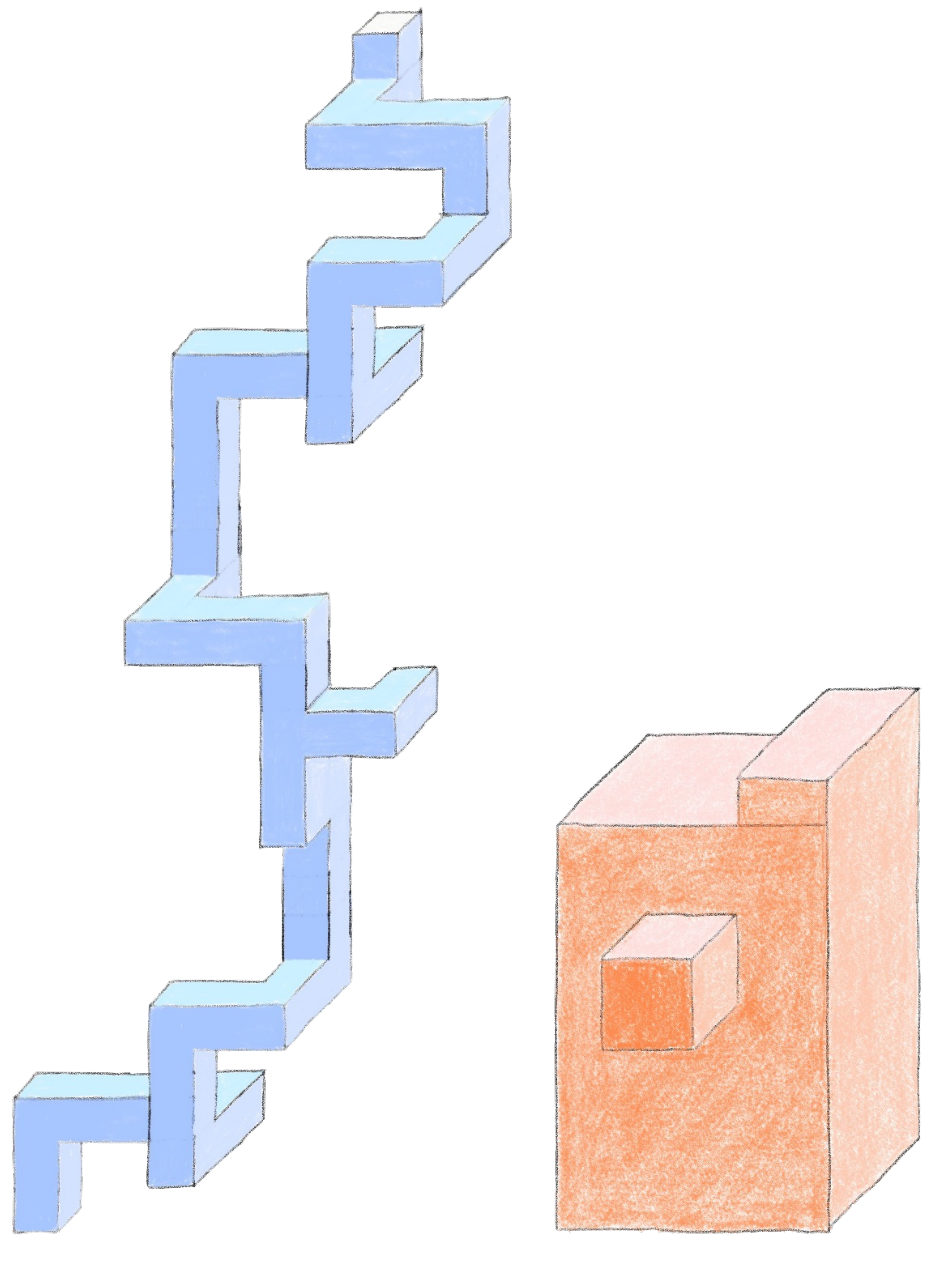}
	};

	\draw[<-] (-6+8.8, -4.15)--(-3.82+8.8, -4.15);
	
	\draw[densely dashed] (-5.58+8.8, 0)--(-5.58+8.8, -4.3);
	\draw[dashed] (-5.25+8.8, 1.9)--(-5.25+8.8, -4.3);
	
	\draw[dotted] (-4.64+8.8, 2.6)--(-4.64+8.8, -4.3);
	
	\draw [decorate,decoration={brace,amplitude=5pt}]
(-6,-1.6+.18) -- (-6,-.27+.18) node [black,midway,xshift=-.5cm] 
{\footnotesize $\sX_{\bar \jmath}$};

	\draw [decorate,decoration={brace,amplitude=5pt}]
(-6.6+8.8,-.5) -- (-6.6+8.8,.7) node [black,midway,xshift=-.5cm] 
{\footnotesize $\sX_{\ell_1}$};

	\draw [decorate,decoration={brace,amplitude=5pt}]
(-6.6+8.8,1.3) -- (-6.6+8.8,2.2) node [black,midway,xshift=-.5cm] 
{\footnotesize $\sX_{\ell_2}$};

	\draw [decorate,decoration={brace,amplitude=5pt}]
(-5.65+8.8,1.92) -- (-5.65+8.8,3.18) node [black,midway,xshift=-.5cm] 
{\footnotesize $\sX_{\ell_3}$};

    \node (fig2) at (-4.4,.18) {
    \includegraphics[width=.4\textwidth]{fig_record_times.pdf}
    };
    
    
    
    \draw[<->, dashed] (3.23-8.8,.13)--(4.95-8.8,.13);
    
    \draw[<->, dashed] (2.61-8.8,.9)--(5.28-8.8,.9);
    \draw[<->, dashed] (2.61-8.8,1.21)--(5.28-8.8,1.21);
    \draw[<->, dashed] (2.61-8.8,1.52)--(5.28-8.8,1.52);
    
    \draw[<->, dashed] (3.54-8.8,1.75)--(4.95-8.8,1.75);
    
    \draw[<->, dashed] (3.54-8.8,2.15)--(4.95-8.8,2.15);
    
    \draw[<->, dashed] (4.16-8.8,2.55)--(4.95-8.8,2.55);
    
    \draw[<->, dashed] (4.46-8.8,3.5)--(5.28-8.8,3.5);
    
    	
    	\node[font = \tiny] at (5.6-8.8,.1) {$i=1$};
    	\node[font = \tiny] at (5.6-8.8,.89) {$i=2$};
    	\node[font= \tiny] at (5.6-8.8, 1.2) {$i=3$};
    
    	\node[font = \tiny] at (5.6-8.8,1.7) {$\vdots$};

    \end{tikzpicture}
    \caption{Left: the index $\bar \jmath$ is determined by the height of the wall $W_{\bar y}$ (orange). The distances $d_\rho(\bar y,\bar \jmath +i)$ are recorded for $i =1,2 \ldots$. Right: the record times of this distance then form the increment indices $\ell_1=\bar \jmath+1, \ell_2,\ell_3$.
 }
 \label{fig:record-times}
\end{figure}
\noindent (See Figure~\ref{fig:record-times}.) Let $\ell_{r+1}: = \infty$ and observe that for every $j= 1,\ldots,r$ and every $\ell_j \leq i < \ell_{j+1}$,
\[ \fD_x(\tilde W_y,i,-v_{j^*+1}, 0) \geq \sqrt{d_\rho(y,\ell_j)^2 + (\hgt(v_i)-\hgt(v_{\bar\jmath+1}))^2} \geq \frac{d_\rho(y,\ell_j) + i-\ell_{j}}{\sqrt2}\,,\]
using the definition of $\bar\jmath$ and that it satisfies $\bar\jmath < \ell_j$.
In particular, there exists $C, C'>0$ such that
\[ \sum_{\ell_j \leq i < \ell_{j+1}} \fD_x(\tilde W_y,i,-v_{j^*+1}, 0) e^{-\bar c \fD_x(\tilde W_y,i,-v_{j^*+1}, 0)} \leq \sum_{\ell_j \leq i < \ell_{j+1}}
C e^{-\frac{\bar c}{\sqrt2} (d_{\rho}(y,\ell_j)+i-\ell_j)} \leq C' e^{-\frac{\bar c}{\sqrt2} d_{\rho}(y,\ell_j)}\,.
\]
Summing over $1\leq j \leq r$, and noticing that $r \leq d_\rho(y,\ell_1)$,
\[
 \sum_{j=1}^r e^{-\frac{\bar c}{\sqrt{2}} d_{\rho}(y,\ell_j)} = \sum_{j=1}^r e^{-\frac{\bar c}{\sqrt{2}} (d_\rho(y,\ell_1)-j)} \leq \bar C\,.
\]
Summing over $y\in\bY$, this is at most $\sum_{y\in \bY} \bar C \leq \bar C \sum_{y\in \bY} \fm(\tilde W_y)$.

The treatment of $g\in\bX_2^\cI$ is identical to the above argument, with the sole difference being the values of the horizontal shifts in the definition of $\fD_x$.
\end{proof}

\begin{lemma}\label{lem:C2-vs-X2-and-X4}There exists $\bar C$ such that for every $\cI \in\bar{\bI}_{x,T,H}$,
\[ \sum_{f\in \bC_2} \sum_{g\in \bX_1^\cI\cup\bX_2^\cI} e^{-\bar c d(f,g)} + e^{-\bar c d(\theta_\udarrow f, \lrvec\theta g)} \leq \bar C\,.\]
\end{lemma}
\begin{proof}
Begin by considering $g\in \bX_1^\cI$. If $y\notin \bY $ is such that $f\in \tilde W_y\cup\lceil \tilde W_y\rceil$ and $g\in\sX_i$ for some $j^*<i<t$, then ({\tt A3}) implies that 
\[ (i-1)/2 < \fD_x(\tilde W_y,i,-v_{j^*+1},0) \leq d(f,g) \wedge d(\theta_\udarrow f, \lrvec\theta g)\,, \]
since $\theta_\udarrow f \in \Theta_\udarrow^{\cI\setminus \cS_x} \tilde W_y$, and $\lrvec\theta g = \theta_{\rho(x-v_{j^*+1})} g$. 
Further, since $i> j^*$, by criterion ({\tt A1}), $\fm(\sX_i) < i-1$, so
\[ \sum_{f\in\bC_2}\sum_{g\in\bX_2^\cI} e^{-\bar c d(f,g)}+e^{-\bar c d(\theta_\udarrow f,\lrvec\theta g)}  \leq 2C \sum_{j^* <i <t} [3\fm(\sX_i)+4] e^{ - \bar c (i-1)/2} \leq 8 C \sum_{j^*<i<t} i e^{-\bar c (i-1)/2}\leq \frac{\bar C}{2}\,,\]
for some $\bar C$. 
The treatment of $g\in \bX_2^\cI$ is identical to the above, with the only difference being in the horizontal shift in the definition of $\fD_x$.
\end{proof}

\begin{proof}[\textbf{\emph{Proof of Proposition~\ref{prop:partition-function-contribution}}}]By Theorem~\ref{thm:cluster-expansion}, it suffices to show that there exists $C>0$ such that for every $\cI \in \bar{\bI}_{x,T,H}$, if $\cJ = \Psi_{x,t}(\cI)$, 
\begin{align*}
    \Big| \sum_{f\in \cI} \g(f,\cI) - \sum_{f' \in \cJ} \g(f',\cJ)\Big| \leq C\fm(\cI; \cJ)\,.
\end{align*}
Using the partition of the faces of $\cI,\cJ$ in Section~\ref{subsec:decomposition-interfaces}, we can expand
\begin{align}
    |\sum_{f\in\cI} \g(f,\cI) - \sum_{f'\in\cJ} \g(f',\cJ)| &\leq \sum_{f\in \bX_A^\cI \cup \bX_B^\cI} |\g(f,\cI)| +\sum_{f\in \bW} |\g(f,\cI)| \label{eq:part-function-splitting-I} \\
    & \qquad + \sum_{f'\in \bX_A^\cJ \cup \bX_B^\cJ} |\g(f',\cJ)| + \sum_{f'\in \bH} |\g(f',\cJ)|+ \sum_{f'\in W_x^\cJ} |\g(f',\cJ)| \label{eq:part-function-splitting-J}\\
    &\qquad + \sum_{f\in \bX_1^\cI\cup \bX_2^\cI} |\g(f,\cI)- \g(\lrvec\theta f,\cJ)|
    + \sum_{f\in \bB} 
    |\g(f,\cI)-\g(\theta_{\udarrow}f,\cJ)| \label{eq:part-function-splitting-IJ} \,.
\end{align}
We show that each term on the right-hand side is comparable to $\fm(\cI; \cJ)$. By~\eqref{eq:g-uniform-bound}, the first sum satisfies
\[ \sum_{f\in\bX_A^\cI\cup\bX_B^\cI} |\g(f,\cI)| \leq 
\bar K |\bX_A^\cI\cup \bX_B^\cI| \]
which is at most $\bar C \bar K \fm(\cI;\cJ)$ by Claim~\ref{clm:X1-and-X3-vs-everything}. The same holds for the first sum in line~\eqref{eq:part-function-splitting-J} by Claim~\ref{clm:X1-and-X3-vs-everything}. By Claim~\ref{clm:W-and-H-vs-everything}, the second sums in~\eqref{eq:part-function-splitting-I} and~\eqref{eq:part-function-splitting-J} are bounded in the same way by $\bar C \bar K \fm(\bW)$, which is in turn at most $3 \bar C \bar K \fm(\cI; \cJ)$ by~\eqref{eq:m(I;J)-WxJ}. The third sum in line~\eqref{eq:part-function-splitting-J} is bounded in this way by $\bar C \bar K \fm(\bW)$ via Claim~\ref{clm:Wx^J}, and this is in turn at most $3 \bar C \bar K \fm(\cI;\cJ)$ by~\eqref{eq:m(I;J)-WxJ}. 

It remains to consider the two sums in line~\eqref{eq:part-function-splitting-IJ}, which by~\eqref{eq:g-exponential-decay} satisfy,
\begin{align*}
\sum_{f\in \bX_1^\cI \cup\bX_2^\cI} |\g(f,\cI)- \g(\lrvec\theta f,\cJ)| &\leq
\sum_{f\in \bX_1^\cI\cup\bX_2^\cI} \bar K e^{-\bar c\br(f,\cI; \lrvec\theta f,\cJ)}\,, \qquad \mbox{and}\,,\\
\sum_{f\in \bB} |\g(f,\cI)- \g(\theta_\udarrow f,\cJ)| &\leq
\sum_{f\in \bB} \bar K e^{-\bar c\br(f,\cI; \theta_\udarrow f,\cJ)}\,.
\end{align*}
To evaluate the radius $\br$, consider the right-hand sides according to the face $g$ attaining $\br(f,\cI; f',\cJ)$.
\begin{enumerate}[(i)]
\item If $g \in \bX_A^\cI\cup\bX_B^\cI \cup\bX_A^\cJ\cup\bX_B^\cJ$, both these sums are at most 
\begin{align*}
     \sum_{g\in \bX_A^\cJ\cup \bX_B^\cI \cup \bX_A^\cJ \cup \bX_B^\cJ} \bar K \Big[ \sum_{f\in \bX_1^\cI \cup \bX_2^\cI} (e^{ - \bar c d(f,g)}+ e^{- \bar c d(\lrvec \theta f,g)})+ \sum_{f\in \bB}(e^{ - \bar c d(f,g)} + e^{ -\bar c d(\theta_{\udarrow} f, g)})\Big]\,.
\end{align*}
Replacing the sums over $f$ by sums over all $f\in \cF(\Z^3)$, Claim~\ref{clm:X1-and-X3-vs-everything} implies this contributes at most $4\bar C \bar K \fm(\cI;\cJ)$.

\item If $g\in \bW \cup \bH\cup W_x^\cJ$, these sums are at most 
\begin{align*}
    \sum_{g\in \bW \cup \bH\cup W_x^\cJ} \bar K \Big[ \sum_{f\in \bX_1^\cI \cup \bX_2^\cI} (e^{ - \bar c d(f,g)}+ e^{- \bar c d(\lrvec \theta f,g)}) + \sum_{f\in \bB}(e^{ - \bar c d(f,g)} + e^{ -\bar c d(\theta_{\udarrow} f, g)})\Big]\,.
\end{align*}
Replacing the sums over $f$ by sums over all $f\in \cF(\Z^3)$,  by Claim~\ref{clm:W-and-H-vs-everything} and Claim~\ref{clm:Wx^J}, this contributes at most $12\bar C \bar K \fm(\cI;\cJ)$. 

\item For $g\in \bX_1^\cI\cup\bX_1^\cJ\cup \bX_2^\cI \cup \bX_2^\cJ$, let us begin with the first sum ($f\in \bX_1^\cI \cup \bX_2^\cI$). If $f\in \bX_1^\cI$, the radius $\br$ could not have been attained by $g\in \bX_1^\cI\cup \bX_1^\cJ$ since all increments in $\bX_1^\cI$ are shifted by the same vector. Then this reduces to 
\begin{align*}
   \bar K  \sum_{f\in \bX_1^\cI} \sum_{g\in \bX_2^\cI} (e^{ - \bar c d(f,g)}+ e^{ - \bar c d(\lrvec \theta f, \lrvec \theta g)})\,, 
\end{align*}
which is at most $\bar C \bar K $ by Lemma~\ref{lem:X2-vs-X4}. If $f\in \bX_2^\cI$ , this reduces to $g\in \bX_1^\cI$ and is handled symmetrically. 

Turning to the sum over $f\in\bB$, it splits into the following: 
\begin{align*}
    \bar K \sum_{g\in \bX_1^\cI \cup \bX_2^\cI} \sum_{f\in \bC_1\cup \bC_2 \cup \bF} (e^{ - \bar c d(f,g)}+ e^{ - \bar c d(\theta_\udarrow f,\lrvec \theta g)})
\end{align*}
The contribution from $f\in \bC_1$ is at most $\bar C \bar K \fm(\cI; \cJ)$ by Lemma~\ref{lem:C1-vs-X2-and-X4}; the contribution from $f\in \bC_2$ is at most $\bar C \bar K $ by Lemma~\ref{lem:C2-vs-X2-and-X4}; the contribution from $f\in \bF$ is at most $\bar C\bar K $  by Lemma~\ref{lem:F-vs-X2-and-X4} as $\bF \subset \cL_0$.

\item For $g \in \bC_1\cup\theta_\udarrow\bC_1 \cup  \bC_2\cup\theta_\udarrow\bC_2 \cup \bF$, the first sum can be expressed as 
\begin{align*}
    \bar K \sum_{f\in \bX_1^\cI \cup \bX_2^\cI} \sum_{g\in \bC_1\cup \bC_2\cup \bF} (e^{ - \bar c d(f,g)}+ e^{ - \bar c d(\lrvec \theta f, \theta_\udarrow g)})\,.
\end{align*}
Up to a change of roles of $f$ and $g$, this is identical to the term considered in the item above, and its contribution is therefore at most $2\bar C \bar K + \bar C \bar K  \fm(\cI;\cJ)$ by Lemmas~\ref{lem:F-vs-X2-and-X4} and~\ref{lem:C1-vs-X2-and-X4}--\ref{lem:C2-vs-X2-and-X4}. 

For the second sum, in which $f \in \bB$, note that if the radius $\br(f, \cI; \theta_{\udarrow} f, \cJ)$ is attained by $g\in \bC_1 \cup \bC_2 \cup \bF \cup \bW$ or by $\theta_{\udarrow}$ of such a $g$, it must be attained by a face in a wall nested in some wall of $\bW$. Since the distance between two faces is at least the distance between their projections, and projections of distinct walls are distinct, the contribution of this term (summed over all possible such $g$) is at most 
$$\bar K \sum_{f\in\bB} \sum_{u\in\rho(\bigcup_{z\in \bD}\tilde W_z)} e^{-\bar c d(\rho(f),u)}\,,$$
which is at most $\bar C \bar K  \fm(\bW)\leq 3\bar C \bar K  \fm(\cI; \cJ)$ by Lemma~\ref{lem:B-vs-B} and~\eqref{eq:m(I;J)-WxJ}. 
\end{enumerate}
Altogether, we deduce that all the summands on the right-hand side of~\eqref{eq:part-function-splitting-I}--\eqref{eq:part-function-splitting-IJ} are bounded by an absolute constant times $\fm(\cI;\cJ)$, implying the desired.
\end{proof}

\subsection{Proof part 2: multiplicity}
We next bound the multiplicity of the map $\Psi_{x,t}$ with a fixed excess area $M$ by an exponential in $M$ (independently of $\beta$). 

\begin{proposition}\label{prop:multiplicity}
There exists some universal $C_\Psi$ such that 
for every $M\geq 1$ and every $x,t$ and $T< H$,
\begin{align*}
\max_{\cJ\in\Psi_{x,t}(\bar {\bI}_{x,T,H})}\left|\{\cI\in\Psi_{x,t}^{-1}(\cJ) \,:\; \fm(\cI;\cJ)=M\}\right| \leq C_\Psi^M\,.
\end{align*}
\end{proposition}

Towards proving Proposition~\ref{prop:exp-tail-all-increments}, we are also interested in a map used to prove an exponential tail on the increment of a pillar that intersects a given height $\ell$ (as opposed to an increment of a given index). For that purpose, for any pillar $\cP_x$ and a half-integer height $\ell$, let 
$$
\tau_\ell= \tau_\ell (\cP_x):=  \min \{t\geq 1: \sX_t \cap \cL_\ell \neq \emptyset \} \vee 1
$$
and define the map $\tilde \Psi_{x,\ell}$ as 
\begin{align*}
    \tilde \Psi_{x,\ell}(\cI) : = \Psi_{x,\tau_\ell(\cP_x)}(\cI)\,.
\end{align*}
Clearly since the bound of Proposition~\ref{prop:partition-function-contribution} is independent of $\cI$ and $t$, the estimate also holds for $\tilde \Psi_{x,\ell}(\cI)$. However, handling the multiplicity is slightly different since interfaces with differing $\tau_\ell$ may be mapped to the same $\cJ\in \tilde \Psi_{x,\ell}(\bar{\bI}_{x,T,H})$. 

\begin{proposition}\label{prop:multiplicity-at-height}
There exists a universal $\tilde C_\Psi>1$ such that for every $M\geq 1$ and every $x, \ell$ and $T< H$ 
\begin{align*}
    \max_{\cJ \in \tilde \Psi_{x,\ell}(\bar{\bI}_{x,T,H})} \big| \{\cI \in \tilde \Psi_{x,\ell}^{-1}(\cJ)\,: \, \fm(\cI;\cJ) = M\} \big| \leq \tilde C_\Psi^M\,.
\end{align*}
\end{proposition}

We prove Propositions~\ref{prop:multiplicity}--\ref{prop:multiplicity-at-height} by constructing a witness that (given $\Psi_{x,t}(\cI)$) is in 1-1 correspondence with the pre-image $\cI$. We then bound the number of all possible such witnesses. 

Let us fix any $\cJ\in \Psi_{x,t}(\bar{\bI}_{x,T,H})$ (or respectively $\cJ\in \tilde \Psi_{x,\ell}(\bar{\bI}_{x,T,H})$). We wish to define an injective map $\Xi=\Xi_{\cJ,x,t}$ (respectively, $\tilde \Xi = \tilde \Xi_{\cJ, x, \ell}$ on $\{\cI \in \Psi_{x,\ell}^{-1}(\cJ)\}$) and bound the cardinality of the set $\Xi(\{\cI\in\Psi_{x,t}^{-1}(\cJ) \,:\; \fm(\cI;\cJ)=M\})$ (resp., $\tilde \Xi(\{\cI\in\Psi_{x,\ell}^{-1}(\cJ) \,:\; \fm(\cI;\cJ)=M\})$.

\medskip
\noindent \textbf{Construction of the witness.}
Fix $x$ and $t$ (respectively $\ell$).
We describe how for a given $\cJ$ and an $\cI \in \Psi_{x,t}^{-1}(\cJ)$ (respectively $\cI \in \tilde \Psi_{x,\ell}^{-1}(\cJ)$) we construct the witness $\Xi_{\cJ,x,t}(\cI)$ (respectively $\tilde \Xi_{\cJ, x,\ell}(\cI)$). In order to do this in a unified manner, we let 
\[\tilde \Xi_{\cJ,x,\ell}(\cI) = \Xi_{\cJ, x, \tau_{\ell}(\cP_x)}(\cI)
\]
and then it suffices to describe how to construct $\Xi (\cI)= \Xi_{\cJ, x,t}(\cI)$ for each $\cI$. 

Our witness $\Xi(\cI)$ will consist of six $*$-connected face-subsets $(\cF_\iota^\gamma)_{\iota \in \{A,B\}, \gamma \in \{\cI, \cJ, \Psi\}}\subset \cF(\Z^3)$, each of which are decorated by coloring its faces \blue\ or \red, and associating to each $v\in \cF_\iota^\gamma$, its own individual face-subset $\Upsilon_v \subset \cF(\cF(\Z^3))$  whose faces are also colored \blue\ or \red.

Let us begin by constructing the six $*$-connected face-subsets $\cF_\iota^\gamma$ and their colorings. 
Partition $\bY$ into $\bY_\sB$, and $\bY_A^\cI \cup \bY_A^\cJ \cup \bY_A^\Psi$ along $\bY_B^\cI \cup \bY_B^\cJ \cup \bY_B^\Psi$  as follows:
\begin{itemize}
    \item Let $\bY_\sB = \{ v_1, y^\dagger \}\cup [x]$ (recalling that $[x]$ indicates $x$ and its four adjacent faces in $\cL_{0,n}$).
    \item Let $\bY_A^\cI$ (resp., $\bY_B^\cI$) be the indices of walls that were first marked for deletion due to one of the criteria ({\tt A2}),({\tt A3}) (resp., ({\tt B2}),({\tt B3})) wherein $\fD_x(\tilde W_y,i,\omega_1,\omega_2)$ was attained by $d(\Theta_\udarrow^{\cI\setminus \cS_x} \tilde W_y,\sX_i)$.
    \item Let $\bY_A^\cJ$ (resp., $\bY_B^\cJ$) be the indices of walls that were first marked for deletion due to ({\tt A2}) or ({\tt A3}) (resp., ({\tt B2}) or ({\tt B3})) wherein $\fD_x(\tilde W_y,i,\omega_1,\omega_2)$ was attained by $d(\Theta_\udarrow^{\cI\setminus \cS_x} \tilde W_y,\theta_{\rho(x+\omega_2)}\Theta_\trivincr\sX_i)$.
    \item Let $\bY_A^\Psi$ (resp., $\bY_B^\Psi$) are the indices of walls that were first marked for deletion due to ({\tt A2}) or ({\tt A3}) (resp., ({\tt B2}) or ({\tt B3})) wherein  $\fD_x(\tilde W_y,i,\omega_1,\omega_2)$ was attained by $d(\Theta_\udarrow^{\cI\setminus \cS_x} \tilde W_y,\theta_{\rho(x+\omega_1+\omega_2)}\sX_i)$.
\end{itemize}

\begin{figure}
  \centering
  \begin{tikzpicture}
  
      \node (pillar) at (-3.5,0)
   {\includegraphics[width=.2\textwidth]{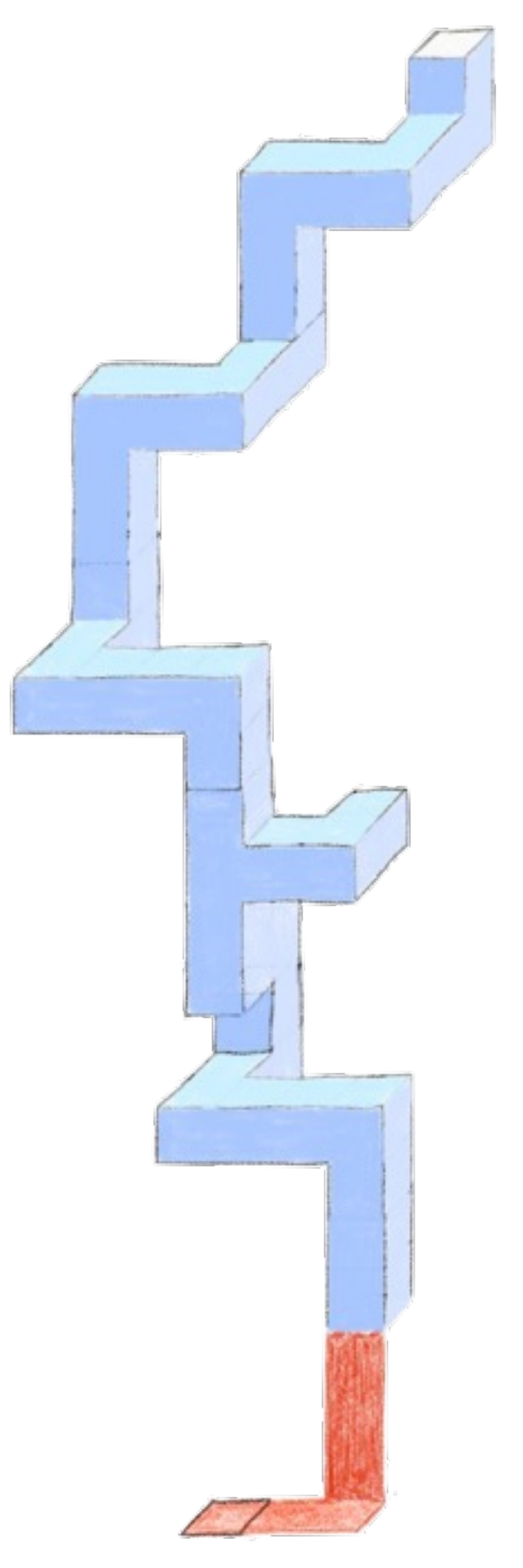}};
   
          \node (pillar) at (3.5,.35)
   {\includegraphics[width=.2\textwidth]{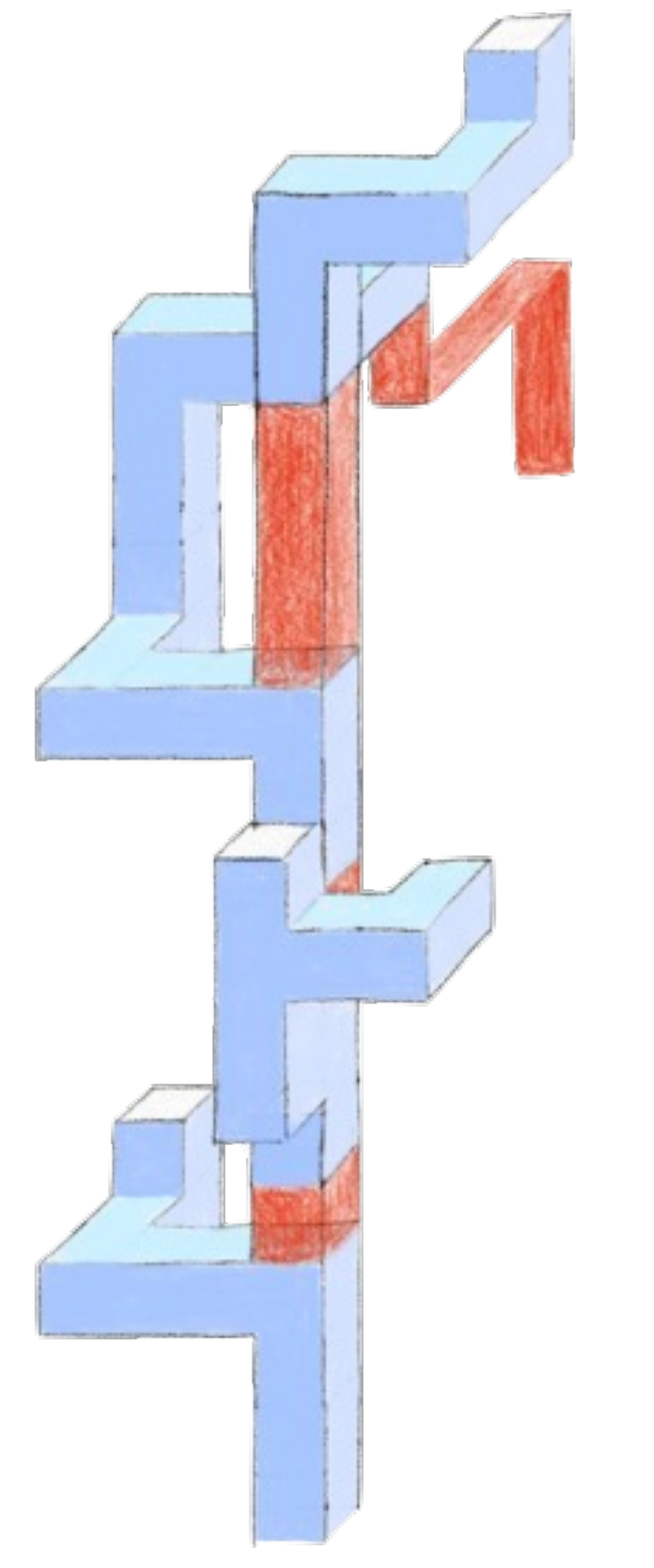}};

    \node[ellipse,draw,fill=gray!10,inner sep=0mm,minimum width=75pt, minimum height=75pt] (zoom) at (-6,-1) {
    };
    
        \node[ellipse,draw,fill=gray!10,inner sep=0mm,minimum width=75pt, minimum height=75pt] (zoom2) at (6,-1) {
    };

    \node[left=-33pt] at (zoom) {        \includegraphics[width=0.13\textwidth]{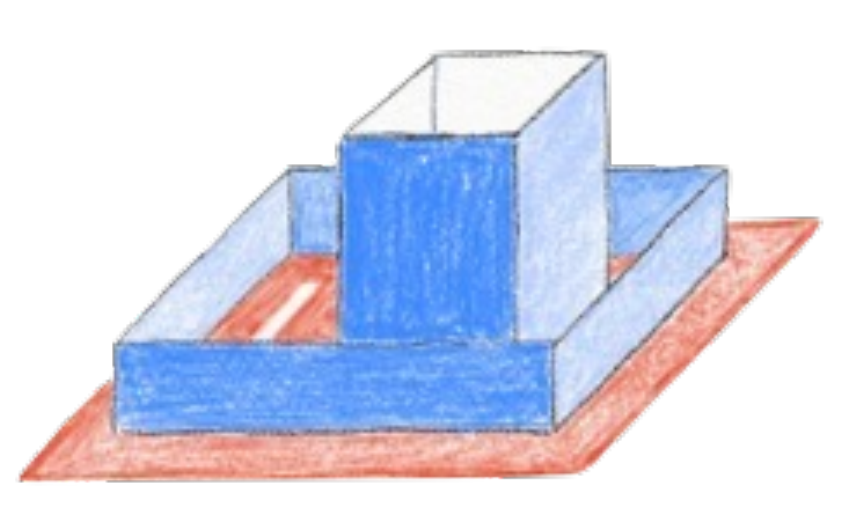}
    };

    \node[circle,draw,fill=black,inner sep=0mm,minimum width=3pt] (org) at (-3.07,-3.45) {
    };
    
        \node[circle,draw,fill=black,inner sep=0mm,minimum width=2.5pt] (org2) at (4.73,1.92) {
    };

       \node[left=-32pt] at (zoom2) {\includegraphics[width=0.13\textwidth]{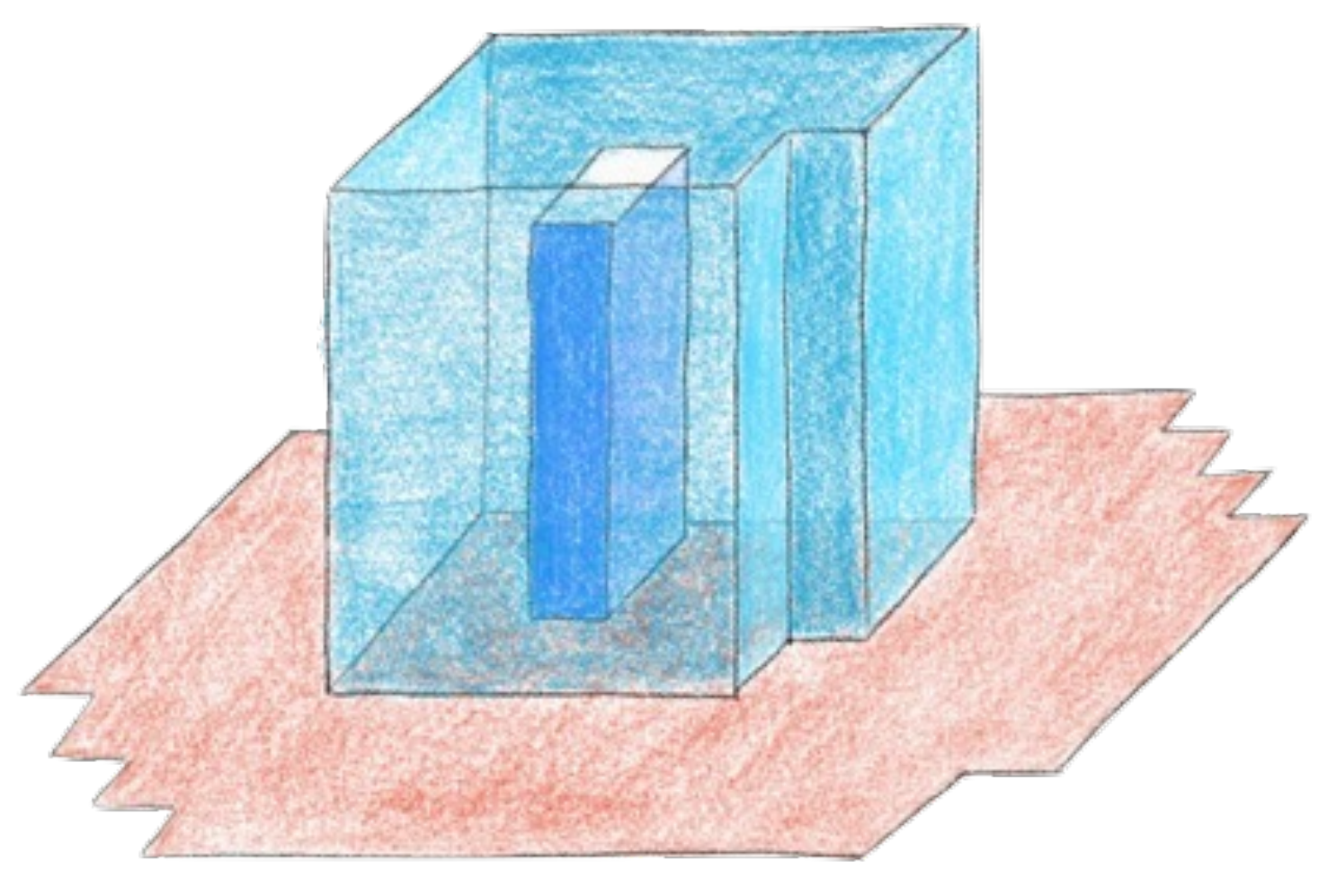}
    };

	\draw[gray!75] (zoom.10) -- (org.120);
	\draw[gray!75] (zoom.280) -- (org.240);
	
		\draw[gray!75] (zoom2.70) -- (org2.350);
	\draw[gray!75] (zoom2.160) -- (org2.240);
	
	  \node[font = \tiny] at (-3.72,-4.64) {$x$};

	  \node[font = \large] at (-3.5,4.5) {$\cF_A^\cI$};
	  \node[font = \large] at (3.5,4.5) {$\cF_A^\Psi$};
	  
	  \node at (-5.9,-.13) {$\Upsilon_{\rho(v_1)}$};
	  
	  \node at (5.9,-.13) {$\Upsilon_{w}$};

  \end{tikzpicture}
  \vspace{-0.27in}
  \caption{Two of the six constituent parts of the witness. Left: $\cF_{A}^\cI$ decorated by $\{\blue,\red\}$ and the vertex-decoration $\Upsilon_{\rho(v_1)}$. Right: $\cF_{A}^{\Psi}$ decorated by $\{\blue, \red\}$ and the vertex-decoration $\Upsilon_w$.}
  \vspace{-0.05in}
    \label{fig:witness-examples}
\end{figure}

When considering the criteria ({\tt A2}),({\tt A3}),({\tt B2}),({\tt B3}), the Euclidean distance between sets of faces in $\R^3$ is attained by vertices of $\Z^3$, which we will endow with an (arbitrary) lexicographic ordering, giving rise to unique minimizers of the distance.
Further, let
\begin{align*} \bX_A^\Psi &= \bigcup\{ \cF(\theta_{\rho(x-v_{\fs_j+1})} \sX_j) \,:\; 1 \leq j \leq j^* \} \,,\\
 \bX_B^\Psi &= \bigcup\{ \cF(\theta_{\rho(x+v_t-v_{j^*+1}-v_{\fs_k+1})} \sX_k) \,:\; t \leq k \leq k^* \} \,.\end{align*}
With these definitions in hand, the witness $\Xi(\cI)$ is constructed as follows: 
\begin{enumerate}[(1)]
\item The \blue\ faces of $\cF_\iota^\gamma$ for $\iota\in\{A,B\}$ and $\gamma\in\{\cI,\cJ,\Psi\}$ are precisely $\bX_\iota^\gamma$ (calculated via running $\Psi_{x,t}$).
\item For each $y\in\bY_\iota^\gamma$ (for $\iota\in\{A,B\}$ and $\gamma\in\{\cI,\cJ,\Psi\}$):
\begin{itemize}
    \item Let $v\in\bX_\iota^\gamma$ and $w\in\Theta_\udarrow^{\cI\setminus \cS_x} \tilde W_y$ be the minimizers of the distance that was violated in the respective deletion criterion (at the first time $y$ was marked for deletion).
    \item Add to $\cF_\iota^\gamma$ a shortest (nearest-neighbor) path of \red\ faces connecting $v$ to $w$.
\end{itemize}
\item Add to $\cF^\cI_A$ a shortest path of \red\ faces connecting $v_1$ and $y^\dagger$, and such a path connecting $v_1$ and $x$; include the face $x$ and color it \red.
\item Add to $\cF_\iota^\Psi$ for $\iota\in\{A,B\}$ every face of $\bX_\iota^\cJ$ that is not already present and color it \red.
\item Let $\Sigma$ be the set of all vertices $w$ for which we added a shortest \red\ path from $v$ to $w$ in step (2) above. Process the vertices in $w\in \Sigma \cup \cV(\bY_\sB)$ 
via some lexicographic order $(w^1 , w^2 , \ldots,w^{|\Sigma \cup \cV(\bY_\sB)|})$: for $i = 1,\ldots, |\Sigma \cup \cV(\bY_\sB)|$,
\begin{itemize}
    \item If $w^i$ is associated with the wall $\tilde W_y$ (i.e., $w^i\in\Theta_\udarrow^{\cI\setminus \cS_x} \tilde W_y$), let the \blue\ faces of $\Upsilon_{w^i}$ be the set  
\[ \fS_{w^i} = \tilde \fF_y \setminus \bigcup_{j<i} \fS_{w^j}\,.\]
\item For each edge or face $u\in\rho(\fS_{w^i})$, add to $\Upsilon_{w^i}$ the \red\ set of faces in the ball of radius $N_\rho(u)$ around $u$ in $\cL_{0,n}$. Complete $\Upsilon_{w^i}$ by connecting all faces added to it, together with the vertex $\rho(w^i)$, via a \red\ minimum size spanning tree of faces in $\cL_{0,n}$.
\end{itemize}
(For every other vertex $w\in \cV(\cF_\iota^\gamma)$, we let $\Upsilon_w=\emptyset$.)
\end{enumerate}

See Figure~\ref{fig:witness-examples} for examples of $\cF_A^\cI$ and $\cF_A^\Psi$ together with their decorations.

\medskip
\noindent \textbf{Reconstructing $\cI$ from the witness.} To see that this indeed yields a ``witness" of the pre-image interface $\cI$, we show that from a witness $\Xi(\cI)$ and the interface $\cJ$, one can reconstruct $\cI$. 

\begin{lemma}\label{lem:witness-injective}
For every $\cJ \in \Psi_{x,t}(\bar \bI_{x,T,H})$ (respectively, $\cJ \in \tilde \Psi_{x,\ell}(\bar \bI_{x,T,H})$), the map $\Xi_{\cJ,x,t}$ (resp., $\tilde \Xi_{\cJ, x, \ell}$) is injective on $\Psi_{x,t}^{-1}(\cJ)$ (resp., $\tilde \Psi_{x, \ell}^{-1}(\cJ)$). 
\end{lemma}
\begin{proof}
It suffices to show that from a given $\cJ$ and any element of  $\Xi_{\cJ, x, t}(\Psi_{x,t}^{-1}(\cJ))$ we can recover, uniquely, $\cI\in \Psi_{x,t}^{-1}(\cJ)$.  
From a witness in $\Xi_{\cJ,x,t}(\Psi_{x,t}^{-1}(\cJ))$, we recover $\cI$ by reconstructing its spine $\cS_x$ together with the standard wall representation of $\cI \setminus \cS_x$. Given $\cS_x$ and $(\tilde W_z)_{z\in \cL_{0,n}}$, we would obtain $\cI$ by first recovering the interface $\cI \setminus \cS_x$ via Lemma~\ref{lem:standard-wall-representation}, then appending to that $\cS_x$. 
\begin{enumerate}
    \item In order to reconstruct the spine $\cS_x$:
    \begin{enumerate}
        \item Extract $\bX_A^\cI$ and $\bX_B^\cI$ as exactly the set of \blue\ faces of $\cF^\cI_A$ and $\cF^\cI_B$ respectively. 
        \item Extract $\bX_1^\cJ$ and $\bX_2^\cJ$ by taking (the bounding faces of) all cells in $\cP_x^\cJ$ between $\hgt(v_{j^\star+1})$ and $\hgt(v_{t})$,  and above $\hgt(v_{k^\star +1})$, respectively (these heights are read off from $\bX_A^\cI$ and $\bX_B^\cI$).
    \item Obtain $\cS_x$ by horizontally shifting $\bX_1^\cJ$ and $\bX_2^\cJ$ so that their bottom cell coincides with the top cell of $\bX_A^\cI$ and $\bX_B^\cI$ respectively.
    \end{enumerate}
    \item In order to reconstruct the standard wall representation of $\cI\setminus \cS_x$:
    \begin{enumerate}
    \item For every vertex $w\in\cV(\cF_\iota^\gamma)$ for $\iota\in\{A,B\}$ and $\gamma\in\{\cI,\cJ,\Psi\}$, add the faces of $\fS_w$ (exactly the set of \blue\ faces of $\Upsilon_w$).
    \item Add the standardizations of all walls of $\cJ\setminus \cP_x^\cJ$.
    \end{enumerate}
\end{enumerate}
For the corresponding reconstruction from a witness in $\tilde \Xi_{\cJ,x,\ell}(\tilde \Psi^{-1}_{x,\ell}(\cJ))$, we recover $\cI$ in exactly the same way, noticing that we can read off $\hgt(v_{\tau_\ell})$ from $\bX_B^{\cI}$.   
\end{proof}

\medskip
\noindent \textbf{Enumerating over possible witnesses.} It remains to enumerate over the set of all possible witnesses of interfaces in $\Psi_{x,t}^{-1}(\cJ)$ and $\tilde \Psi^{-1}_{x,\ell}(\cJ)$ with excess area $\fm(\cI;\cJ) = M$ and show it is at most exponential in $M$. 
\begin{lemma}\label{lem:multiplicity-witness}
There exists some universal $C_\Psi>0$ such that for every $M\geq 1$, $x,t$, and $\cJ \in \Psi_{x,t}(\bar{\bI}_{x,T,H})$,
\begin{align*}
|\{\Xi_{\cJ,x,t} (\cI): \cI\in \Psi^{-1}_{x,t}(\cJ)\,,\, \fm(\cI;\cJ) = M\}| \leq C_\Psi^M\,.
\end{align*}
Similarly, there exists $\tilde C_{\Psi}$ such that for every $M\geq 1 $, $x,\ell$ and $\cJ \in \tilde \Psi_{x,\ell}(\bar{\bI}_{x,T,H})$, 
\begin{align*}
    |\{\tilde \Xi_{\cJ,x,\ell} (\cI): \cI \in \tilde \Psi^{-1}_{x,\ell}(\cJ)\,,\, \fm(\cI;\cJ) = M\}| \leq \tilde C_\Psi^M\,.
\end{align*}
\end{lemma}
Combining the above lemma with Lemma~\ref{lem:witness-injective} immediately implies Propositions~\ref{prop:multiplicity}--\ref{prop:multiplicity-at-height}.

\begin{proof}
Let us prove the bound on the number of possible witnesses corresponding to $\Psi_{x,t}$ and $\tilde \Psi_{x,\ell}$ simultaneously, describing in the proof the parts that are different between $\Psi_{x,t}$ and $\tilde \Psi_{x,\ell}$. 

Fix $M, x,t, \ell$ and $T< H$ and $\cJ \in \Psi_{x,t}(\bar{\bI}_{x,T,H})$ (respectively $\cJ \in \Psi_{x,\ell}(\bar{\bI}_{x,T,H})$), and consider the number of possible witnesses $\Xi(\cI)$ for $\cI$ satisfying $\fm(\cI;\cJ) = M$. We decompose this into the number of possible choices of colored face-sets $\cF_\iota^\gamma$, and subsequently, the number of choices of decorations to the vertices of $\cF_\iota^\gamma$ via the number of choices of colored face-sets $(\Upsilon_w)_w$. Clearly, their product bounds the number of possible choices of witnesses.

\medskip
\noindent \emph{Number of faces in $\cF^\gamma_\iota$.} 
We first bound the number of faces in each $\cF_\iota^\gamma$ for $i\in \{A,B\}$ and $\gamma \in \{\cI,\cJ, \Psi\}$. 
\begin{enumerate}[(1)]
\item The number of \blue\ faces in $\cF^\gamma_\iota$ is exactly  $|\bX_\iota^\gamma|$. For $\gamma\in\{\cI,\cJ\}$, this quantity is at most $\bar C\fm(\cI;\cJ)$ by Claim~\ref{clm:X1-and-X3-vs-everything}. For $\gamma=\Psi$, we have that 
\[ |\bX_\iota^\Psi| \leq |\bX_\iota^\cI| + |\bX_\iota^\cJ| \leq 2\bar C \fm(\cI;\cJ)\,.\]
(Note that $\bX_\iota^\Psi$ consists of shifts of the increments composing $\bX_\iota^\cI$; these shifts add an additional cell whenever $\fs_i \neq \fs_{i-1}$, due to the additional 4 faces of the shared cut-point between consecutive increments. We compensate for these via the term $|\bX_\iota^\cJ|$.)
\item For each $w\in\Sigma $ associated with some wall $\tilde W_y$ (for some $y\in\bY_\iota^\gamma$):
\begin{itemize}
    \item If $y\neq y_\iota^*$, then the number of \red\ faces that were added to connect $w\in \Theta_{\udarrow}^{\cI \setminus \cS_x} \tilde W_y$ to $|\bX_\iota^\gamma|$ is at most $2\fm(\tilde W_y)$ (since $\tilde W_y$ violated criteria ({\tt A2}) or ({\tt B2})), where the factor of $2$ accounts for the transition from Euclidean distance to the graph distance in $\Z^3$.
    \item If $y= y_\iota^*$, then the number of \red\ faces that were added to connect $w\in \Theta_{\udarrow}^{\cI \setminus \cS_x} \tilde W_y$ to $|\bX_\iota^\gamma|$ is at most $(j^*-1) \vee (k^*-t)$  (since $\tilde W_y$ either violated criteria ({\tt A3}) or ({\tt B3})).
\end{itemize}
Summing these over all $y\in \bY_\iota^\gamma$ gives at most  $\fm(\bW)+(j^*-1) + (k^*-t)$ additional \red\ faces, which is at most $15\fm(\cI;\cJ)$ by Claim~\ref{clm:m(W)}. 
\item The number of \red\ faces added to connect $v_1$ to $y^\dagger$ as well as to $x$ is 
at most $\fm(\tilde\fF_{v_1})$, since all of these are part of $\cP_x$, and hence share some nesting wall by Observation~\ref{obs:pillar-nested-sequence-of-walls}.  Thus, the number of such faces is at most $2\fm(\bW) \leq 6\fm(\cI;\cJ)$ by Claim~\ref{clm:m(W)}.
\item The number of \red\ faces added to $\cF_\iota^\Psi$ by $\bX_\iota^\cJ$ is at most $\bar C \fm(\cI;\cJ)$ by Claim~\ref{clm:X1-and-X3-vs-everything}.
\end{enumerate}
Altogether, we see that there exists a universal $C>0$ such that for every $M\geq 1$, 
\begin{align*}
    \max_{\cI \in \Psi_{x,t}^{-1}(\cJ):\fm(\cI;\cJ) = M}\,\max_{\iota \in \{A,B\}, \gamma \in \{\cI, \cJ, \Psi\}} \, |\cF_\iota^\gamma| \leq C M\,.
\end{align*}
Since the constant $C$ above was uniform over the choice of $t$, and the witness constructed by $\tilde \Xi_{x,\ell}$ agrees with the witness constructed by $\Xi_{x,t}$ for some $t=t(\cI)$, they apply equally to the witnesses coming from $\tilde \Xi_{x,\ell}$. 

\medskip
\noindent \emph{Number of possible choices of $\cF_{\iota}^\gamma$.} 
We begin by enumerating over the choices of $\cF_A^\cI$: by Fact~\ref{fact:number-of-connected-face-sets}, the number of $*$-connected face sets rooted at $x$ (predetermined) with at most $C M$ faces is at most $s^{C M}$; multiplying this by $2^{CM}$ for the choices of \blue\ and \red\ colorings of these faces, bounds the total number of possible choices for $\cF_A^\cI$. Notice that the choice of $\cF_A^\cI$ reveals $v_1$ as its lowest cell that is bounded by (four) \blue\ faces, and $v_{j^*+1}$ as its highest such cell. This also determines whether $t>j^*$ or not.   

Next, if $t>j^*$, we enumerate over the choices of $\cF_B^\cI$. In the case of $\Xi_{x,t}$, its root, $v_t$, is determined by our above choice of $\cF_A^\cI$ as follows: starting from $v$, the cut-point of $\cP_x^\cJ$ at height $\hgt(v_{j^*+1})$, count $t-(j^*+1)$ extra increments (with $t$ predetermined) to a cut-point $v'$ (marking the top of $\bX_1^\cJ$ and the bottom of $\bX_B^\cJ$). The root $v_t$ is $\theta_{v'-v}(v_{j^*+1})$. Enumerating over $\cF_B^\cI$ then amounts to another factor of $(2s)^{CM}$.

In the case of $\tilde \Xi_{x,\ell}$, we can enumerate over choices of root $v_{\tau_\ell}$ by enumerating over the cut-point $v'$ of $\cJ$ whose height coincides with $\hgt(v_{\tau_\ell})$. In $\cJ$, the cutpoint $v'$ must be within a height of at most $\fm(\sX_{\tau_\ell}) \leq M$ from $\ell$, so there are at most $M$ such choices of cut-point $v'$. From that we recover the root $v_{\tau_\ell}$ as $\theta_{v'-v}(v_{j^*+1})$; choosing $\cF_B^\cI$ thus amounts to a factor of $M (2s)^{CM}$.

We bound the number of possible choices for $\cF_A^\cJ$ from above by $s^{CM}$, the number of $*$-connected sets of $CM$ faces, rooted at $x+(0,0,\hgt(v_1))$ (recall that $v_1$ can be read off of $\cF_A^\cI$), multiplied by a $2^{CM}$ factor for coloring those faces by \blue\ and \red. If $t>j^*$, the number of choices of $\cF_B^\cJ$ is at most the number of $*$-connected sets of $CM$ faces, rooted at the cut-point $v'$ mentioned above (which we read off of $\cJ$ and $\cF_A^\cI$), along with its coloring by \blue\ and \red, which combine to at most $(2s)^{CM}$.

To enumerate over $\cF_A^\Psi$, we first argue that it is a $*$-connected set of faces. In order to see this, recall that the \blue\ faces of $\cF_A^\Psi$ are comprised of various subsets of shifted increment sequences, which can only become disconnected when there is a change in the shift applied: i.e., at increments where $\fs_{j+1} \neq \fs_{j}$. Suppose that $\fs_{j+1} \neq \fs_{j}$ for some $1\leq j< j^*$; by Observation~\ref{obs:fs_i}, this occurs if and only if $\fs_{j+1} = j$ and $\sX_{j+1}$ is being shifted in $\bX_A^\Psi$ by
\[\rho(x-v_{\fs_{j+1}+1}) = \rho(x-v_{j+1})\,. \]
It suffices to show that for every such $j$, the cell $\theta_{\rho(x-v_{j+1})}v_{j+1}$ is $*$-connected to the cell $x+(0,0,\hgt(v_1))$ by faces in $\cF_{A}^{\Psi}$. 
 This follows since the bounding faces of the cell $\theta_{\rho(x-v_{j+1})} v_{j+1}$ are also in $\bX_A^\cJ$ (as a subset of $\theta_{\rho(x)}\Theta_\trivincr\sX_{j+1}$, to which $\sX_{j+1}$ was mapped), $\bX_A^\cJ$ connects this cell to $x+(0,0,\hgt(v_1))$ (via trivial increments), and $\bX_A^\cJ \subset \cF_A^\Psi$. As a connected set of at most $CM$ faces, colored by \blue\ and \red\ and rooted at $x+(0,0,\hgt(v_1))$ (which is dictated by our choice of $\cF_A^\cI$), there are at most $(2s)^{CM}$ choices for $\cF_A^\Psi$.

Similarly, if $t>j^*$, to see that $\cF_B^\Psi$ is $*$-connected, note that for $t \leq k < k^*$, the shifted increments $k$ and $k+1$ can only be disconnected if $\fs_{k+1}\neq \fs_k$, which occurs if and only if $\fs_{k+1}=k$, in which case the increment $\sX_{k+1}$ will be shifted in $\bX_B^\Psi$ by
\[ \rho(x+v_t-v_{j^*+1}-v_{\fs_{k+1}+1}) = \rho(x+v_t-v_{j^*+1}-v_{k+1}) \,.\]
The fact that $\bX_B^\cJ$ (which, as before, connects $\theta_{\rho(x+v_t-v_{j^*+1}-v_{k+1})}v_{k+1}$ to $x+v_t-v_{j^*+1}$ via trivial increments) is a subset of the faces of $\cF_B^\Psi$, implies that $\cF_B^\Psi$ is $*$-connected, as claimed.
As a $*$-connected set of at most $CM$ faces, colored by \blue\ and \red\ and rooted at $x+v_t-v_{j^*+1}$ (read off of our choice of $\cF_A^\cI$ and $\cJ$), there are at most $(2s)^{CM}$ choices for $\cF_B^\Psi$. In the case of $\tilde \Xi_{x,\ell}$ where this is $x+v_{\tau_\ell} - v_{j^* +1}$, recall that our choice of $\cF_B^\cI$ picked out the cut-point $v'$ from which we read off $v_{\tau_\ell}$, so the same bound holds.

\medskip
\noindent \emph{Number of possible decorations $(\Upsilon_w)$ to $w\in \cV(\cF_\iota^\gamma)$: }Let us first count the combined number of faces 
\begin{align}\label{eq:total-faces-to-decorate-with}
    \sum_{\iota \in \{A,B\}, \gamma \in \{\cI, \cJ, \Psi\}} \sum_{w\in \cV(\cF_\iota^\gamma)} |\Upsilon_w|
\end{align}
By construction, the sets $\fS_w$ are all disjoint. The number of \blue\ faces in total over all $\Upsilon_w$ is therefore at most $\sum_{z\in \bD} |W_z|\leq \sum_{z\in \bD} 2\fm(W_z) \leq 6\fm(\cI; \cJ)$ (as no wall is double counted). 

For each $w$, the number of \red\ faces added to $\Upsilon_w$ in $\cL_0$ in the ball of radius $\sqrt{N_\rho(u)}$ centered at an edge-or-face $u\in \rho(\fS_w)$ is at most, using Claim~\ref{clm:m(W)},
\[\sum_{w} \sum_{u\in \rho(\fS_w)} N_\rho(u)\leq \sum_{w}|\fS_w|\leq 2\fm(\bW)\leq 6 \fm(\cI;\cJ)\,.\]

By induction, every $\fS_w$ consists of the groups of walls of a nested sequence of walls. Indeed, when we allocate $\fS_w$, if $\tilde W$ already decorates some previously processed vertex $z$, then necessarily the entire group of walls of $\tilde W$ must also have been allocated to $\fS_z$, so the remainder is still the group of walls of a nested sequence of walls. For each $w$, denote this nested sequence of walls by $W_1^w \Subset W_2^w \Subset \ldots$.

Within every $\fS_w$, all close walls are connected, via the additional  \red\ faces in $\cL_0$ in balls of radius $\sqrt{N_\rho (u)}$ for $u\in \rho(\fS_w)$. Each of the $*$-connected components obtained in this way (whose \blue\ faces are precisely a group of walls) corresponds to some $W_i^w$ (say the innermost one it contains). Finally, by definition, $\rho(w)$ is interior to $W_1^w$. 
Therefore, to obtain a spanning tree of the face-set, we can include shortest paths of faces from $\rho(w)$ to $W_1^w$, and then from $W_i^w$ to $W_{i+1}^{w}$ for every $i$. This adds at most $\sum_{i}|W_i^w| \leq 2 \fm(\fS_w)$ many faces,
and the minimum spanning tree adds at most that many \red\ faces. 
Summing over all $w$, this last contribution (again by~\eqref{eq:m(I;J)-WxJ}) is also at most $2\sum_w \fm(\fS_w)\leq 6\fm(\cI;\cJ)$.

Altogether, we deduce that the total number of faces~\eqref{eq:total-faces-to-decorate-with} is at most $18\fm(\cI;\cJ) = 18 M$.

In order to enumerate over all such possible decorating face-sets, let us first decide how many of the $18M$ faces are allocated to each $\Upsilon_w$. For every $\iota\in\{A,B\}$ and $\gamma\in\{\cI,\cJ,\Psi\}$ we have that $\cV(\cF_\iota^\gamma) \leq 4CM$ by our bound on the number of faces in that set; in particular, there are at most $24CM$ vertices, between which we wish to partition at most $18M$ decorating faces. The number of such partitions is at most $$\binom{(24C + 18)M-1}{24CM-1} \leq 2^{(24C + 18)M}\,.$$ For each such partition, if $k_w$ is the number of decorating faces assigned to $w$ (so that $\sum_w k_w \leq 18M$) then we have $s^{k_w}$ choices for a $*$-connected decorating face subset rooted at $\rho(w)$, and $2^{k_w}$ choices of \red\ and \blue\ colors for these faces. Thus, the total number of choices for the decorating colored face subsets corresponding to this partition is at most $(2s)^{18M}$.

\medskip
Multiplying all of the above enumerations yields the desired bounds for some $C_\Psi$ and $\tilde C_\Psi$. 
\end{proof}

\subsection{Proofs of Theorem~\ref{thm:exp-tail-base} and Proposition~\ref{prop:exp-tail-all-increments}}

\begin{proof}[\textbf{\emph{Proof of Theorem~\ref{thm:exp-tail-base}}}]
Recall from Claim~\ref{clm:m(W)} that for every $t$ and every $\cI$, we have that 
$$\fm(\cI; \Psi_{x,t}(\cI))\geq 2(\hgt(v_1)- \tfrac 12)+2d_\rho(v_1, x) \geq 2|v_1 - (x+(0,0, \tfrac 12))|$$ as well as $\fm(\cI;\Psi_{x,t}(\cI)) \geq  \fm(\sB_x)$ and $\fm(\cI;\Psi_{x,t}(\cI))\geq \fm(\sX_t)$ so that it suffices to prove that
  \begin{align*}
       \mu_{n}^{\mp} \big(\fm(\cI; \Psi_{x,t}(\cI)) \geq r\mid \bI_{x,T,H}\big) \leq C\exp\big[-  (\beta-C) (r\wedge d(x,\partial \Lambda_n))\big]\,.
  \end{align*} 
To see this, recall the definition of $\bar{\bI}_{x,T,H}$, and
express $\mu_{n}^{\mp} (\fm(\cI; \Psi_{x,t}(\cI)) \geq r,  \bI_{x,T,H})$ as at most 
\begin{align*}
    \mu_n^\mp (\bar{\bI}_{x,T,H}^c \mid \bI_{x,T,H}) \mu_n^\mp(\bI_{x,T,H}) + \mu_n^\mp(\fm(\cI; \Psi_{x,t}(\cI)) \geq r,  \bar \bI_{x,T,H})\,.
\end{align*}
The first term above is bounded by $e^{ - 2(\beta - C)d(x,\partial \Lambda_n)} \mu_n^\mp(\bI_{x,T,H})$ by Proposition~\ref{prop:tameness}, so let us turn to the second term: for every $r\geq 1$,
\begin{align*}
    \sum_{M\geq r}\,\, \sum_{\substack{\cI\in \bar \bI_{x,T,H}\\ \fm(\cI;\Psi_{x,t}(\cI)) = M}} \mu_n^\mp (\cI) & \leq \sum_{M\geq r} \sum_{\substack{\cI\in \bar \bI_{x,T,H}\\ \fm(\cI;\Psi_{x,t}(\cI)) = M}} e^{ - (\beta- C)M} \mu_n^\mp(\Psi_{x,t}(\cI))  \\
    & = \sum_{M\geq r}\,\, \sum_{\cJ\in \Psi_{x,t}(\bar \bI_{x,T,H})}  \mu_n^\mp(\cJ) \,\, \sum_{\substack{\cI\in \Psi_{x,t}^{-1}(\cJ) \\ m(\cI; \Psi_{x,t}(\cI))=M}} e^{-(\beta - C)M} \\
    & \leq \sum_{M\geq r} C_\Psi^M e^{- (\beta- C)M} \mu^\mp_n(\Psi_{x,t}(\bar \bI_{x,T,H}))\,.
\end{align*}
In the first inequality above, we used  Proposition~\ref{prop:partition-function-contribution} (and the fact that $\bar \bI_{x,T,H}$ is in the domain of $\Psi_{x,t}$ by Proposition~\ref{prop:well-defined}) and in the second inequality, we used Proposition~\ref{prop:multiplicity}. 

Now, noting by Proposition~\ref{prop:well-defined} that 
\[\Psi_{x,t}(\bar \bI_{x,T,H}) \subset \bar \bI_{x,T,H}
\]
we deduce that
\[
\mu_n^\mp(\fm(\cI; \Psi_{x,t}(\cI)) \geq r,  \bar \bI_{x,T,H})\leq C e^{-(\beta -C -\log C_\Psi) r}\mu_n^\mp(\bar \bI_{x,T,H}) \leq C e^{-(\beta -C -\log C_\Psi) r}\mu_n^\mp( \bI_{x,T,H})\,. 
\]
Combining these estimates and dividing through by $\mu_n^\mp(\bI_{x,T,H})$ then yeilds the desired conditional bound. 
\end{proof}

\begin{proof}[\textbf{\emph{Proof of Proposition~\ref{prop:exp-tail-all-increments}}}]
By definition of $\tilde \Psi_{x,\ell}$, for every half-integer $\ell$, we have 
\[\fm(\cI; \tilde \Psi_{x,\ell}(\cI))\geq |\cF(\cP_x \cap \cL_\ell)|-4\,.
\]
Indeed this follows from the fact that if $\cP_x\cap \cL_\ell$ is not a single cell, either it is part of an increment in the spine, in which case $\tau_\ell$ is the index of that increment and we use $\fm(\cI;\tilde \Psi_{x,\ell}(\cI))\geq \fm(\sX_{\tau_\ell})$, or it is part of the base, in which case this follows from $\fm(\cI;\tilde \Psi_{x,\ell}(\cI))= \fm(\cI; \Psi_{x,1}(\cI))\geq \fm(\sB_x)$.

As such, it suffices for us to show that for every $r\geq 1$,
  \begin{align*}
       \mu_{n}^{\mp} \big(\fm(\cI; \tilde \Psi_{x,\ell}(\cI)) \geq r\mid \bI_{x,T,H}\big) \leq C\exp\big[-  (\beta-C) (r\wedge d(x,\partial \Lambda_n))\big] 
  \end{align*} 
Arguing as in the proof of Theorem~\ref{thm:exp-tail-base}, we bound the left-hand side above by Proposition~\ref{prop:tameness} by
\begin{align*}
    e^{-2(\beta - C) d(x,\partial \Lambda_n)}\mu_n^\mp(\bI_{x,T,H})+ \mu_n^\mp(\fm(\cI; \tilde \Psi_{x,\ell}(\cI))\geq r, \bar{\bI}_{x,T,H})\,.
\end{align*}
The second term above is then bounded as 
\begin{align*}
    \sum_{M\geq r} \sum_{\substack {\cI \in \bar{\bI}_{x,T,H} \\ \fm(\cI; \tilde \Psi_{x,\ell}(\cI)) = M}} \mu_n^\mp(\cI) &  \leq \sum_{M\geq r} \sum_{\cJ\in \tilde \Psi_{x,\ell}(\bar {\bI}_{x,T,H})} \mu_n^\mp(\cJ) \sum_{\substack{\cI \in \tilde \Psi_{x,\ell}^{-1}(\cJ) \\ \fm(\cI; \cJ)= M}} e^{ - (\beta - C)M} \\ 
    & \leq \sum_{M\geq r} \tilde C_\Psi^M e^{ - (\beta- C) M} \mu_n^\mp (\tilde \Psi_{x,\ell}(\bar{\bI}_{x,T,H}))\,.
\end{align*}
Then, noting by Proposition~\ref{prop:well-defined}, that for every $\cI \in \bar{\bI}_{x,T,H}$, for every $\ell$, $\Psi_{x,\ell}(\cI)\in \bar{\bI}_{x,T,H}$, implies that  
\begin{align*}
    \tilde \Psi_{x,\ell}(\bar{\bI}_{x,T,H})\subset \bar{\bI}_{x,T,H}\,.
\end{align*}
We then deduce that 
\begin{align*}
    \mu_n^\mp(\fm(\cI; \tilde \Psi_{x,\ell}(\cI))\geq r, \bar{\bI}_{x,T,H}) \leq Ce^{ - (\beta - C - \log \tilde C_\Psi)r}\mu_n^\mp(\bar{\bI}_{x,T,H})  \leq Ce^{ - (\beta - C - \log \tilde C_\Psi)r}\mu_n^\mp({\bI}_{x,T,H})\,,
\end{align*}
at which point, dividing through by $\mu_n^\mp(\bI_{x,T,H})$ implies the desired. 
\end{proof}

\section{A refined sub-multiplicativity bound}\label{sec:submult}
In order to establish tightness for the centered maximum $M_n$ under $\mu_n^\mp$, we need to replace the approximate sub-multiplicativity bound on $\mu_n^\mp(\hgt(\cP_x)\geq h_1+h_2)$ obtained in~\cite{GL19a}---which had an $O(\exp(\log^2 (h_1 \vee h_2)))$ multiplicative error term---by one in which the multiplicative error is only $O(1)$. We will in fact show this with an error that is $1+\epsilon_\beta$ for some sequence $\epsilon_\beta>0$ that vanishes as $\beta\to\infty$.

Let us denote by $A_h^x$ the event, measurable with respect to the configuration on $\cC(\Z^2 \times \llb 0,h\rrb)$, given by 
\begin{align}\label{eq:Ax-def}
 A_h^x = \left\{\sigma: x+(0,0,\tfrac 12) \xleftrightarrow[\Z^2\times\llb 0,h\rrb]{+} \cL_{h-\frac 12} \right\}\qquad\mbox{for $x\in\cF(\cL_0)$}\,,
\end{align}
so that, recalling the definition~\eqref{eq:alpha-alpha-h-def} of $\alpha_h$, we have
\[ \alpha_h = -\log \mu_{\Z^3}^{\mp}(A_h^o)\qquad\mbox{where}\qquad o=(\tfrac12,\tfrac12,0)\,.\]
Further let
\begin{align}\label{eq:Ex-def}
 E_h^x = \left\{\sigma\,:\; \hgt(\cP_x) \geq h \right\} \qquad\mbox{for $x\in\cF(\cL_0)$}\,.
\end{align}
We showed in~\cite[Eq.~(6.3)]{GL19a} (this will also follow from Claim~\ref{clm:equality-of-events-cond} below) that for  $n$ large and every $x\in \cL_{0,n}$,
\begin{align}\label{eq:equality-of-events}
(1-\epsilon_\beta) \mu^\mp_n(A_h^x) \leq \mu^\mp_n(E_h^x) \leq (1+\epsilon_\beta) \mu^\mp_n (A_h^x)\,,
\end{align}
where $\epsilon_\beta\to 0$ as $\beta\to\infty$, and the same also applies under $\mu_{\Z^3}^\mp$. Thus, for another such sequence $\epsilon_\beta$,
\begin{align}\label{eq:tilde-alpha-def} \widetilde\alpha_h = -\log \mu_{\Z^3}^{\mp}(E_h^o)\qquad\mbox{satisfies}\qquad
 \left|\widetilde\alpha_h - \alpha_h \right| \leq \epsilon_\beta
\,,\end{align}
and we recall from Proposition~\ref{prop:Ex-bounds} that, for some $C>0$ and every $h\geq 1$,
\[ 4\beta-C \leq -\frac1h\log \mu_n^\mp(E_h^o) \leq 4\beta + e^{-4\beta}+\frac{C}h\,.\]
(The existence of $\alpha=\lim_{h\to\infty}\alpha_h/h$, as established in the prequel~\cite{GL19a}, implies that $4\beta-C\leq\alpha\leq 4\beta+e^{-4\beta}$. Our next results will rederive the limit and give it a more accurate description---see Corollary~\ref{cor:alpha-super-additive} below.)
The inequalities in~\eqref{eq:equality-of-events} tie the approximate sub-multiplicativity of $\{\mu_{n}^{\mp}(E^o_h)\}_{h\geq 1}$ to that of $\{\mu_{\Z^3}^{\mp}(A^o_h)\}_{h\geq 1}$ (equivalently, the super-additivity of $\widetilde\alpha_h$ to that of $\alpha_h$), which the following result establishes.
\begin{proposition}\label{prop:submult} There exists $\beta_0>0$ such that for every $\beta>\beta_0$, every $h=h_1+h_2$ for $h_1,h_2\geq 1$ which may depend on $n$, and every $x,x_1,x_2\in\cL_{0,n}$ such that $d(x,\partial \Lambda_n)\gg h$ and $d(x_i,\partial\Lambda_n)\gg h$ for $i=1,2$,
\begin{align}\label{eq:submult-A_h}
 \mu_n^\mp(A^{x}_h) \leq (1+\epsilon_\beta)\, \mu_n^\mp(A_{h_1}^{x_1}) \mu_n^\mp(A_{h_2}^{x_2}) \,,
\end{align}
where $\epsilon_\beta>0$ vanishes as $\beta\to\infty$. Consequently, for another such sequence $\epsilon_\beta>0$, 
\begin{align}\label{eq:submult-E_h}
\mu_n^{\mp}(E^x_h) \leq (1+\epsilon_\beta) \, \mu_n^{\mp}(E^{x_1}_{h_1}) \mu_n^{\mp}(E^{x_2}_{h_2}) \,.
\end{align}
\end{proposition}
In light of the preceding inequalities, the above proposition readily implies the following corollary.

\begin{corollary}\label{cor:alpha-super-additive} There exists $\beta_0$ such that for every $\beta>\beta_0$ and every $h_1,h_2\geq 1$,
\begin{equation}\label{eq:super-additivity}  \alpha_{h_1}+\alpha_{h_2} - \epsilon_\beta \leq \alpha_{h_1+h_2} \leq \alpha_{h_1} +(4\beta+e^{-4\beta}) h_2\end{equation}
for some  $\epsilon_\beta>0$ that vanishes as $\beta\to\infty$.
In particular, the limit $\alpha = \lim_{h\to\infty} \alpha_h /h$ ($=\lim_{h\to\infty}\widetilde\alpha_h/h$) exists and satisfies $\alpha= \sup_h(\alpha_h-\epsilon_\beta)/h$.
\end{corollary}

\begin{proof}
For the left-hand side, fix any $h_1,h_2$, take $x,x_1,x_2=o$ and send $n\to\infty$ (whence $\mu_n^\mp(A_h^o)\to\mu_{\Z^3}^\mp(A_h^o)$ is by the weak convergence of $\mu_n^\mp$ to $\mu_{\Z^3}^\mp$ as established by Dobrushin, and the analogous fact for $E_h^o$ follows from Corollary~\ref{cor:pillar-dependence-on-volume}).

Similarly, for the right-hand side it suffices for us to prove, in the setting of Proposition~\ref{prop:submult}, that 
\begin{align*}
    \mu_n^\mp(A_h^o)\geq \mu_n^\mp(A_{h_1}^{o})e^{-(4\beta +e^{-4\beta})h_2}\,,
\end{align*}
and then send $n\to\infty$. Reveal the entire configuration $\sigma$ on $\Lambda_n \cap \cL_{<h_1}$ under $\mu_n^\mp (\cdot \mid A_{h_1}^{o})$. On the event~$A_{h_1}^o$, in the configuration we revealed, $o+(0,0,\frac 12)$ is connected by plus sites to $\cL_{h_1 - \frac 12}$: call $Y-(0,0,\frac 12)$ an arbitrary site in the plus cluster of $o+(0,0,\frac 12)$ at height $h_1 - \frac 12$. Then, by a simple calculation (see, e.g., the proof in~\cite{GL19a} of the similar left-hand of Proposition~\ref{prop:Ex-bounds}), independently of the configuration outside of the set of cells $\{Y+(0,0,\frac k2):k=1,\ldots,h_2\}$, the probability that those $h_2$ cells are all plus---and therefore~$A_{h}^o$ holds---is at least $e^{-(4\beta +4e^{4\beta})h_2}$.
\end{proof}

We will need the following comparison which will imply the inequality~\eqref{eq:equality-of-events} that was stated above.
\begin{claim}\label{clm:equality-of-events-cond}
There exist $\beta_0>0$ and a sequence $\epsilon_\beta>0$ vanishing as $\beta\to\infty$ such that for all $\beta>\beta_0$, every~$n$ and every $x\in\cL_{0,n}$,
\begin{align*} \mu_n^\mp(E_h^x \mid A_h^x) &\geq 1-\epsilon_\beta\,,\qquad\mbox{ and }\qquad
 \mu_n^\mp(A_h^x \mid E_h^x) \geq 1-\epsilon_\beta\,.
 \end{align*}
\end{claim}
\begin{proof}
Let us begin with the first inequality: by definition, we have for every $x\in \cL_{0,n}$ and $h\geq 1$,  
\[(A_h^x \cap \mbox{$\bigcap_{k\geq 1}$}\{x-(0,0, \tfrac k2)\in \sigma(\cI)\}) \subset (E_h^x\cap A_h^x)\,.
\]
By Proposition~\ref{prop:dobrushin-exp-tail}, $x$ is not interior to any wall of $\cI$, and therefore $\bigcap_{k\geq 1} \{x-(0,0,\frac k2)\in \sigma(\cI)\}$, except with probability $\epsilon_\beta$ going to zero as $\beta \to\infty$. Then by the FKG inequality, we deduce that 
$$\mu_n^\mp(E_h^x\cap A_h^x) \geq \mu_n^\mp(A_h^x \cap \mbox{$\bigcap_{k\geq 1}$}\{x-(0,0, \tfrac k2)\in \sigma(\cI)\}) \geq (1-\epsilon_\beta)\mu_n^\mp(A_h^x)$$
implying the desired after dividing through by $\mu_n^\mp(A_h^x)$.

Let us turn to the second inequality. We can expose the inner  boundary of the plus $*$-connected component of the cell-set $\sigma(\cI)$ under the measure $\mu_n^\mp(\cdot \mid E_h^x)$ as follows: reveal the entire minus $*$-connected component of the boundary $\partial \Lambda \cap \cL_{>0}$ by exposing the minus $*$-connected component and all $*$-adjacent (bounding) plus spins---exactly the plus (inner) boundary of the $*$-connected plus component of $\sigma(\cI)$ together with (inner) boundaries of finite plus bubbles in the minus phase. Since $x+(0,0,\frac 12)$ is plus in $\sigma(\cI)$ (on the event $E_h^x$), and for every $\cI$, the measure of $\mu_n^\mp(\cdot)$ on $\{v:\sigma(\cI)_v = +1\}$ stochastically dominates that induced by $\mu_{\Z^3}^+$ on $\{v:\sigma(\cI)_v = +1\}$, we have 
\begin{align*}
    \mu_n^\mp(A_h^x \mid E_h^x) &  \geq \inf_{\cI\in E_h^x} \mu_{\{v:\sigma(\cI)_v=+1\}}^+(x+(0,0,\tfrac 12) 
    \xleftrightarrow[\cL_{>0}]{+}
     \cS_x) \\
    & \geq \mu_{\Z^3}^+(x+(0,0,\tfrac 12)
    \xleftrightarrow[\cL_{>0}]{+}
    \infty)\,;
\end{align*}
this is seen to be at least $1-\epsilon_\beta$ by the classical Peierls argument. 
\end{proof}
The proof of Proposition~\ref{prop:submult} follows the same argument that was used to establish the weaker approximate sub-multiplicativity bound in~\cite{GL19a}, whereas here the error terms can be better tracked and controlled via the improved estimates on the shape of the pillar (in particular the exponential tail on the size of the base conditioned on the pillar reaching height $h$, vs.\ the bound in the prequel which had an extra $O(\log h)$~term).
For completeness, we include the full argument instead of only listing the needed modifications. We begin with recalling several decorrelation estimates for pillars which are needed for the proof.

\begin{proposition}[{\cite{Dobrushin72b}, \cite[Lemma 5]{Dobrushin73}, as well as \cite[Prop.~2.3]{BLP79b}}]\label{prop:dependence-on-volume}
There exists $C>0$ such that for every $\beta>\beta_0$, every $n\leq m$, for any subset $F \subset \cL_{0,n}$, 
\begin{align*}
\big\|\mu^\mp_n\big((\sF_y)_{y\in F} \in \cdot \big) - \mu^\mp_m\big((\sF_y)_{y\in F} \in \cdot \big)\big\|_{\tv} \leq C e^{- d(F,\partial \Lambda_n)/C}\,.
\end{align*}
In particular, sending $m$ to $\infty$, and via the tightness of $(\sF_y)_{y\in F}$, this holds if we replace $\mu_m$ by $\mu_{\Z^3}^{\mp}$. 
\end{proposition}

\begin{proposition}[{\cite{Dobrushin73}, see also~\cite[Proposition 2.1]{BLP79b}}]\label{prop:group-of-wall-correlations}
There exist $\beta_0>0$ and $C>0$ such that for every $\beta>\beta_0$, every $n$ and every two subsets $F_1,F_2 \subset \cL_{0,n}$, 
\begin{align*}
&\left\|\mu^\mp_n \left((\sF_y)_{y\in F_1}\in\cdot, (\sF_z)_{z\in F_2} \in\cdot\right )  - \mu^\mp_n \big((\sF_y)_{y\in F_1}\in \cdot\big)  \mu_n^\mp \big((\sF_z)_{z\in F_2} \in \cdot \big) \right\|_\tv  \leq C e^{-d(F_1,F_2)/C}\,.
\end{align*}
\end{proposition}  
Propositions~\ref{prop:dependence-on-volume}--\ref{prop:group-of-wall-correlations} readily translate to similar estimates on the collections of pillars (see the short proofs of Corollaries 6.4 and 6.6 in~\cite{GL19a}, addressing the special cases  where $F, F_1,F_2$ were balls of radius $r$ about some fixed faces $x_n,y_n \in \cL_{0,n}$, and following from the respective special cases of the above propositions).
\begin{corollary}[see {\cite[Corollary~6.4]{GL19a}}]\label{cor:pillar-dependence-on-volume}
There exist $\beta_0>0$ and $C>0$ such that for every $\beta>\beta_0$, every subset $F_n\subset \cL_{0,n}$ and every subset $F_m \subset \cL_{0,m}$ which is a horizontal translation of $F_n$,
\begin{align*}
\|\mu^\mp_n \big((\cP_{y})_{y\in F_n} \in \cdot\big) - \mu^\mp_m \big((\cP_y)_{y\in F_m} \in \cdot \big) \|_{\tv} \leq C e^{-(d(F_n, \partial \Lambda_n) \vee d(F_m,\partial \Lambda_m))/C}\,.
\end{align*}
\end{corollary}

\begin{corollary}[see {\cite[Corollary~6.6]{GL19a}}]\label{cor:pillar-correlations}
There exist $\beta_0>0$ and some $C>0$ such that for every $\beta>\beta_0$, every $n$ and every two subsets $F_1, F_2 \subset \cL_{0,n}$, 
\begin{align*}
\|\mu^\mp_n ((\cP_y)_{y\in F_1}\in \cdot, (\cP_y)_{y\in F_2} \in \cdot) - \mu^\mp_n ((\cP_y)_{y\in F_1} \in \cdot)\mu^\mp_n ((\cP_y)_{y\in F_2}\in \cdot) \|_\tv \leq C e^{- d(F_1,F_2)/C}\,.
\end{align*}
\end{corollary}

\begin{proof}[\textbf{\emph{Proof of Proposition~\ref{prop:submult}}}]
Recall that $h=h_1+h_2$, and suppose without loss of generality that $h_1\geq h_2$.
Define the vertical shift of the event $A^x_h$, for a given vertex $x$, as
\begin{align*}
\theta_{h_1} A_{h_2}^x = \left\{ \sigma: x+(0,0,h_1+\tfrac 12)
\xleftrightarrow[\Z^2 \times \llb h_1, h_1+h_2\rrb]{+} \cL_{h_1 + h_2- \frac 12} \right\}\,,
\end{align*}
noting that $\mu_{n}^{\mp} (\theta_{h_1} A^x_{h_2})= \mu_{n}^{\mp, -h_1}(A^x_{h_2})$, where we denote by $(\mp,-m)$ boundary conditions are those that are plus on $\partial \Lambda \cap \cL_{<-m}$ and minus on $\partial \Lambda\cap \cL_{>-m}$. Hence, for every $n$, by monotonicity in boundary conditions, 
\begin{align*}
\mu_{n}^{\mp, -h_1}(A^x_{h_2}) \leq \mu^\mp_{n} (A^x_{h_2})\,.
\end{align*}
A naive approach to establishing sub-multiplicativity would be to expose the plus $*$-component of $x+(0,0,\frac12)$ in the slab $\Z^2\times[0,h_1]$, wherein the measurable event $A_{h_1}^x$ guarantees that $x+(0,0,\frac12)$ is $*$-connected to $\cL_{h_1-\frac12}$, and as the tip of the $*$-component at $\cL_{h_1-\frac12}$ is now situated in the minus phase, the conditional probability of $A_{h}^x$ should be at most that of the unconditional $A_{h_2}^x$. However, revealing the plus $*$-component introduces some positive information (the connection event $A_{h_1}^x$ is increasing) along with negative information (e.g., minus spins along its boundary). We will control this using our new estimates on the shape of the pillar.

Denote by $\cA$ the $*$-connected plus component of $x+(0,0,\frac 12)$ in $\llb -n_{h},n_{h} \rrb^2 \times \llb 0,h_1\rrb$ (noting that the event~$A_{h_1}^x$ merely says that $\cA$ intersects the slab $\cL_{h_1-\frac12}$). An important fact which we will use later on is that, on the event that $\cP_x \neq \emptyset$, this plus $*$-component $\cA$ is a subset of the plus sites in $\cP_x$. 

\begin{definition}\label{def:sA}
Let $\sA$ denote the set of possible realizations of $\cA$ that satisfy the following properties:
\begin{enumerate}
\item The intersection of $\cA$ with $\cL_{\frac 12}$ is the single cell whose lower bounding face is $x$.
\item The intersection of $\cA$ with $\cL_{h_1-\frac 12}$ is a single cell; denote its upper bounding face by $Y\in\cL_{h_1}$;
\item We further have $d(x,\rho(Y)) \leq d(x,\partial\Lambda_n)/2$.  
\end{enumerate}
Let $\Gamma$ denote the (neither increasing nor decreasing) event $\{\cA\in\sA\}$, noting that $\Gamma\subset A_{h_1}^x$.
\end{definition}

\begin{claim}\label{claim:submult-1}
In the setting of Proposition~\ref{prop:submult}, there exists a sequence $\epsilon_\beta$ going to zero as $\beta \to\infty$ such that
\begin{align*}
\mu_n^\mp(A^x_h,\, \Gamma) \leq  (1+\epsilon_\beta) \mu^\mp_n(A_{h_1}^{x_1}) \mu^\mp_n (A_{h_2}^{x_2})\,.
\end{align*}
\end{claim}

\begin{claim}\label{claim:submult-2}
In the setting of Proposition~\ref{prop:submult}, there exists a sequence $\epsilon_\beta$ going to zero as $\beta \to\infty$ such that
\begin{align*}
\mu^\mp_n (A^x_h) \leq (1+\epsilon_\beta)\mu^\mp_n(A^x_h,\, \Gamma) \,.
\end{align*}
\end{claim}

\begin{proof}[\emph{\textbf{Proof of Claim~\ref{claim:submult-1}}}] 
Set 
\[r = \min\{d(x_1,\partial\Lambda_n)\,,\,d(x_2,\partial\Lambda_n)\,,\,d(x,\partial\Lambda_n)\}\,,\]
and recall that $r \gg h$ by assumption; thus, Corollary~\ref{cor:pillar-dependence-on-volume} (for $m=n$) implies that for each $i=1,2$,
\[ \|\mu^\mp_n(\cP_x\in\cdot) - \mu^\mp_n(\cP_{x_i}\in\cdot)\|_{\tv} \leq C \exp(-r/C)
\]
for some $C$ independent of $\beta$. 
By relating the events $A_{h_i}^x$ and $E_{h_i}^x$ via~\eqref{eq:equality-of-events}, we then obtain that
\begin{align} \mu^\mp_n(A^x_{h_i}) \leq (1+\epsilon_\beta)\mu^\mp_n(E_{h_i}^x) &\leq  (1+\epsilon_\beta)\mu^\mp_n(E^{x_i}_{h_i}) + C\exp(-r/C) \nonumber\\
&\leq (1+\epsilon_\beta)^2\mu^\mp_n(A_{h_i}^{x_i}) + C\exp(-r/C) \nonumber\\
&\leq (1+o(1))(1+\epsilon_\beta)^2 \mu^\mp_n\!(A_{h_i}^{x_i})\label{eq:mu-n_h-mu-n_h_i}
\end{align}
using $\mu^\mp_n(A_{h_i}^{x_i}) \geq \exp(-(4\beta+ e^{-4\beta}) h_i)$ and $r \gg h_i$ by our hypothesis.
Thus, the claim will follow once we show that, for some other sequence $\epsilon_\beta>0$ that vanishes as $\beta\to\infty$,
\begin{align}\label{eq:claim-submult-1-reduction1}
\mu^\mp_n (A^x_h,\,\Gamma)\leq (1+\epsilon_\beta) \mu^\mp_n(A_{h_1}^x)\mu^\mp_n(A_{h_2}^x)\,.
\end{align}
Recall that $\Gamma\subset A_{h_1}^{x}$, whence
\[ \mu^\mp_{n}(A^{x}_{h},\, \Gamma)  = \mu^\mp_n (\Gamma) \mu^\mp_{n}\big(A_{h}^{x}\mid \Gamma) 
\leq \mu^\mp_n (A^{x}_{h_1}) \mu^\mp_{n}\big(A_{h}^{x}\mid \Gamma)\,,
\]
so in order to establish~\eqref{eq:claim-submult-1-reduction1} it remains to show that
\begin{equation}\label{eq:claim-submult-1-reduction2} \mu_{n}^\mp\big(A_{h}^{x}\mid \Gamma) \leq (1+\epsilon_\beta) \mu_n^\mp(A_{h_2}^{x})\,.
\end{equation}
Denoting by $\E_\Gamma$ expectation w.r.t.\ $\mu_{n}^\mp(\cdot\mid \Gamma)$, whereby $\cA$ accepts values in $\sA$, we have that
\begin{align*}
    \mu_{n}^\mp(A^{x}_h\mid \Gamma) &= \E_\Gamma\left[ \mu_{n}^\mp (A_h^{x}\mid \cA) \right] = \E_\Gamma\left[ \mu_{n}^\mp(\theta_{h_1}A_{h_2}^{Y}\mid\cA)\right]\,,
\end{align*}
since $\Gamma$ stipulates that $Y-(0,0,\frac12)$ is the unique cell in the intersection of $\cA$ with $\cL_{h_1-\frac12}$, whence $Y+(0,0,\frac12)$ must  then be $*$-connected to $\cL_{h-\frac12}$ in $\cC(\Z^2\times \llb h_1,h\rrb)$ in order for $A_{h}^{x}$ to occur. 

Reveal $\cA\in\sA$, the $*$-connected plus component of $x+(0,0,\frac12)$, only in the slab $\cC(\llb -n,n\rrb^2 \times \llb 0,h_1\rrb)$. A key observation now is that the boundary spins of the $*$-component revealed in this manner are plus at $x+ (0,0,\frac12)$ and $Y-(0,0,\frac12)$, and minus on all sites in $\Lambda_n \setminus \cA$ that are $*$-adjacent to $\cA$. Indeed, the boundary conditions on $\partial \Lambda_n$ are all minus between heights $[0,h_1]$, and our definition of $\sA$ forces every $*$-adjacent spin in $\cL_{>0}\cap \cL_{<h_1}$ to be minus except at $x+ (0,0,\frac12)$ and $Y-(0,0,\frac12)$ (those are plus as per $\sA$). Denoting by $\partial\cA$ these boundary spins, and by $(\mp,\partial\cA)$ their addition to our Dobrushin boundary conditions,
the domain Markov property implies that
\[ \E_\Gamma\left[ \mu_{n}^\mp(\theta_{h_1}A_{h_2}^{Y}\mid\cA)\right] = \E_\Gamma\left[\mu_{n}^{\mp,\partial\cA}(\theta_{h_1}A_{h_2}^Y)\right] \leq \sup_{\cA\in\sA} \mu_{n}^{\mp,\partial\cA}(\theta_{h_1}A_{h_2}^Y) \,.\]
For a fixed $\cA\in\sA$ (hence fixed $Y$), in view of the above fact that $\partial\cA$ includes plus spins only at $x+(0,0,\frac12)$ and $Y-(0,0,\frac12)$, 
the FKG inequality w.r.t.\ the Ising measure conditioned on $\sigma_{x+(0,0,\frac12)}=\sigma_{Y-(0,0,\frac12)}=+1$ enables us to omit the conditioning on its minus spins and obtain that
\begin{align*} \sup_{\cA\in\sA} \mu_{n}^{\mp,\partial\cA}(\theta_{h_1}A_h^Y) &\leq  \sup_{\cA\in\sA} \mu_{n}^\mp\left(\theta_{h_1}A_h^Y \mid \sigma_{x+(0,0,\frac12)}=\sigma_{Y-(0,0,\frac12)}=+1\right) \\ 
&\leq 
 \sup_{\cA\in\sA}
\mu_{n}^{\mp, -h_1}\left(A^{\rho(Y)}_{h_2}\mid \sigma_{x+(0,0,-h_1+\frac12)}=\sigma_{\rho(Y)-(0,0,\frac12)}=+1\right)
\end{align*}
by translation. Another application of FKG---now for monotonicity in boundary conditions---allows us to move from $\mu_{n}^{\mp,-h_1}$ to $\mu_{n}^\mp$, and conclude that the last expression is at most
\begin{align*} \sup_{\cA\in\sA}
\mu_{n}^{\mp}\left(A^{\rho(Y)}_{h_2}\mid \sigma_{x+(0,0,-h_1+\frac12)}=\sigma_{\rho(Y)-(0,0,\frac12)}=+1\right) \leq 
(1+\epsilon_\beta)^2 \sup_{\cA\in\sA} 
\mu_{n}^{\mp}\left(A^{\rho(Y)}_{h_2}\right)
\,,
 \end{align*}
where the last inequality is justified as follows. For a fixed face $u\in\cL_{0,n}$ (here we would take $u=\rho(Y)$ for a worst-case realization of $\cA$), if we denote $B_x=\{\sigma_{x+(0,0,-h_1+\frac12)}=+1\}$ and $B_u=\{\sigma_{u-(0,0,\frac12)}=+1\}$, then $\mu_{n}^\mp(A_{h_2}^u \mid B_x\,,B_u) \leq \mu_{n}^\mp(A_{h_2}^u) / \mu_{n}^\mp(B_x\,, B_u)$; now, $\mu^\mp_{n}(B_x\,,B_u) \geq \mu_{n}^\mp(B_x)\mu_{n}^\mp(B_u)$ by FKG, thus it remains to show that each of the events $B_x$ and $B_u$ has probability at least $1-\epsilon_\beta$ under $\mu_{n}^\mp$. By the results of Dobrushin (see, e.g., Proposition~\ref{prop:dobrushin-exp-tail}), if $z=(z_1,z_2,z_3)$ is a fixed point $z_3<0$, then $(z_1,z_2,0)$ has no walls of the interface $\cI$ nesting it except with probability $\epsilon_\beta$. In particular, $z\in\sigma(\cI)$, whence a Peierls argument shows that with probability $1-\epsilon_\beta$ its spin is plus. 

Finally, for each $\cA\in\sA$,
deterministically $d(x,\rho(Y))\leq d(x,\partial\Lambda_n)/2$, so by the triangle inequality
\[ d(\rho(Y),\partial\Lambda_n) \geq d(x,\partial\Lambda_n)-d(x,\rho(Y)) \geq  d(x,\partial\Lambda_n)/2 \gg h
\]
by assumption. Thus, the same argument used to compare $x$ to $x_i$ in~\eqref{eq:mu-n_h-mu-n_h_i} shows that
\[ \mu_n^\mp(A_{h_2}^{\rho(Y)}) \leq (1+\epsilon_\beta)  \mu_n^\mp(A_{h_2}^x)\,,\]
establishing~\eqref{eq:claim-submult-1-reduction2} and thus concluding the proof.
\end{proof}

\begin{proof}[\emph{\textbf{Proof of Claim~\ref{claim:submult-2}}}] 
Writing 
\[  \mu_n^\mp(\Gamma\,,A_h^x\mid E_h^x) \leq \frac{\mu_n^\mp(\Gamma,\, A_h^x)}{\mu_n^\mp(E_h^x)} = \frac{\mu_n^\mp(A_h^x)}{\mu_n^\mp(E_h^x)} \mu_n^\mp(\Gamma \mid A_h^x) \leq (1+\epsilon_\beta) \mu_n^\mp(\Gamma \mid A_h^x)\,,\]
with the last inequality by~\eqref{eq:equality-of-events},
it remains to show  $\mu_n^\mp(\Gamma,\,A_h^x\mid E_h^x)\geq 1-\epsilon'_\beta$ for some other sequence~$\epsilon'_\beta>0$ vanishing as $\beta\to\infty$, which will altogether imply that $\mu_n^\mp(\Gamma,\, A_h^x) / \mu_n^\mp(A_h^x)\geq (1-\epsilon'_\beta)/(1+\epsilon_\beta)$, as required.
Using that $\mu_n^\mp((A_h^x)^c\mid E_h^x) \leq \epsilon_\beta$ by Claim~\ref{clm:equality-of-events-cond}, we have
\[\mu_n^\mp\left(\Gamma^c \cup (A_h^x)^c \mid E_h^x\right) \leq
\mu_n^\mp\left((A_h^x)^c \mid E_h^x\right) + \mu_n^\mp\left(\Gamma^c\,, A_h^x \mid E_h^x\right) \leq \epsilon_\beta + \mu_n^\mp\left(\Gamma^c\,, A_h^x \mid E_h^x\right)\,,
\]
and it remains to bound the last term in the right-hand by $\epsilon_\beta$. Examining the criteria for $\Gamma$ in Definition~\ref{def:sA}, observe that $A_h^x $ implies (through the inclusion $A_{h_1}^x \supset A_h^x$) that $\cA$ intersects both $\cL_{\frac12}$ and $\cL_{h_1-\frac12}$. Hence \[ \Gamma^c\cap A_h^x = (B_1 \cup B_2 \cup B_3)\cap A_h^x \qquad\mbox{where}\qquad 
\begin{cases}
B_1 = \left\{ |\cC(\cA \cap \cL_{\frac12})|\geq 2\right\}\,, \\
\noalign{\smallskip}
B_2 = \left\{ |\cC(\cA \cap \cL_{h_1-\frac12})|\geq 2\right\}\,, \\
\noalign{\smallskip}
B_3 = \Big\{ d(x,\rho(Y)) > \frac12 d(x,\partial\Lambda_n)\Big\}\,.
\end{cases}
\]
Since $E_h^x = \bI_{x,0,h}$, the aforementioned fact that $\cA$ is a subset of the plus spins in $\cP_x$ implies that
\[ \mu_n^\mp(B_1 \mid E_h^x) \leq \mu_n^\mp ( \sB_x \neq \emptyset \mid \bI_{x,0,h})\qquad\mbox{and}\qquad
\mu_n^\mp(B_2 \mid E_h^x) \leq \mu_n^\mp ( |\cC(\cP_x\cap \cL_{h_1-\frac12})|\geq 2 \mid \bI_{x,0,h})\,.\]
Theorem~\ref{thm:exp-tail-base} and Proposition~\ref{prop:exp-tail-all-increments} respectively show (using the hypothesis $h\ll d(x,\partial\Lambda_n)$) that these two probabilities are at most $\exp(-(\beta-C))$.

Finally, the event $ B_3$ implies that $d(x,\rho(y))>\frac12 d(x,\partial\Lambda_n)$ for some face $y\in\cP_x$. By~\eqref{eq:geometric-observation-tameness}, this implies $\diam(\sB_x)+ \frac14\fm(\cS_x) > \frac12 d(x,\partial\Lambda_n)$, whence by Proposition~\ref{prop:tameness},
\[ \mu_n^\mp( B_3 \mid \bI_{x,0,h}) \leq \mu_n^\mp(\diam(\sB_x)+ \tfrac14\fm(\cS_x) > \tfrac12 d(x,\partial\Lambda_n) \mid \bI_{x,0,h}) \leq C e^{-(\beta-C)d(x,\partial\Lambda_n)}=o(1)\,.\] 
Combined, we have that $\mu_n^\mp(\Gamma^c \cap A_h^x) \leq 2\exp(-(\beta-C))+o(1)$, as needed.
\end{proof}
Combining Claims~\ref{claim:submult-1}--\ref{claim:submult-2} concludes the proof. 
\end{proof}

\section{Tightness and exponential tails of the maximum}\label{sec:tightness}
In this section we prove left and right exponential tails for $M_n- m_n^\star$, as stated in the next proposition.  
\begin{proposition}\label{prop:exp-tails}
There exist $\beta_0>0$ and a sequence $\epsilon_\beta>0$ vanishing as $\beta\to\infty$ such that the following holds for all $\beta>\beta_0$.
Letting $\widetilde\alpha_h$ be as in~\eqref{eq:tilde-alpha-def} and $ m_n^\star$ be as in~\eqref{eq:mn-def}, for every $1\leq \ell \leq \sqrt{\log n}$,
\begin{align*}
\mu_n^\mp (M_n &\geq  m_n^\star +\ell) \leq  (1+\epsilon_\beta) \exp\left(2\beta- \widetilde\alpha_\ell\right)\,,\\
\mu_n^\mp (M_n &< m_n^\star - \ell) \leq (1+\epsilon_\beta) \exp\left(2\beta-\widetilde\alpha_\ell\right)\,. 
\end{align*}
\end{proposition}
Before proving this result, we will establish some preliminary estimates. 
Recalling that $m_n^\star$ is the first~$h$ such that $\alpha_h > 2\log(2n)-2\beta$, the relation between $\alpha_h,\widetilde\alpha_h$ in~\eqref{eq:tilde-alpha-def} and the bound $\alpha_{h+1}\leq\alpha_h+4\beta+\epsilon_\beta$ by~\eqref{eq:super-additivity} together imply that for the sequence $\widetilde\alpha_h$ we have
\begin{equation}
\label{eq:widetilde-m_n}
2\log(2n) - 2\beta - \epsilon_\beta \leq \widetilde\alpha_{m_n^\star} \leq 2\log(2n) + 2\beta + \epsilon_\beta\,.
\end{equation}
Next, recall that $E_h^x=\{\hgt(\cP_x)\geq h\}$ (see~\eqref{eq:Ex-def}), so that $\{M_n\geq h\}=\bigcup_{x\in\cL_{0,n}}E_h^x$; we will separate the analysis of $E_h^x$ for $x$ near and away from $\partial\Lambda_n$ as follows. Define the interior of $\cL_{0,n}$,
\begin{align} 
\cL_{0,n}^- = \cF\big(\llb -(n - \log^2 n) ,\,n - \log^2 n \rrb^{2}\times \{0\}\big)\,,\label{eq:L_0-minus-def} 
\end{align}
and observe that, by Corollary~\ref{cor:pillar-dependence-on-volume}, for $x\in\cL_{0,n}^-$ we can couple $\mu_n^\mp(\cP_x\in\cdot)$ to $\mu_{\Z^3}^\mp(\cP_o\in\cdot)$ and find that, for some fixed $c>0$,
\begin{align}\label{eq:Ehx-tilde-alpha-h} \mu_n^\mp(E_h^x) = \mu_{\Z^3}^\mp(E_h^o) + O(e^{-c\log^2 n}) = (1+o(1))e^{-\widetilde\alpha_h}
\quad\mbox{if $1\leq h \ll \log^2 n$}\,,
\end{align}
absorbing the $O(e^{-c\log^2 n})$ as $\widetilde\alpha_h \leq \alpha h + \epsilon_\beta $ by Corollary~\ref{cor:alpha-super-additive}, so 
$\mu_n^\mp(E_h^o) = e^{-\widetilde\alpha_h}=e^{-o(\log^2 n)}$ for $h\ll\log^2 n$.

Further recalling the definition of the event $A_h^x$ from~\eqref{eq:Ax-def}, for $h\in \Z_+$ let
\begin{align}\label{eq:G-def}
G^x_h =  A^x_h \cap
E^x_h \cap \{ \sB_x = \emptyset \}\,,
\end{align}
and define the counter 
\[Z_h = \sum_{x\in {\cL}_0^-} \one\{G_h^x\}\,.
\]

\begin{claim}\label{claim:expectation-lb}
There exist $\beta_0>0$ and a sequence $\epsilon_\beta>0$ vanishing as $\beta\to\infty$ such that that for every $\beta>\beta_0$ the following hold.
If $1\leq h \ll \log^2 n$ then for every $x\in\cL_{0,n}^-$ and large enough $n$,
\begin{align}\label{eq:G-vs-E} \mu_n^\mp(G^x_h) \geq (1-\epsilon_\beta)\mu_n^\mp(E_h^x)\,.\end{align}
Consequently, if $1 \leq h= m_n^\star-\ell$ for $\ell>0$ then 
\[ \E [Z_h] \geq (1-\epsilon_\beta)e^{\widetilde\alpha_{\ell}-2\beta}\,.
\]
\end{claim}
\begin{proof}
For every $x\in\cL_{0,n}^-$
we have
\[ \mu_n^\mp\left(G^x_h \mid E_{h}^x\right) \leq \mu_n^\mp((A_h^x)^c\mid E_h^x) + \mu_n^\mp (\sB_x\neq\emptyset\mid E_h^x)
\leq \epsilon_\beta + Ce^{-2\beta}\,,\]
using that $\mu_n^\mp((A_h^x)^c\mid E_h^x)\leq\epsilon_\beta$ by Claim~\ref{clm:equality-of-events-cond} and that $\mu_n^\mp (\sB_x\neq\emptyset\mid E_h^x)\leq \epsilon_\beta$ by part~(a) of Theorem~\ref{thm:exp-tail-base}, where we took $r=1$ and used that $E_h^x=\bI_{x,0,h}$ and $h \ll \log^2 n \leq d(x,\partial\Lambda_n$)
since $x\in\cL_{0,n}^-$; this yields~\eqref{eq:G-vs-E}.

For the second part of the claim, notice that $1 \leq h <  m_n^\star  = (\frac{2}\alpha+o(1))\log n$ by Corollary~\ref{cor:alpha-super-additive}. 
By~\eqref{eq:Ehx-tilde-alpha-h} (now $h=O(\log n)$) we have $\mu_n^\mp(E_h^x) =(1+o(1))e^{-\widetilde\alpha_h}$, and the super-additivity in Corollary~\ref{cor:alpha-super-additive} shows that
\begin{align*} \widetilde\alpha_h = \widetilde\alpha_{ m_n^\star-\ell} \leq \widetilde\alpha_{ m_n^\star} - \widetilde\alpha_{\ell} + \epsilon_\beta &\leq 2\log(2n) -\widetilde\alpha_\ell +2\beta +2\epsilon_\beta \,,\end{align*}
using~\eqref{eq:widetilde-m_n} for the last inequality.
Combining these, while noting that $|\cL_{0,n}^-|=(1-o(1))|\cL_{0,n}|= (4-o(1))n^2$, we obtain that the expectation of $Z_{-\ell}$ under $\mu_n^\mp$ satisfies
\[ \E[Z_h] \geq (1-o(1))(1-\epsilon_\beta) e^{\widetilde\alpha_{\ell}-2\beta-2\epsilon_\beta} \geq (1-\epsilon'_\beta) e^{\widetilde\alpha_{\ell}-2\beta}\,,
\]
for some other sequence $\epsilon'_\beta>0$ vanishing as $\beta\to\infty$.
\end{proof}

\begin{claim}\label{claim:effect-of-x-pillar}
There exist $\beta_0>0$ and a sequence $\epsilon_\beta>0$ vanishing as $\beta\to\infty$ such that for every $\beta>\beta_0$, every $1\leq h \ll \log^2 n$ and every $x \neq y\in \cL_{0,n}^-$ such that $d(x,y)\leq \log^2 n$, if $n$ is large enough then
\[\mu_n^\mp(G^y_h \mid G^x_h) \leq (1+\epsilon_\beta)e^{ 6\beta } \mu_n^\mp(G^y_h)\,.
\] 
\end{claim}
\begin{proof}
We use a a similar revealing procedure to that used in the proof of sub-multiplicativity above to reveal $\cP_x$, without obtaining too much positive information about $\cP_y$. Let $\cA_x$ be the $*$-connected plus component of $x+(0,0,\frac 12)$ in $\cL_{>0}$ and let $\cA_y$ be the $*$-connected plus component of $y+(0,0,\frac 12)$ in $\cL_{>0}$. If we reveal $\cA_x$ on the event $G_h^x$, we expose the plus $*$-component $\cA_x$ along with all $*$-adjacent (bounding) minus spins in $\cL_{>0}$. On the event $G_h^x$, whereby the first cut-point of $\cP_x$ is $x+(0,0,\frac 12)$, the exterior boundaries of $\cP_x$ and $\cA_x$ coincide, and therefore, the event $G_h^x$ is measurable with respect to the set of sites revealed in this manner. As such, we can express 
\begin{align*}
\mu_n^\mp(G_h^y \mid G_h^x) \leq \sup_{\cA_x\in G_h^x}  \mu_n^\mp (E_h^y \mid \cA_x)\,.
\end{align*}
The boundary sites revealed by $\cA_x \in G_h^x$ are all minus except a single plus site at $x+(0,0,\frac 12)$, and so by the FKG inequality and the fact that $E_h^y$ is an increasing event, this is at most 
\begin{align*}
\mu_n^\mp(E_h^y \mid \sigma_{x+(0,0,\frac 12)} = +1) \leq e^{6\beta} \mu_n^\mp (E_h^y)\,.
\end{align*}
The fact that $\mu_n^\mp (G_h^y) \geq (1-\epsilon_\beta) \mu_n^\mp (E_h^y)$ by~\eqref{eq:G-vs-E} concludes the proof.
\end{proof}

\begin{claim}\label{claim:covariance-bound}
There exists $\beta_0>0$ such that for every $\beta>\beta_0$ there is some $C>0$ such that  for every $x, y\in\cL_{0,n}^-$, we have
\begin{align*}
\left|\mu_n^\mp(G_h^x)\mu_n^\mp(G_h^y) - \mu_n^\mp(G_h^x,\,G_h^y)\right| \leq Ce^{-d(x,y)/C}\,.
\end{align*}
\end{claim}
\begin{proof}
Notice that the pair of events $G_h^x$ and $G_h^y$ are measurable with respect to the pair of walls $W_x,W_y$. This is because the bounding faces of the spine $\cS_x$ (respectively, $\cS_y$) are all part of the same wall as shown in Claim~\ref{clm:cut-point-property}, and the wall $W_x$ (resp., $W_y$) contains the four bounding faces of $x+(0,0, \frac 12)$ (resp., $y+(0,0, \frac 12)$). As such, we can bound the difference above as 
\[
|\mu_n^\mp(G_h^x) \mu_n(G_h^y) - \mu_n^\mp(G_h^x, G_h^y)| \leq \|\mu_n^\mp (W_x \in \cdot) \mu_n^\mp(W_y\in \cdot) - \mu_n^\mp (W_x\in \cdot, W_y\in \cdot)\|_\tv\,,
\]
which is at most $Ce^{- d(x,y)/C}$ by Proposition~\ref{prop:group-of-wall-correlations}.
\end{proof}

\subsection{Exponential tails for the maximum}
We are now ready to deduce that the centered maximum has left and right exponential tails (and is therefore tight).

\begin{proof}[\textbf{\emph{Proof of Proposition~\ref{prop:exp-tails}}}]
We begin with the right tail. Letting 
\[ h= m_n^\star+\ell\qquad\mbox{ for }\qquad 1\leq\ell \leq (1/\beta)\log n\]
(n.b.\ we could have taken here $\ell\leq (C/\beta)\log n$ for any absolute constant $C$), we have
\begin{align*}
\mu_n^\mp(M_n \geq  m_n^\star + \ell) & \leq \sum_{x\in \cL_{0,n} \setminus \cL_{0,n}^-} \mu_n^\mp(E_h^x) + \sum_{x\in \cL_{0,n}^-} \mu_n^\mp(E_h^x)
 \leq |\cL_{0,n}\setminus \cL_{0,n}^-|  e^{ - (4\beta-C)h} + |\cL_{0,n}^-| (1-o(1)) e^{-\widetilde\alpha_h}\,,
\end{align*}
using Proposition~\ref{prop:Ex-bounds} for the first sum and~\eqref{eq:Ehx-tilde-alpha-h} for the second one. 
By Corollary~\ref{cor:alpha-super-additive}, we have that \[ \widetilde\alpha_h - (4\beta-C) h \leq 
(\alpha + \epsilon_\beta - 4\beta + C) h 
\leq (C+2\epsilon_\beta) h \leq \epsilon'_\beta \log n
\,,\]
where the last inequality used the assumption on $\ell$ and the facts that $ m_n^\star = (2/\alpha+o(1))\log n$ and $\alpha > 4\beta-C$.
When combined with the fact that $|\cL_{0,n}\setminus \cL_{0,n}^-|/|\cL_{0,n}| = O(\frac{\log^2 n}n)$, this implies that
\[\frac{|\cL_{0,n}\setminus\cL_{0,n}^-|e^{-(4\beta-C)h}}
{|\cL_{0,n}^-| e^{-\widetilde\alpha_h}} \leq n^{-1+\epsilon'_\beta + o(1)}\,,
\]
which, in light of the first inequality in the proof, shows that for $\beta$ large enough (so as to have $\epsilon'_\beta<1$),
\begin{align*} \mu_n^\mp(M_n\geq m_n^\star+\ell) &\leq (1+o(1))|\cL_{0,n}^-|e^{-\widetilde\alpha_h} \leq (1+\epsilon_\beta) |\cL_{0,n}^-| \exp(-\widetilde\alpha_{ m^\star_n} -\widetilde\alpha_\ell) \\ 
&\leq (1+\epsilon_\beta) e^{2\beta+\epsilon_\beta-\widetilde\alpha_\ell} \leq (1+\epsilon'_\beta)e^{2\beta-\widetilde\alpha_\ell}\,,
\end{align*}
using Proposition~\ref{prop:submult} in the first line and~\eqref{eq:widetilde-m_n} in the second line. This establishes the right tail.

\begin{remark}
One can extend the right tail bound to hold for all $\ell$ (as opposed to $\ell\leq(1/\beta)\log n$)---albeit with a sub-optimal rate: there exists some $C>0$ such that
\begin{align}
    \mu_n^\mp(M_n \geq  m_n^\star +\ell) \leq C e^{-(2\beta-C)\ell} \qquad \mbox{for all $\ell>0$.}\label{eq:exp-tail-all-ell}
\end{align}
Indeed, consider $\ell>(1/\beta)\log n$ (having already established the desired right tail for smaller values of $\ell$). The bound $\mu_n^\mp(E_h^x) \leq C \exp(-(4\beta-C)h)$ in Proposition~\ref{prop:dobrushin-exp-tail} holds uniformly over all $x\in\cL_{0,n}$, and so
\[ \mu_n^\mp(E_{ m_n^\star+\ell}^x) \leq \mu_n^\mp(E_\ell^x) \leq C e^{-(4\beta-C)\ell} \leq C n^{-2} e^{-(2\beta-C)\ell}\,,
\]
whence
\[ \mu_n^\mp(M_n\geq m_n^\star + \ell) \leq  |\cL_{0,n}| C n^{-2} e^{-(2\beta-C)\ell} \leq 4 C e^{-(2\beta-C)\ell}\,,
\]
as claimed.
\end{remark}

Let us now turn to the lower tail for $M_n$. Let 
\[ h= m_n^\star-\ell\qquad\mbox{ for }\qquad 1\leq\ell \leq \sqrt{\log n}\,.\]
Since $Z_h > 0$ implies $E_h^x$ for some $x\in\cL_{0,n}^-$, and in particular that $M_n \geq h$, it will suffice to establish an appropriate upper bound on $\mu_n^\mp(Z_h =0)$, which we will
infer from a second moment calculation. Write 
\begin{align*}
\E[Z_h^2]  = \sum_{x,y\in \cL_{0,n}^-} \mu^\mp_n(G_h^x,\, G_h^y) 
 \leq  &\sum_{x\in \cL_{0,n}^-} \mu^\mp_n(G_h^x) \\ 
 + &\sum_{x\in \cL_{0,n}^-} \sum_{y \in \cL_{0,n}^-\cap B(x,\log^{2} n)\,:\,y\neq x } \mu^\mp_n(G_h^x) \mu^\mp_n(G_h^y\mid G_h^x) \\ 
 + &\sum_{x\in \cL_{0,n}^-} \sum_{y\in \cL_{0,n}^-\setminus  B(x, \log^{2} n)} \big(\mu^\mp_n(G_h^x)\mu^\mp_n(G_h^y) + |\mu^\mp_n(G_h^x)\mu^\mp_n(G_h^y)- \mu^\mp_n(G_h^x,\,G_h^y)|\big)\,.
\end{align*} 
Denoting these three summations by $\Xi_1,\Xi_2,\Xi_3$ (in order), we first observe that $\Xi_1$ is exactly $\E[Z_h]$. For the second summation, we apply Claim~\ref{claim:effect-of-x-pillar}, yielding 
\begin{align*}
\Xi_2 \leq 4 n^2 \log^{4}n (1+\epsilon_\beta) e^{ 6\beta} \sup_{x,y\in\cL_{0,n}^-}\mu^\mp_n(G_h^x)\mu^\mp_n(G_h^y) \leq
n^{2+o(1)}  \sup_{x\in\cL_{0,n}^-} \mu^\mp_n(E_h^x)^2\,.
\end{align*}
Using~\eqref{eq:Ehx-tilde-alpha-h} (here $h\leq  m_n^\star=O(\log n)$) we have $\mu_n^\mp(E_h^x) = (1+o(1))e^{-\widetilde\alpha_h}$, and $|\widetilde\alpha_h -  \widetilde\alpha_{m_n^\star}| = O(\ell) = O(\sqrt{\log n})$ by Corollary~\ref{cor:alpha-super-additive}, whence $e^{-\widetilde \alpha_h} = n^{-2+o(1)}$ by~\eqref{eq:widetilde-m_n}. Combined with the last equation, 
\[ \Xi_2 \leq n^{-2+o(1)} = o(1)\,.\]
Finally, for the last summation,
\begin{align*}
\sum_{x\in\cL_{0,n}^-}\sum_{y\in\cL_{0,n}^-\setminus B(x,\log^{2}n)} \mu^\mp_n(G_h^x)\mu^\mp_n(G_h^y) \leq \E[Z_h]^2 \,, 
\end{align*}
whereas, by Claim~\ref{claim:covariance-bound},
\begin{align*}
\sum_{x\in\cL_{0,n}^-}\sum_{y\in\cL_{0,n}^-\setminus B(x,\log^{2}n)}  |\mu^\mp_n(G_h^x)\mu^\mp_n(G_h^y)- \mu^\mp_n(G_h^x,\,G_h^y)| \leq C|\cL_{0,n}^-|^2 e^{-\log^2 n/C} = o(1)\,,
\end{align*}
and we deduce that $\Xi_3 \leq \E[Z_h]^2 + o(1)$. 
Putting all of these together, we obtain by the Paley--Zygmund inequality that
\begin{align*}
\mu^\mp_n(Z_h>0)  \geq \frac{\E[Z_h]^2}{\E [Z_h^2]} 
 \geq \frac{\E[Z_h]^2}{\E[Z_h]^2  + \E[Z_h] +o(1)}\,.
\end{align*}
As $\E[Z_{h}] \geq (1-\epsilon_\beta)e^{\widetilde\alpha_{\ell}-2\beta}$ by Claim~\ref{claim:expectation-lb}, we see that
\begin{align*}
    \mu_n^\mp(Z_h=0) \leq \frac{1+o(1)}{\E[Z_h]+1+o(1)}\leq (1+\epsilon_\beta') e^{2\beta-\widetilde\alpha_{\ell}}\,,
\end{align*}
and as $\{Z_h=0\}$ is implied by $\{M_n < h\}$, this concludes the proof.
\end{proof}

\subsection{The expectation and median of the maximum}
The following is a straightforward consequence of the results we have established in this section:
\begin{corollary}\label{cor:mean-median}
There exist $\beta_0>0$ and a sequence $\epsilon_\beta>0$ vanishing as $\beta\to\infty$ such that for all $\beta>\beta_0$, if $m_n^\star$ is defined as in~\eqref{eq:mn-def} and $m_n$ is a median of $M_n$ then $m_n \in [m_n^\star-1,m_n^\star]$ and $m_n^\star-1-\epsilon_\beta \leq \E[M_n]\leq m_n^\star+\epsilon_\beta$.
\end{corollary}
\begin{proof}
For the bounds on the median $m_n$, by Proposition~\ref{prop:exp-tails} we have that \[\mu_n^\mp(M_n\geq m_n^\star + 1) \leq (1+\epsilon)e^{2\beta-\widetilde\alpha_1} \leq Ce^{-2\beta} \leq \epsilon'_\beta\,,\]
using that $\exp(-\widetilde\alpha_1) = \mu_{\Z^3}^\mp(E_1^o) \leq C e^{-4\beta}$. This implies that, once $\beta$ is large enough such that $\epsilon'_\beta < \frac12$, the median satisfies $m_n \leq m_n^\star$. Further, by that same proposition, $\mu_n^\mp(M_n\leq m_n^\star - 2) \leq (1+\epsilon)e^{2\beta-\widetilde\alpha_1} \leq  \epsilon'_\beta$, whence $m_n \geq m_n^\star - 1$.

For the bound on the expectation, note that by Proposition~\ref{prop:exp-tails}, as argued above for the median, we have
\begin{align*}
    \sum_{\ell=1}^{\sqrt{\log n}} \mu_n^\mp(M_n-m_n^\star \geq \ell) < \epsilon_\beta \qquad\mbox{and}\qquad
    \sum_{\ell=1}^{\sqrt{\log n}} \mu_n^\mp(m_n^\star-1 - M_n \geq \ell) < \epsilon_\beta\,.
\end{align*} 
Therefore, (denoting by $a_+$ the positive part of $a$)
\begin{align}
\E[ (M_n - m_n^\star)_+] \leq \sum_{\ell\geq 1} \mu_n^\mp(M_n-m_n^\star \geq \ell) &\leq 
\sum_{\ell=1}^{\sqrt{\log n}} \mu_n^\mp(M_n-m_n^\star \geq \ell)
+ \sum_{\ell\geq \sqrt{\log n}} \mu_n^\mp(M_n-m_n^\star \geq \ell) \nonumber\\
&\leq \epsilon_\beta + Ce^{-(2\beta-C)\sqrt{\log n}} \leq \epsilon'_\beta
 \label{eq:Mn-mn-plus}\,,
\end{align}
where we used the uniform bound~\eqref{eq:exp-tail-all-ell} on the right tail to obtain the second line. 
At the same time,
\begin{align} \E[ (m_n^\star - 1 - M_n)_+] &\leq  m_n^\star \mu_n^\mp(M_n \leq m_n^\star - \sqrt{\log n})  +  \sum_{\ell = 1}^{\sqrt{\log n}} \mu_n^\mp(m_n^\star-1-M_n \geq \ell) \nonumber\\
& \leq \epsilon_\beta + O(\log n) e^{-c\sqrt{\log n}} \leq \epsilon'_\beta\,.\label{eq:Mn-mn-minus}
\end{align}
Letting
\[ p_n = \mu_n^\mp(M_n < m_n^\star)\,,\]
we may express 
\begin{align*}
\E[ M_n \one\{M_n \geq m_n^\star\} ] &=  m_n^\star (1-p_n) + \E[ (M_n - m_n^\star)_+ ] \,,\\
\E[ M_n \one\{M_n \leq m_n^\star-1\} ] &= (m_n^\star-1) p_n -\E[ (m_n^\star-1-M_n)_+ ] \,,
\end{align*}
and deduce from~\eqref{eq:Mn-mn-plus} and~\eqref{eq:Mn-mn-minus} that
\begin{align*}  m_n^\star (1-p_n) &\leq \E[M_n\one\{M_n\geq m_n^\star\}] \leq   m_n^\star(1-p_n) + \epsilon_\beta\,,\end{align*}
whereas
\begin{align*}
 (m_n^\star-1)p_n - \epsilon_\beta &\leq \E[ M_n \one\{M_n \leq m_n^\star-1\} ] \leq  (m_n^\star-1)p_n\,.\end{align*}
Combining these, $|\E[M_n] - \xi_n |\leq \epsilon_\beta$, where $\xi_n = m_n^\star - p_n \in [m_n^\star-1,m_n^\star]$, as required.
\end{proof}

\section{Gumbel tail estimates for the maximum}\label{sec:gumbel}

\subsection{Coupling of different scales}
The following proposition compares $M_n$, the maximum height of the interface $\cI$ under $\mu_n^\mp$, to the maxima of i.i.d.\ copies on boxes of a smaller scale. This will later be used to deduce Gumbel tail bounds for the centered maximum. 
\begin{proposition}\label{prop:multiscale-linear}
There exists $\beta_0>0$ such that the following holds for all $\beta>\beta_0$. Fix $\gamma>0$, let $L_n$ be a sequence with
$n (\log n)^{-\gamma} <  L_n  < n$,
and set $\kappa_n = (\lfloor n/ L_n \rfloor)^2$ and $R_n = n \bmod L_n$. 
Then
\begin{align*}
    \left\|\mu^\mp_n\left(M_n\in\cdot\right) - \P\left(\max\{Y_1,\ldots,Y_{\kappa_n}\}\in\cdot\right)\right\|_\tv \leq  O\Big((R_n/n)^{1/3} + (\log n)^{-10}\Big)\,,
\end{align*}
where $Y_1,\ldots,Y_{\kappa_n}$ are i.i.d.\ with law $\P(Y_i\in\cdot) = \mu_{L_n}^\mp(M_{L_n}\in\cdot)$. In particular,
\begin{align*}
     \sup_x \left|\mu^{\mp}_n(M_n \leq x) - \mu^{\mp}_{L_n}(M_{L_n} \leq x)^{\kappa_n}\right| \leq O\Big((R_n/n)^{1/3} + (\log n)^{-10}\Big)\,.
\end{align*}
\end{proposition}
We first need the following simple claim, ruling out the improbable scenario where the maximum is attained above a fixed microscopic subset of faces of $\cL_{0,n}$.
\begin{claim}\label{clm:corridors}
For every $\delta>0$ there exists $\beta_0>0$ such that the following holds for all $\beta>\beta_0$. 
Let $S_0\subset \cL_{0,n}$ be a deterministic set of faces of size $|S_0| \leq n^{2-\delta}$. Let $M_n$ be the maximum height of $\cI$ under $\mu_n^\mp$, and let $M_n^- = \max\{\hgt(\cP_x)\,:\; x\in\cL_{0,n}\setminus S_0\}$. Then for every large enough $n$, we have $\mu_n^\mp(M_n \neq M_n^-) \leq (\log n)^{-3\beta}$.
\end{claim}
\begin{proof}
Let $ h=m_n^\star-\log\log n$. We may bound the sought probability by
\begin{align*}\mu_{n}^{\mp}(M_n \neq M_{n}^-) &\leq \mu_n^\mp\Big(\max_{x\in S_0}\hgt(\cP_x) \geq h\Big) + \mu_n^{\mp}(M_n < h) \\
&\leq \left| S_0\right|
\exp(-(4\beta-C)h) + O(\exp(-\widetilde\alpha_{\log \log n}))\,,
\end{align*}
with the last inequality relying on Proposition~\ref{prop:Ex-bounds} to bound the first probability and 
Proposition~\ref{prop:exp-tails} to bound the second one. The last term is at most $ O((\log n)^{-4\beta-C})$ by the definition~\eqref{eq:tilde-alpha-def} of $\widetilde\alpha_h$ and the inequality succeeding it. That same inequality implies that \[ (4\beta-C)h \geq \widetilde\alpha_h - \epsilon_\beta h \geq \widetilde\alpha_{m_n^\star} - (\epsilon_\beta-o(1))h \geq (2-\epsilon'_\beta)\log n \,,\]
where the first inequality used that  $\widetilde\alpha_{m_n^\star}-\widetilde\alpha_h=O( \log\log n)=o(h)$ by Corollary~\ref{cor:alpha-super-additive}. 
In light of this,
\begin{align*}
    \left|S_0\right| e^{-(4\beta-C)h}\leq 
    n^{2-\delta}\,
     n^{-2+\epsilon'_\beta} \leq n^{-\delta+\epsilon'_\beta+o(1)} < n^{-\delta/2+o(1)}
\end{align*}
for $\beta$ large enough so that $\epsilon'_\beta<\delta/2$. Hence, combined with the above inequality on $\mu_n^\mp(M_n<h)$, we get that
$\mu_{n}^{\mp}(M_n \neq M_{n}^-)$ is $ O((\log n)^{-4\beta+C}) < (\log n)^{-3\beta}$ for large enough $n$ provided $\beta>C$.
\end{proof}
\begin{proof}[\textbf{\emph{Proof of Proposition~\ref{prop:multiscale-linear}}}]
First consider the case $R_n=0$.
Partition $\cL_{0,n}$ into disjoint boxes $\cB_1,\ldots,\cB_{\kappa_n}$, each of side length $2L_n$, and further let 
\begin{align*}
 \cB_i^- = \{ x \in \cB_i \,:\; d(x,\partial \cB_i) \geq \log^2 n\}\qquad(i=1,\ldots,\kappa_n)\,.
 \end{align*}
We will show that $M_n$ is equal to $\max_{i} \bar M_{L_n}^{(i)}$ with high probability, where
\begin{align*}
    \bar M_{L_n}^{(i)} = \max\{ \hgt(\cP_x) \,:\; x \in \cB_i^-\}\,,
\end{align*}
which in turn can be coupled (with negligible error) to $\kappa_n$ i.i.d.\ copies of $\bar M_{L_n}^{(1)}$ under $\mu_{\cB_1}^\mp$. Each of those i.i.d.\ copies will then be coupled to $M_{L_n}$.

First, since $\kappa_n=O((\log n)^{2\gamma})$, we have $|\cL_{0,n}\setminus (\bigcup_{i=1}^{\kappa_n} \cB_i^-)|\leq \kappa_n \cdot 4 L_n \log^2 n = O\big(n (\log n)^{2\gamma+2}\big)$, and thus infer from Claim~\ref{clm:corridors} that (with some room),
\begin{equation}\label{eq:couple-Mn-bar-Mn} \mu_n^\mp\Big(M_n \neq \max_{1\leq i\leq\kappa_n} \bar M_{L_n}^{(i)}\Big) < (\log n)^{-10}\,.\end{equation}
Second, as the boxes $\cB_i^-$ have pairwise distances at least $\log^2 n$, iterating Corollary~\ref{cor:pillar-correlations} per box shows that 
\[ \bigg\| \mu_n^\mp\left( (\cP_x)_{x\in \bigcup_i \cB_i^-}\in\cdot\right) - \prod_i \mu_n^\mp\left( (\cP_x)_{x\in\cB_i^-}\in\cdot\right) \bigg\|_\tv = O(\kappa_n\,  e^{-c\log^2 n}) < n^{-10}\,.\]
Moreover, 
Corollary~\ref{cor:pillar-dependence-on-volume} gives
\[ \sum_{i=1}^{\kappa_n}\left\| \mu_n^\mp\left( (\cP_x)_{x\in\cB_i^-} \in\cdot \right) - \mu_{\cB_1}^\mp\left( (\cP_x)_{x\in\cB_1^-}\in\cdot\right) \right\|_\tv = O(\kappa_n\, e^{-c\log^2 n}) < n^{-10}\,;
\]
thus, 
 with probability $1-O(n^{-10})$ we may couple $(\bar M_{L_n}^{(1)},\ldots,\bar M_{L_n}^{(\kappa_n)})$ under $\mu_n^\mp$ to $(Y^-_1,\ldots,Y^-_{\kappa_n})$ where $Y^-_i$ are i.i.d.\ distributed according to $\mu_{\cB_i}^\mp(\bar M_{L_n}^{(i)}\in\cdot)$. Letting $Y_i$ be i.i.d.\ copies of  $\mu_{\cB_i}^\mp(\max_{x\in \cB_i} \hgt(\cP_x)\in \cdot)$, we again apply Claim~\ref{clm:corridors}, this time to $\mu_{\cB_i}^\mp$ with $S_0 = \cB_i \setminus \cB_i^-$, so that $|S_0|= O(|L_n|\log^2 n) = L_n^{1+o(1)} $, and
 \[\sum_{i=1}^{\kappa_n}\mu_{\cB_i}^\mp(Y_i \neq Y_i^-) < \kappa_n (\log n)^{-3\beta} < (\log n)^{-10}\]
 for $\beta$ large. 
 Altogether, the total variation distance between the law of $M_n$ and $\max\{Y_1,\ldots,Y_{\kappa_n}\}$---equal in distribution to $\kappa_n$ i.i.d.\ copies of $M_{L_n}$ under $\mu_{L_n}^{\mp}$---is $O((\log n)^{-10})$, as required.

It remains to handle the case $R_n \neq 0$. Here, we will partition $\cL_{0,n}$ into boxes $\cB_1,\ldots,\cB_{\kappa_n},\hat\cB_{\kappa_n+1},\ldots,\hat\cB_{\hat\kappa_n}$, where the boxes $\cB_1,\ldots,\cB_{\kappa_n}$ have side-length $2L_n$, as before, and the remaining boxes $\hat\cB_j$ ($j>\kappa_n$) have the shorter side-length $2R_n$ (so that $\hat\kappa_n = (\lceil n/L_n\rceil)^2$). We use the same definition of $\cB_i^-$ also for $\hat\cB_j^-$, that is,
\[ \hat\cB_j^- = \{ x\in \hat\cB_j \,:\; d(x,\partial\hat\cB_j) \geq \log^2 n\}\qquad (j=\kappa_n+1,\ldots,\hat\kappa_n)\,,\]
noting that it may be the case that $\hat\cB_j^-=\emptyset$ (whenever $R_n\leq \log^2 n$). However, we would still want to couple $M_n$ to $\max_{i\leq\kappa_n}\bar M_{L_n}^{(i)}$, as we did in~\eqref{eq:couple-Mn-bar-Mn}, ignoring the exceptional boxes $\hat\cB_j$. To achieve this, apply Claim~\ref{clm:corridors} with $S_0 = \cL_{0,n}\setminus (\bigcup_{i=1}^{\kappa_n}\cB_i^- \cup \bigcup_{j=\kappa_n+1}^{\hat\kappa_n}\hat\cB_j^-)$, which, as before, has $|S_0| = O(\hat\kappa_n L_n \log^2 n) = O(n(\log n)^{2\gamma+2})$, and hence 
\begin{equation}\label{eq:coupling-Mn-S0-with-remainder} 
\mu_n^\mp\Big(M_n \neq \max_{x\in \cL_{0,n}\setminus S_0}\hgt(\cP_x) \Big) < (\log n)^{-10}\,.
\end{equation}
We may treat $S_1 = \bigcup_{j>\kappa_n}\hat\cB_j^-$ as follows:  recall from Corollary~\ref{cor:alpha-super-additive} that we have $\alpha_{m_n^\star-\ell} \geq \alpha_{m_n^\star}-(4\beta+e^{-4\beta})\ell$ and $\alpha_\ell\geq (4\beta-C)\ell$, and thus for $\ell>0$ to be specified below,
\begin{align*} \mu_n^\mp\Big(M_n = \max_{x\in S_1} \hgt(\cP_x)\Big) &\leq 
 |S_1|\max_{x\in S_1} \mu_n^\mp(\hgt(\cP_x) \geq m_n^\star-\ell) + \mu_n^\mp(M_n < m_n^\star-\ell) \\
&\leq O\bigg(|S_1| e^{-\alpha_{m_n^\star-\ell}} + e^{-\alpha_{\ell}}\bigg) \leq O\bigg(\frac{|S_1|}{n^2} e^{(4\beta+e^{-4\beta})\ell} + e^{-(4\beta-C) \ell}\bigg) \,,
\end{align*}
where we used that $\mu_n^\mp(\hgt(\cP_x)\geq h)=(1-o(1))e^{-\alpha_{h}}$ for $x\in S_1$ and $h=m_n^\star-\ell$ (as $h \ll \log^2 n \leq  d(x,\partial\hat\cB_j)$ for such~$x$), the definition of $\alpha_{m_n^\star}$, and Proposition~\ref{prop:exp-tails}. Choose 
\[ \ell = \ell_1 \;\wedge\;\ell_2\qquad\mbox{where}\qquad\ell_1 = \frac{3}{\beta}\log\log n \qquad\mbox{and}\qquad\ell_2 =\frac1{8\beta}\log\Big(\frac{n}{R_n}\Big) \,. \]
We see that $\ell = \ell_1$ implies $R_n/n \leq (\log n)^{-24}$, in which case, using $|S_1|=O(n R_n)$,
\begin{align*}
    e^{-(4\beta-C)\ell} + \frac{|S_1|}{n^2} e^{(4\beta+e^{-4\beta})\ell} \leq (\log n)^{-12 + \epsilon_\beta} + O(R_n/n) (\log n)^{12+\epsilon_\beta} = O\left( (\log n)^{-12+\epsilon_\beta}\right)\,.
\end{align*}
On the other hand, when $\ell=\ell_2$ we have
\begin{align*}
    e^{-(4\beta-C)\ell} + \frac{|S_1|}{n^2} e^{(4\beta+e^{-4\beta})\ell} \leq (R_n/n)^{\frac12 - \epsilon_\beta}
    + O\Big( (R_n/n)^{1-(\frac12+\epsilon_\beta)}\Big) = O\Big( (R_n/n)^{\frac12-\epsilon_\beta}\Big)\,.
\end{align*}
Altogether, for large enough $\beta$ we find that 
\begin{align*} \mu_n^\mp\Big(M_n = \max_{x\in S_1} \hgt(\cP_x)\Big) \leq O\Big((R_n/n)^{1/3} + (\log n)^{-10}\Big)\,.
\end{align*}
Combining this with~\eqref{eq:coupling-Mn-S0-with-remainder}, while noticing that $\cL_{0,n}\setminus (S_0\cup S_1)$ is nothing but $\bigcup_{i\leq\kappa_n}\cB_i^-$, we obtain that
\[ \mu_n^\mp\Big(M_n \neq \max_{i\leq\kappa_n} \bar M_{L_n}^{(i)}\Big) \leq O\Big((R_n/n)^{1/3} + (\log n)^{-10}\Big)\,,\]
at which point the original analysis of the law of $\max_{i\leq \kappa_n}\bar M_{L_n}^{(i)}$, showing that it is coupled to the maximum of $\kappa_n$ i.i.d.\ copies of $M_{L_n}$ under $\mu_{L_n}^\mp$, completes the proof. 
\end{proof}

\subsection{From multi-scale coupling to Gumbel tails}\label{sec:coupling-to-gumbel}
We will first prove the sought bounds in the special case when the side length $2n$ is a power of $2$. This will be extended to the general case at the end of~\S\ref{sec:coupling-to-gumbel}.

\subsubsection{Left tail} The following lemma establishes the doubly exponential left tail of the centered maximum. 
\begin{lemma}
\label{lem:left-tail-power-of-2}
There exists $\beta_0>0$ such that for every $\beta>\beta_0$ the following holds. For every fixed $\ell\geq 1$ and every large enough $n$ that is a power of 2,
\[ \exp\big(-e^{(4\beta + e^{-4\beta})\ell+2}\big) \leq \mu_n^\mp(M_n \leq m_n^\star - \ell) \leq \exp\big(-e^{\alpha_\ell - 12\beta}\big)\,. \]
\end{lemma}
\begin{proof}
The proof of both inequalities will follow from coupling $\mu_n^\mp(M_n\in\cdot)$ to the maxima of smaller scales.
We begin with the lower bound.
Consider $n_1 = n 2^{-j}$ and $n_2 = n 2^{-(j+1}) $. Since $m_{n_i}^\star$ is the minimal $h$ such that $\alpha_h$ exceeds the threshold $2\log(2n_i)-2\beta$, the difference of these thresholds between $m_{n_1}^\star$ and $m_{n_2}^\star$ is precisely $2\log2$, whereas $\alpha_{h+1}-\alpha_h \geq \alpha_1-\epsilon_\beta \geq 4\beta-C-\epsilon_\beta$ holds for every $h\geq 1$ by Corollary~\ref{cor:alpha-super-additive}. Hence, 
\[  m_{n_1}^\star -1 \leq  m_{n_2}^\star \leq  m_{n_1}^\star\,,\]
and we may consider $L_{n,k}=n 2^{-k}$ for the minimal $k\geq 0$ that would satisfy
\[ m_{L_{n,k}}^\star = m_n^\star - \ell\,.\]
(The fact that $\ell\geq 1$ implies that $k>0$.) 
We claim that this $k$ satisfies
 \[ k\leq k_1 := \left\lceil \frac{(4\beta+e^{-4\beta})\ell}{2\log 2}\right\rceil \,.\]
To see this, recall from Corollary~\ref{cor:alpha-super-additive} and the inequality below~\eqref{eq:tilde-alpha-def} that
\[ 2\log(2n)-2\beta <\alpha_{m_n^\star} \leq \alpha_{m_n^\star-\ell}+(4\beta+e^{-4\beta})\ell\,,\]
thus (using that $2\log(2n) = 2\log(2L_{n,k}) + k\log 4$ for every $k\geq 0$)
\[  \alpha_{m_n^\star-\ell} > 2\log(2L_{n,k_1})-2\beta + k_1\log 4 - (4\beta+e^{-4\beta})\ell \geq 2\log(2L_{n,k_1})-2\beta\,,\]
so $m_{L_{n,k_1}}^\star \leq m_n^\star-\ell$ and $k\leq k_1$ by definition.
Since $k_1=O(1)$, we have $n/L_{n,k}=O(1)$ and Proposition~\ref{prop:multiscale-linear} implies that
\begin{align*}
     \mu_n^\mp(M_n \leq m_n^\star -\ell) &= \mu_{L_{n,k}}^\mp(M_{L_{n,k}} \leq m_{L_{n,k}}^\star)^{4^{k}} + o(1) \geq (1-\epsilon_\beta)^{4^{k_1}}+o(1) \\
    &\geq \exp(-4^{k_1}) \geq \exp(-4 e^{(4\beta+e^{-4\beta})\ell})\,,
\end{align*}
where we used that $1-\epsilon_\beta> e^{-1}$ for every large enough $\beta$ in the transition between the lines, absorbing the $o(1)$-term for $n$ large enough in the process. This implies the desired lower bound.

For the upper bound, let $L_{n,k}=n2^{-k}$ for the minimal $k\geq 0$ that would satisfy
\[ m_{L_{n,k}}^\star = m_n^\star-\ell+2\,.\]
Further assume for now that $\ell\geq 3$ (hence $k> 0$); our resulting upper bound will hold trivially for $\ell=1,2$.
We immediately note that $k \leq k_1$, since we saw above that $m_{L_{n,k_1}}^\star\leq m_n^\star-\ell$. We will need a lower bound on $k$ to yield the required tail estimate.
Once again appealing to Corollary~\ref{cor:alpha-super-additive}, we have 
\[ \alpha_{m_n^\star-\ell+2} \leq \alpha_{m_n^\star}-\alpha_{\ell-2} + \epsilon_\beta \leq 2\log(2n) + 2\beta - \alpha_{\ell-2} + \epsilon'_\beta
\]
for some other sequence $\epsilon'_\beta>0$ vanishing as $\beta\to\infty$, 
where we used~\eqref{eq:widetilde-m_n} and the relation between $\alpha_h,\widetilde{\alpha}_h$ in~\eqref{eq:tilde-alpha-def}. 
It now follows that
\[ k \geq k_2 := \left\lfloor \frac{\alpha_{\ell-2} - 4\beta-1}{2\log 2}\right\rfloor\,,\]
since,  for $\beta$ large enough so that $\epsilon'_\beta<1$, we have
\begin{align*} \alpha_{m_n^\star-\ell+2} &\leq 2\log(2L_{n,k_2}) - 2\beta + (4\beta+\epsilon'_\beta - \alpha_{\ell-2} + k_2\log 4) 
< 2\log(2L_{n,k_2}) - 2\beta\,.
\end{align*}
Applying Proposition~\ref{prop:multiscale-linear} (recalling that $k \leq k_1$ and so $n/L_{n,k}=O(1)$ as before), we deduce that
\[ \mu_n^\mp(M_n \leq m_n^\star-\ell) = \mu_{L_{n,k}}^\mp(M_{L_{n,k}} \leq m_{L_{n,k}}^\star-2)^{4^{k}} +o(1) \leq \left((1+\epsilon_\beta)e^{2\beta-\widetilde\alpha_1}\right)^{4^{k_2}}\,,\]
using Proposition~\ref{prop:exp-tails} for the last inequality. Using that $\widetilde\alpha_1 \geq 4\beta - C$, and absorbing $\log(1+\epsilon_\beta)$ into this constant, we see that
\begin{equation}\label{eq:right-tail-upper-wrt-alpha} \mu_n^\mp(M_n\leq m_n^\star-\ell) \leq \exp(-(2\beta - C)4^{k_2}) \leq \exp\left(-(2\beta - C)e^{\alpha_{\ell-2}-4\beta-1}\right)
\leq \exp\left(-e^{\alpha_{\ell}-12\beta}\right)\end{equation}
 where in the last inequality we used that $\alpha_{\ell-2}\geq \alpha_\ell-8\beta-\epsilon_\beta$ (again by Corollary~\ref{cor:alpha-super-additive}) and thereafter added the term $1+\epsilon_\beta$ to the exponent in exchange for the factor $2\beta-C$, which is valid for large $\beta$. (Note that, as promised above, the resulting bound holds also for $\ell=1,2$, becoming trivial since $\alpha_1\leq \alpha_2 \leq 8\beta+\epsilon_\beta$.)
\end{proof}

\subsubsection{Right tail} The exponential upper bound on the right tail of the centered maximum was established in Proposition~\ref{prop:exp-tails}, implying via the relation~\eqref{eq:tilde-alpha-def} between $\alpha_h,\widetilde\alpha_h$ that (with room to spare), for every $\ell\geq 1$,
\[ \mu_n^\mp(M_n\leq m_n^\star+\ell) \leq \exp( -\alpha_\ell + 3\beta)\,.\]
It remains to provide a corresponding lower bound, as given by the following lemma.
\begin{lemma}\label{lem:right-tail-power-of-2}
There exists $\beta_0>0$ such that for every $\beta>\beta_0$ the following holds. For every fixed $\ell\geq 1$ and every large enough $n$ that is a power of 2,
\[  \mu_n^\mp(M_n \leq m_n^\star + \ell) \geq \exp\big(-(4\beta + e^{-4\beta})\ell-4(\beta+1)\big) \,. \]
\end{lemma}
\begin{proof}
The proof will follow from coupling i.i.d.\ copies of $\mu_n^\mp(M_n\in\cdot)$ to the maximum of a larger scale.
As in the proof of Lemma~\ref{lem:left-tail-power-of-2}---now viewing increasing rather than decreasing side lengths---we have that if $n_1=n 2^j$ and $n_2=n 2^{j+1}$ then
$ m_{n_1}^\star \leq m_{n_2}^\star \leq m_{n_1}^\star + 1$, 
and therefore we may consider $L_{n,k} = n 2^k$ for the minimal $k\geq 0$ (in fact $k>0$ necessarily) that satisfies
\[ m_{L_{n,k}}^\star = m_n^\star + \ell + 1 \,.\]
We claim that
\[ k \leq k_1 := \left\lceil \frac{4\beta + 1 + (4\beta + e^{-4\beta})\ell}{2\log 2}\right\rceil\,.\]
Too see this, recall from Corollary~\ref{cor:alpha-super-additive} and~\eqref{eq:widetilde-m_n} (combined with the usual relation between $\widetilde\alpha_h,\alpha_h$) that
\begin{align*} \alpha_{m_n^\star+\ell} \leq \alpha_{m_n^\star} + (4\beta + e^{-4\beta})\ell \leq 2\log (2n) + 2\beta + \epsilon_\beta + (4\beta+e^{-4\beta})\ell\,.
\end{align*}
Writing $2\log(2n)=2\log(2L_{n,k})-k\log 4$, we get that for any $k\geq 0$,
\begin{align*} \alpha_{m_n^\star+\ell} \leq 2\log(2L_{n,k}) - 2\beta + (4\beta + \epsilon_\beta + (4\beta+e^{-4\beta})\ell - k\log 4) \,,
\end{align*}
and substituting $k=k_1$ as chosen above now yields (for $\beta$ large enough so that $\epsilon_\beta<1$)
\[ \alpha_{m_n^\star+\ell} < 2\log(2L_{n,k_1}) - 2\beta\,,\]
and therefore $m_{L_{n,k_1}}^\star > m_n^\star+\ell$; that is, $ m_{L_{n,k_1}}^\star \geq m_n^\star+\ell+1$, implying that $k\leq k_1$ as claimed.

Since $k_1=O(1)$, so $n/L_{n,k}=O(1)$, we may invoke Proposition~\ref{prop:multiscale-linear} and find that
\[ \mu_{L_{n,k}}^\mp(M_{L_{n,k}}< m_{L_{n,k}}^\star-1) = \mu_n^\mp(M_n<m_n^\star+\ell)^{4^{k}}\,,\]
and so
\[ \mu_{L_{n,k}}^\mp(M_{L_{n,k}}\geq m_{L_{n,k}}^\star - 1) \leq 4^{k} \mu_{n}^\mp(M_n\geq m_n^\star+\ell)\,.\]
Using Proposition~\ref{prop:exp-tails} to bound the left-hand side from below by $1-\epsilon_\beta$, we obtain that
\begin{align*} \mu_n^\mp(M_n\geq m_n^\star+\ell) &\geq (1-\epsilon_\beta) 4^{-k} \geq (1-\epsilon_\beta) 4^{-k_1} \geq (\tfrac1{4}-\epsilon'_\beta) \exp\left(-4\beta - 1 - (4\beta+e^{-4\beta})\ell\right) \\
&\geq \exp\left(-4\beta - 3 - (4\beta+e^{-4\beta})\ell\right)
\end{align*}
for large enough $\beta$, as required.
\end{proof}

\begin{proof}[\textbf{\emph{Proof of Theorem~\ref{mainthm:gumbel-tails}}}]
For $n$ that is a power of 2, the bounds in Theorem~\ref{mainthm:gumbel-tails} were all established: 
the lower bounds were
obtained in Lemmas~\ref{lem:left-tail-power-of-2} 
and~\ref{lem:right-tail-power-of-2} for a choice of $\bar\alpha = 4\beta+e^{-4\beta}$; the upper bounds were obtained by Lemma~\ref{lem:left-tail-power-of-2} and
by~\eqref{eq:right-tail-upper-wrt-alpha} (which followed from Proposition~\ref{prop:exp-tails}).
It remains to extend the estimates in Lemmas~\ref{lem:left-tail-power-of-2} and~\ref{lem:right-tail-power-of-2} to general $n$, which will follow from the decorrelation inequalities of~\S\ref{sec:submult}. 

Let $N$ be a power of $2$ such that $N/4 < n <N/2$. By Corollary~\ref{cor:pillar-dependence-on-volume} we have that
\[ \left\| \mu_n^\mp\left((\cP_x)_{x\in \cL_{0,n}^-}\in\cdot\right) - 
\mu_N^\mp\left((\cP_x)_{x\in \cL_{0,n}^-}\in\cdot\right) \right\|_\tv = O(e^{-c\log ^2 n}) = o(1)\,.
\]
Recall that $m_n^\star \leq m_N^\star \leq m_n^\star+1$ 
(since the scales changed by at most a factor of 2 whereas $\alpha_{h+1}-\alpha_h \geq 4\beta-C$, as explained in the proofs of Lemmas~\ref{lem:left-tail-power-of-2} and~\ref{lem:right-tail-power-of-2}). Furthermore, if $M_n^- =  \max\{\hgt(\cP_x) \,:\;x\in\cL_{0,n}^-\}$, then Claim~\ref{clm:corridors} shows that $\mu_n^\mp(M_n \neq M_n^-)=o(1) $, and so
\[ \mu_n^\mp(M_n\leq m_n^\star-\ell) = \mu_n^\mp(M_n^- \leq m_n^\star -\ell) +o(1) \geq \mu_N^\mp(M_N \leq m_N^\star - \ell-1)+o(1)\,.
\]
A lower bound on the probability in the right-hand is given by
Lemma~\ref{lem:left-tail-power-of-2}, whereby (recalling $\bar\alpha=4\beta+e^{-4\beta}$)
\[   \mu_N^\mp(M_N \leq m_N^\star - \ell-1) \geq \exp\left(-e^{\bar\alpha(\ell+1)+3}\right) \geq \exp\left(-e^{\bar\alpha\ell +4(\beta+1)}\right)\,,
\]
a lower bound that extends to $\mu_n^\mp(M_n\leq m_n^\star-\ell)+o(1)$ via the preceding inequality.

The remaining two inequalities (the upper bound on the right tail in~\eqref{eq:right-tail-upper-wrt-alpha} was already established for all~$n$) will follow from a comparison of $\mu_n^\mp$ to $\mu_N^\mp$ where $N$ is a power of $2$ such that $2N < n < 4N$. The same coupling mentioned above shows that 
\[ \left\| \mu_N^\mp\left((\cP_x)_{x\in \cL_{0,N}^-}\in\cdot\right) - 
\mu_n^\mp\left((\cP_x)_{x\in \cL_{0,N}^-}\in\cdot\right) \right\|_\tv = O(e^{-c\log ^2 n}) = o(1)\,.
\]
As before $\mu_N^\mp(M_N \neq M_N^-) = o(1)$, and now we have $m_N^\star \leq m_n^\star \leq m_N^\star +1$, whence
\[ \mu_n^\mp(M_n\leq m_n^\star-\ell) \leq \mu_N^\mp(M_N^- \leq m_N^\star - \ell+1) +o(1) = \mu_N^\mp(M_N \leq m_N^\star - \ell+1) + o(1)\,.
\]
By the upper bound in Lemma~\ref{lem:left-tail-power-of-2},
\[ \mu_N^\mp(M_N \leq m_N^\star - \ell+1) \leq \exp\left(-e^{\alpha_{\ell-1}-12\beta}\right) \leq \exp\left(-e^{\alpha_{\ell}-(16\beta+1)}\right)\,.
\]
Similarly, we have
\[ \mu_n^\mp(M_n \geq m_n^\star + \ell) 
\geq \mu_N^\mp(M_N^- \geq m_N^\star + \ell+1) -o(1) = 
\mu_N^\mp(M_N \geq m_N^\star + \ell+1) -o(1) \,,\]
whereas by Lemma~\ref{lem:right-tail-power-of-2},
\begin{align*}
    \mu_N^\mp(M_N \geq m_N^\star + \ell+1) \geq \exp\left(-\bar\alpha(\ell+1)-4(\beta+1)\right) \geq
\exp\left(-\bar\alpha\ell-(8\beta+5)\right)\,,
\end{align*} 
 thus concluding the proof.
\end{proof}

\subsection{Non-convergence of the centered maximum}
The following simple corollary of Proposition~\ref{prop:multiscale-linear} will be used to derive Proposition~\ref{mainprop:divergence}.
\begin{corollary}\label{cor:G-G^k}
Let $M_n$ be the maximum height of $\cI$ under $\mu_n^\mp$. Fix $\omega\in \R_+$. Then
\begin{align*}
    \limsup_{n\to\infty}\sup_x \left|\mu^{\mp}_n(M_n \leq x) - \mu^{\mp}_{\lfloor n/\omega \rfloor}(M_{\lfloor n/\omega \rfloor} \leq x)^{\omega^2}\right| = 0\,.
\end{align*}
\end{corollary}
\begin{proof}
Fix any $\epsilon>0$, set $n_\omega = \lfloor n/\omega\rfloor$ for brevity, and let 
\[ \Gamma = \Gamma_{\epsilon,n} \cup \Gamma_{\epsilon,n_\omega}\qquad\mbox{where}\qquad \Gamma_{\epsilon,n} = \left\{ x\in \R \,:\; \mu^\mp_n(M_n\leq x) \geq \epsilon\right\}\,. 
\]
By writing
\begin{align*} \sup_{x_n} \left|\mu^\mp_n(M_n\leq x_n) - \mu^\mp_{n_\omega}(M_{n_\omega}\leq x_n)^{\omega^2}\right| \leq \epsilon &+ \sup_{x_n\in\Gamma} \left|\mu^\mp_n(M_n\leq x_n) - \mu^\mp_{n_\omega}(M_{n_\omega}\leq x_n)^{\omega^2}\right|
\,,\end{align*}
it will suffice to prove that the second term in the right-hand vanishes as $n\to\infty$ for any such choice of $\epsilon$.

Applying Proposition~\ref{prop:multiscale-linear} to $\mu_n^\mp(M_n\in\cdot)$ with $L_n=\lfloor n/\log n\rfloor$, whereby $\kappa_n = (1+o(1))\log^2 n$, we obtain  
\[ \sup_{x_n\in\Gamma} \left| \mu_n^\mp(M_n\leq x_n) - p_{n,x_n}^{\kappa_n} \right| = o(1)\,,\qquad\mbox{where}\qquad p_{n,x_n}=\mu_{L_n}^\mp(M_{L_n}\leq x_n)\,.
\]
Another application of Proposition~\ref{prop:multiscale-linear}, this time to $\mu_{n_\omega}^\mp(M_{n_\omega}\in\cdot)$ yet with the same $L_n=\lfloor n/\log n\rfloor$, gives
\[ \sup_{x_n\in\Gamma}\left|\mu_{n_\omega}^\mp(M_{n_\omega}\leq x_n) - p_{n,x_n}^{\kappa'_n} \right|= o(1)\,,
\]
where $p_{n,x_n}$ is as before, and $\kappa'_n = (\omega^{-2}+o(1)) \kappa_n$ (recall that $\omega\in\R_+$ is fixed). 

For $x_n\in\Gamma_{\epsilon,n}$, we have that $p_{n,x_n}^{\kappa_n} \geq \epsilon - o(1)$, whence  $p_{n,x_n}^{(1+o(1))f_n}= p_{n,x_n}^{f_n} + o(1)$ if $f_n$ has order $\log^2 n$, and the same conclusion holds if $x_n\in\Gamma_{\epsilon,n_\omega}$ (as $\kappa'_n$ is also of order $\log^2 n$). In particular,
\[ \sup_{x_n\in\Gamma}\left|\mu_n^\mp(M_n\leq x_n) - p_{n,x_n}^{\log^2 n} \right| = o(1)\qquad\mbox{and}\qquad
 \sup_{x_n\in\Gamma}\left|\mu_{n_\omega}^\mp(M_{n_\omega}\leq x_n) - p_{n,x_n}^{\omega^{-2} \log^2 n} \right|= o(1)\,,
\]
and combining these inequalities yields that,
\[\sup_{x_n\in\Gamma}\left| \mu_n^\mp (M_n\leq x_n) - \mu_{n_\omega}^\mp(M_{n_\omega}\leq x_n)^{\omega^2} \right| = o(1)\,,\]
as required.
\end{proof}
\begin{proof}[\textbf{\emph{Proof of Proposition~\ref{mainprop:divergence}}}]
Suppose that $m_n$ is a sequence such that $\{M_n-m_n\}$ weakly converges to a nondegenerate random variable $M_\infty$ with a distribution function $G$. Since $m_n \bmod 1$ must converge (as $M_n$ is integer valued), we may assume w.l.o.g.\ that $m_n\in\Z$ (whence $G$ is also integer valued). Observe that the bounds in Theorem~\ref{mainthm:gumbel-tails} (in fact already those of Proposition~\ref{prop:exp-tails}) imply that we must have 
\[\limsup_{n\to\infty}|m_n-m_n^\star|<\infty\,.\]
Fix $k\geq 2$, and consider Corollary~\ref{cor:G-G^k} with $\omega=\sqrt{k}$. By assumption, $\lim_{n\to\infty}\mu_n^\mp(M_n\leq m_n+x) = G(x)$ for all $x$, 
and the above corollary then implies that
\[ \lim_{n\to\infty} \mu_{\lfloor n/\sqrt{k}\rfloor}^\mp(M_{\lfloor n/\sqrt{k}\rfloor}\leq m_n+x) = G^{1/k}(x)\,.\]
However, $|m_n^\star - m_{\lfloor n/\sqrt{k}\rfloor}^\star | \leq C(\beta) \log k $ (see, e.g., the estimate~\eqref{eq:widetilde-m_n}), and the above bound on $|m_n-m_n^\star|$ thus implies that
\[ a_{k,n} := m_n - m_{\lfloor n/\sqrt{k}\rfloor} \]
satisfies $\limsup_{n\to\infty}|a_{k,n}|<\infty$. Let $(a_{k,n_j})_{j\geq 1}$ be a converging subsequence of $(a_{k,n})_{n\geq 1}$, and denote its limit by $a_k$, whereby, recalling that $m_n\in\Z$, we must have $a_{k,n_j}=a_k$ for all sufficiently large $j$. Thus, for every large enough~$j$,
\[ \mu^\mp_{\lfloor \frac{n_j}{\sqrt k}\rfloor}\left(M_{\lfloor \frac{n_j}{\sqrt k}\rfloor} \leq m_{n_j} + x\right)
= \mu^\mp_{\lfloor \frac{n_j}{\sqrt k}\rfloor}\left(M_{\lfloor \frac{n_j}{\sqrt k}\rfloor} \leq m_{\lfloor \frac{n_j}{\sqrt k}\rfloor} + x+a_k\right) \xrightarrow[j\to\infty]{} G(x+a_k)\,,
\]
where the last equality is again by our weak convergence assumption. 
Together, this implies that for every $x\in\R$ we have $G^k(x+a_k)=G(x)$; having established this for every $k\geq 2$, we find that $G$ is max-stable, yet $G$ is discrete, contradicting the fact that the only (nondegenerate) max-stable distributions are continuous ones, belonging to one of the three classes of extreme value distributions (see, e.g.,~\cite[Thm.~1.3.1]{LLR83}).
\end{proof}

\subsection*{Acknowledgment} We are grateful to an anonymous referee for valuable comments. E.L.~was supported in part by NSF grant DMS-1812095.

\bibliographystyle{abbrv}

\bibliography{ising}

\end{document}